  \documentclass[11pt,letterpaper]{amsart}
\RequirePackage[colorlinks,citecolor=blue,urlcolor=blue,linkcolor=blue]{hyperref}
\hypersetup{
colorlinks = true,
citecolor=blue,
urlcolor=blue,
linkcolor=blue,
pdfpagemode = UseNone
}

\usepackage{accsupp}

\usepackage{graphicx,hyperref,xcolor}
\usepackage{enumerate,amsmath}
\usepackage{soul,fancybox}
 \usepackage{tikz}
 \usetikzlibrary{patterns}
 \usepackage{geometry}

   \usepackage{pgf,float}

\usepackage{subcaption}

\usepackage[normalem]{ulem}
\newtheorem{theorem}{Theorem}[section]
\newtheorem{lemma}[theorem]{Lemma}
\newtheorem{proposition}[theorem]{Proposition}
\newtheorem{corollary}[theorem]{Corollary}
\newtheorem{claim}[theorem]{Claim}
\theoremstyle{definition}
\newtheorem{definition}{Definition}[section]

\newtheorem{assumption}{Assumption}[section]
\newtheorem{remark}[theorem]{Remark}

\numberwithin{equation}{section}

\def\topp#1{^{(#1)}}
\def\vv#1{\boldsymbol{#1}}
\def\vv#1{\boldsymbol{#1}}

\newcommand{\eps}{\varepsilon}
\newcommand{\la}{\lambda}
\newcommand{\RR}{\mathbb{R}}
\newcommand{\EE}{\mathbb{E}}
\newcommand{\CC}{\mathbb{C}}
\newcommand{\ZZ}{\mathbb{Z}}
\newcommand{\NN}{\mathbb{N}}

\newcommand{\II}{\vv I}

\newcommand{\GGG}{\vv G}
\newcommand{\GG}{\mathcal{Z}}

    \def\re{\textnormal {Re}}
    \def\im{\textnormal {Im}}
     \def\d{{\textnormal d}}
    \def\i{{\textnormal i}}
     \def\Res{\textnormal{Res}}
\makeatletter
\renewcommand*\env@matrix[1][\arraystretch]{%
  \edef\arraystretch{#1}%
  \hskip -\arraycolsep
  \let\@ifnextchar\new@ifnextchar
  \array{*\c@MaxMatrixCols c}}
\makeatother

 \usepackage{framed,pgf}
\newcounter{oldeq}
\newcounter{usesofarxiv}
 \newcommand{\arxiv}[1]{
\setcounter{oldeq}{\value{equation}}
 \addtocounter{usesofarxiv}{1}
 \setcounter{equation}{0}
\def\theoldeq{\theequation}
\def\theequation{x-\arabic{usesofarxiv}.\arabic{equation}}
\def\theequation{\arabic{section}.\arabic{usesofarxiv}.\arabic{equation}}
\def\theequation{\thesection.\arabic{usesofarxiv}.\arabic{equation}}
  \colorlet{shadecolor}{gray!10}
{
\begin{shaded}
\footnotesize
#1
\normalsize
\end{shaded}
   \setcounter{equation}{\value{oldeq}}
\numberwithin{equation}{section}
}}
\newcommand{\hide}[1]{\shadowbox{\tiny     \hl{\%}$\backslash$hide~$\Rightarrow$~view}}
\renewcommand{\hide}[1]{\shadowbox{\tiny     \hl{delete}} {\color{gray} \footnotesize #1}}

\begin{document}
\title[Log-gamma polymer on a strip]{Stationary measures for log-gamma polymer on a strip and in half-space}
\author{W{\l}odek Bryc}%
\address
{
W{\l}odzimierz Bryc\\
Department of Mathematical Sciences\\
University of Cincinnati\\
2815 Commons Way\\
Cincinnati, OH, 45221-0025, USA.
}
\email{wlodek.bryc@gmail.com}%

\author{Jacek  Weso{\l}owski}
\address{Jacek Weso{\l}owski, Faculty of Mathematics and Information Science,
Warsaw University of Technology, ul. Koszykowa 75, 00-662
Warszawa, Poland; Statistics Poland, Al. Niepodleg{\l}o\'sci 208, 00-925 Warsaw}
\email{jacek.wesolowski@pw.edu.pl}

\begin{abstract}
We study stationary measures of the log-gamma polymer on a finite diagonal strip and in a half-space. We establish the phase diagram for the stationary measure on the strip. To this end we develop a representation of  {the Laplace transform of} this stationary measure by  independent $\mathrm{Beta}_{II}$ random variables. We also present an analytic approach that extends Barraquand's contour integral representation of the Laplace transform.
We prove that, as the strip width tends to infinity, the stationary measure of the log-gamma polymer on the strip converges to the stationary measure of the half-space log-gamma polymer. Finally, we derive a contour integral formula for the Laplace transform of the stationary measure of the half-space log-gamma polymer.
\end{abstract}
\maketitle

\arxiv
{This is an expanded version of the paper  with additional details.
}

\section{Introduction}

The directed polymer in a random environment was introduced in the statistical physics literature in \cite{huse1985pinning}. While directed polymer models can be defined in higher dimensions or in continuous space and time, in this work we restrict attention to models on the lattice $\mathbb{Z}^2$. Of particular interest are \emph{integrable directed polymer models}, which admit exact descriptions of their stationary and fluctuation properties and play a central role in the rigorous study of the KPZ (Kardar--Parisi--Zhang) universality class.
A prominent example is the \emph{log-gamma polymer} on the quadrant $\mathbb{Z}_{>0}^2$,  a solvable directed polymer model with inverse-gamma-distributed weights introduced by Sepp\"al\"ainen \cite{Seppalainen2012scaling}, who identified a one-parameter family of stationary measures consisting of multiplicative random walks with inverse-gamma-distributed increments.

Restricting the log-gamma polymer to smaller sublattices gives rise to nontrivial boundary effects. Barraquand and Corwin \cite{Barraquand-Corwin-AOP2024} showed that the half-space log-gamma polymer, first studied in \cite{OSZ2014}, admits a family of stationary measures depending on the bulk and boundary parameters together with an additional parameter controlling the free-energy drift at infinity.

Barraquand, Corwin, and Yang \cite{barraquand2024stationary} introduced the log-gamma polymer on a strip and proved the existence and uniqueness of its stationary measure, which depends on the bulk parameter and two boundary parameters. They also obtained an explicit representation of the stationary measure in terms of a pair of suitably reweighted independent log-gamma random walks. Subsequently, Barraquand \cite{Barraquand-2024-integral} derived an explicit contour integral formula  for the multipoint Laplace transform of the stationary measure.

In this paper, we first establish the phase diagram of the log-gamma polymer on a finite strip. We show that the stationary measure exhibits boundary-driven phases whose structure closely parallels the phase diagram of the open asymmetric simple exclusion process \cite{derrida93exact}, with the free-energy increments playing the role of occupation variables. In particular, different regimes of the boundary parameters give rise to distinct asymptotic behaviors, separated by a transition line where the boundary effects balance, analogous to the coexistence line in open ASEP \cite[Theorem~1.6]{WWY-2024}.

We then establish that, as the strip width tends to infinity, the stationary measure on the strip converges to a stationary measure of the half-space log-gamma polymer. The limiting measure is characterized by a drift parameter determined by the phase diagram of the strip model.

 Finally, we derive a contour integral   formula for the Laplace transform of the stationary measure of the  half-space log-gamma polymer. %

 Our proofs are based on a mixture of probabilistic and analytic approaches: we develop
representations of the Laplace transform in terms of independent $\mathrm{Beta}_{II}$ random variables, we use classical Barnes-type integral identities to establish contour integral representations, and we provide explicit formulas for their  analytic continuation.

  \subsection{Log-gamma polymer model} \sloppy
Throughout the paper we use the gamma function $\Gamma(z)$ for $z\in\CC\setminus\ZZ_{\le0}$, where
$\ZZ_{\le0}=\{0,-1,-2,\ldots\}$. Recall that $1/\Gamma(z)$ extends to an entire function on $\CC$ and vanishes precisely on $\ZZ_{\le0}$.
We write $\psi(x)=\Gamma'(x)/\Gamma(x)$ for the digamma function.

 For a random variable $X$ with distribution $\mu$, we write $X\simeq\mu$. The notation $X\stackrel{d}{=}Y$ means that random variables $X$ and $Y$ have the same distribution.
A positive random variable $X$ is said to have the gamma distribution with parameter $p>0$ if it has density $x^{p-1}e^{-x}/\Gamma(p)$ on $(0,\infty)$; then we write $X\simeq\mathrm{Gamma}(p)$.

\medskip
For $N=1,2,\dots$,
we consider a strip
$$\mathbb{S}_N:=\{(j,n): n\in \ZZ_{\ge 0}, j\in\{n, n+1, \dots, n+N\}\}\subset \ZZ_{\ge 0}^2$$
of width $N$  of the $\mathbb{Z}^2$ lattice  with vertices $(j,n)$. Vertices $(n,n)$ for $n\geq 0$ are called the left boundary vertices; vertices $(n+N,n)$ for $n\geq 0$  are called the right boundary vertices and the remaining vertices $(j,n)$ for $n\geq 0$ and $n<j<n+N$ are called the bulk vertices.   Edges of the lattice $\mathbb{Z}^2$ connecting two neighboring vertices of the strip are called edges of the strip.

Let $\alpha>0$ be a bulk parameter and $u,v\in\RR$ be boundary parameters such that $u+\alpha>0 $ and $v+\alpha>0$.
\normalsize  Let $\{W_{j,n}\}_{1\leq n\leq j \leq n+N}$
be the family of independent  random variables indexed by the vertices of the strip. For any bulk vertex $(j,n)$ we assume that $1/W_{j,n}\simeq \mathrm{Gamma}(2\alpha)$. For any left boundary vertex $(n,n)$ we assume that $1/W_{n,n}\simeq \mathrm{Gamma}(\alpha+u) $. For any right boundary vertex $(n+N,n)$ we assume $1/W_{n+N,n}\simeq\rm{Gamma}(v+\alpha)$.
For a given initial condition $\vv H_0:\{0,1,\dots, N\}\to \RR$ with  $\vv H_0(0)=0$ and  $(j,n)\in \mathbb{S}_N$ we
 define the associated free energy and polymer partition
function  for all points of the strip:
\begin{equation}
  \label{Log-G-Poly}
  H_{j,n}:= \log Z_{j,n}, \quad Z_{j,n}=  \sum_{i=1}^N e^{\vv H_0(i)}\sum_{\pi:(i,1)\to(j,n)} \prod_{k=0}^{\ell(\pi)}W_{\pi(k)},
\end{equation}
where the second sum is taken over all  up-right paths $\pi:\{0,\dots,K\}\to \mathbb{S}_N$, with endpoints $\pi(0)=(i,1)$, $\pi(K)=(j,n)$. The
up-right { path} means  that   $\pi(k+1)-\pi(k)\in\{(0,1),(1,0)\}$ for $k=0,1,\dots$, and $K=\ell(\pi)$ is the length of the path so that $K=n-1+j-i$.
   For illustration, see Figure~\ref{fig:sub1}.

Formula      \eqref{Log-G-Poly} is equivalent to the statement that the array $\{Z_{j,n}\}$ satisfies the  recurrence
\begin{equation}
  \label{one-step+}
  Z_{j,n}=W_{j,n}\left(Z_{j-1,n}+Z_{j,n-1}\right), \quad (j,n)\in \mathbb{S}_N, \quad  n\geq 1,
\end{equation}
together with the initial condition $Z_{j,0}=e^{\vv H_0(j)}$ for  $j=0,\dots,N$ and the boundary conditions $Z_{j,n}=0$ for $(j,n)\not\in \mathbb{S}_N$ and $n\geq 0$.

We interpret %
$H_{i+n,n}$
as a  stochastic growth model on the finite space  $i\in\{0,\dots, N\}$, with  $n$  playing the role of  the \emph{ time variable}.  We write %
$\vv H_n=(H_{i+n,n}-H_{n,n})_{i=0,\dots,N}$ for the %
vector of the free energy of the polymer at %
time $n=1,2,\dots$.
{ Note that  the first component of $\vv H_n$ is   $0$ (nevertheless, it is convenient to retain the first component in the notation).}
Then $(\vv H_n)_{n=0,1,\dots}$ forms a Markov chain.   See Figure~\ref{fig:sub2}. %

\begin{definition}[stationary measure] \label{Def11} We say that a probability measure  $\nu$ on $\RR^{N+1}$ is a {\em stationary measure}   for the log-gamma polymer on a strip if, whenever the initial condition $\vv H_0$  has distribution $\nu$ and is independent of the collection
  $\{W_{j,n}\}_{0\leq n\leq j \leq n+N}$,
the random vector
$\vv H_n$  also has distribution $\nu$ for all $n\geq 1$.
\end{definition}

 \begin{figure}[H]
    \begin{subfigure}{0.48\textwidth}
        \centering  \vfill
        \BeginAccSupp{ActualText={An example of an up-right path}}
      \begin{tikzpicture}[scale=.5]
    \draw[step=1cm,gray,very thin,dotted] (-1,-1) grid (11,5);
     \draw[gray, ultra thick] (0,0)    -- (6,0);
  \draw (0,0) node[below] {\tiny $0$};
    \draw (6,0) node[below] {\tiny $N$};

\draw    (1,1) node[above, blue]{\tiny $W_{11}$};
    \draw [fill,blue]   (1,1) circle (.1);

    \draw    (2,1) node[above]{\tiny $W_{21}$};
       \draw [fill]   (2,1) circle (.1);

        \draw [fill]   (6,1) circle (.1);
\draw    (7.4,1) node[above,red]{\tiny $W_{N+1,1}$};
        \draw [fill,red]   (7,1) circle (.1);

       \draw    (3,2) node[above]{\tiny $W_{32}$};

       \draw    (2,2) node[above, blue]{\tiny $W_{22}$};
        \draw[-] (0,-.1) to (0,.1);
          \draw[-] (1,-.1) to (1,.1);
           \draw[-] (2,-.1) to (2,.1);
            \draw[-] (3,-.1) to (3,.1);
              \draw[-] (4,-.1) to (4,.1);
               \draw[-] (5,-.1) to (5,.1);
         \draw[-] (6,-.1) to (6,.1);

         \draw[-,gray,very thin] (1,0) -- (1,1)   -- (6,1) -- (6, 0);
          \draw[-,gray,very thin] (2,0) -- (2,2)   -- (7,2) -- (7, 1) -- (6,1);
            \draw[-,gray,very thin] (3,0) -- (3,3)   -- (8,3) -- (8, 2) -- (7,2);
           \draw[-,gray,very thin] (4,0) -- (4,4)   -- (9,4) -- (9, 3) -- (8,3);
             \draw[-,gray,very thin] (5,0) -- (5,5)   -- (10,5) -- (10, 4) -- (9,4);
          \draw[-,gray,very thin] (6,1) --   (6,5.5);
          \draw[-,gray,very thin] (7,2) --   (7,5.5);
          \draw[-,gray,very thin] (8,3) --   (8,5.5);
          \draw[-,gray,very thin] (9,4) --   (9,5.5);
          \draw[-,gray,very thin] (10,5) --   (10,5.5);
          \draw[-,gray,very thin] (10,5) --   (11,5) -- (11,5.5);

   \foreach \x in {1, ..., 5} {
    \foreach \y in {1,..., 5}  {

            \filldraw[blue!50, draw=blue, thick] (\x, \x) circle (.05);
              \filldraw[red!50, draw=red, thick] (\x+6, \x) circle (.05);
            \filldraw[thick] (\x+\y, \x) circle (.05);
        }

 }
      \draw    (0,4) node[left ]{\tiny $n$};
            \draw    (8,0) node[below]{\tiny $j$};
            \draw [fill]   (8,4) circle (.1);
      \draw[-,thick] (4,0)  -- (4, 1);
         \draw[-, ultra thick] (4,1)  -- (5, 1) -- (5, 2) -- (5,3)--(6,3) -- (7,3) --  (7,4) -- (8,4);
            \draw    (7.4,4) node[above right]{\tiny $Z_{j,n}$};
        \draw    (4,0) node[below]{\tiny $Z_{i,0}$};
   \end{tikzpicture}
   \vfill
   \subcaption{%
    The black path is an example of an up-right path $\pi$ with $\ell(\pi)=7$, contributing
$
Z_{i,0}W_{i,1}\dots W_{j-1,n}W_{j,n}
$
to partition function in \eqref{Log-G-Poly}.}
        \label{fig:sub1}
    \end{subfigure}
   \hfill
    \begin{subfigure}[b]{0.48\textwidth}
        \centering \vfill
         \BeginAccSupp{ActualText={Two black piecewise linear curves represent that represent two vectors}}
     \begin{tikzpicture}[scale=.32]

      \foreach \x in {0,..., 5} {
\def\slope{2}
\draw[-,gray,very thin] (\slope*\x, \x) --   (\slope*\x+6, \x);
              \draw[-,gray,very thin] (\x, 0) --   (\slope*\x, \x);
 \draw[-,gray,very thin] (\slope*\x+6, \x) --   ( {\slope*\x+6+(5.5-\x)},  5.5);
        }

 \draw[gray, ultra thick] (0,0)    -- (6,0);
   \draw[-] (0,-.1) to (0,.1);
          \draw[-] (1,-.1) to (1,.1);
           \draw[-] (2,-.1) to (2,.1);
            \draw[-] (3,-.1) to (3,.1);
              \draw[-] (4,-.1) to (4,.1);
               \draw[-] (5,-.1) to (5,.1);
         \draw[-] (6,-.1) to (6,.1);
 \draw (0,0) node[below]{\tiny $0$};
 \draw (6,0) node[below]{\tiny $N$};
 \draw (4,0) node[below]{\tiny $i$};
\fill[pattern=vertical lines, pattern color=gray!60] (0, 0) --  (1, 1.2) -- (2,1.5) -- (3,0.5) -- (4,0.5) --  (5,1.3) -- (6, 1) -- (6,0) -- (0,0) -- cycle;
\draw[-,thick] (0, 0) --  (1, 1.2) -- (2,1.5) -- (3,0.5) -- (4,0.5) --  (5,1.3) -- (6, 1) -- (6,0);

\fill[pattern=crosshatch, pattern color=blue!10]  (0,0) -- (2, 4) -- (4, 6) -- (6 , 6.5) -- (8,8) -- (8,4)  -- (0,0) -- cycle;

\draw[-,blue,dashed] (0,0) -- (2, 4) -- (4, 6) -- (6 , 6.5) -- (8,8) -- (10,8.3) -- (12, 9.6);
\draw[-,blue, thick,dotted] (10,5) to (10, 8.3);
\draw[-,blue, thick] (2,1) to (2, 4);
\draw[-,blue, thick] (4,2) to (4, 6);
\draw[-,blue, thick] (6,3) to (6 ,6.5);
\draw[-,blue, thick] (8,4) to (8,8);
\draw[-,blue, thick] (8,8) to (8,4.5);

\draw  (7.9, 8) node[left]{\tiny \color{blue} $H_{n,n}$};
\filldraw[blue!50, draw=blue, thick] (8, 8) circle (.05);
\filldraw[black] (8, 4) circle (.05);

\fill[pattern=crosshatch, pattern color=blue!10] (10.3,8.5)-- (12,9.6) -- (12, 7.5) -- (11,8.5);

\fill[pattern=vertical lines, pattern color=gray!60] (0, 0) --  (1, 1.2) -- (2,1.5) -- (3,0.5) -- (4,0.5) --  (5,1.3) -- (6, 1) -- (6,0) -- (0,0) -- cycle;
\draw[-,thick] (0, 0) --  (1, 1.2) -- (2,1.5) -- (3,0.5) -- (4,0.5) --  (5,1.3) -- (6, 1);

\draw[-] (1,0) -- (1, 1.2);
\draw[-] (2,0) --  (2,1.5);
\draw[-] (3,0) --  (3,0.5);
\draw[-] (4,0) -- (4,0.5);
\draw[-] (5,0) --  (5,1.3);
\draw[-] (6,0) --  (6,1);

\fill[pattern=horizontal lines light gray, pattern color=gray!20] (8,4)--(8, 8) --  (9, 9.2)  -- (10,8.5)--(11,8.5) -- (12,7.5) --  (13,8.5) -- (14, 7) -- (14,4) -- (13.35,4) -- (13,4.5) -- (12.5,4) --(11.5,4)  -- (11, 4.5) -- (10,4.5)  -- (9,5.2) -- (8, 4) ;

 \draw[-,dashed,thick]   (8, 8) --  (9, 9.2) -- (10,8.5) -- (11,8.5) -- (12,7.5) --  (13,8.5) -- (14, 7);

\fill[pattern=vertical lines, pattern color=gray!60]     (8, 4) --  (9, 5.2) -- (10,4.5) -- (11,4.5) -- (12,3.5) --  (13,4.5) -- (14, 3) -- (14, 4);

\draw[-]  (11,4) to (11,4.5);
\draw[-]  (10,4) to(10,4.5);
\draw[-]  (9,4) to (9,  5.2);
\draw[-]  (13,4) to (13,4.5);
\draw[-]  (14,4) to (14,3);
\draw[-]  (12,4) to (12,3.5);

\draw[gray, thick] (8,4)    -- (14,4);
\draw[-,gray] (14,4)    -- (14,4.5);
\draw[-, thick] (14,7)    -- (14,4);
\draw[-,thick]   (8, 4) --  (9, 5.2) -- (10,4.5) -- (11,4.5) -- (12,3.5) --  (13,4.5) -- (14, 3)  ;

\draw (6,1) node[right]{\tiny $\vv H_0$};
\draw (14,3) node[right]{\tiny $\vv H_n$};

\draw (12,7.5) node[left]{\tiny $\log Z_{j,n}$};
\draw[-]  (12,4) to (12,7.5);
     \filldraw[black] (12, 7.5) circle (.05);
       \filldraw[black] (12, 4) circle (.05);
\draw  (14, 7) node[right]{\tiny \color{red} $H_{n+N,n}$};
\filldraw[red!50, draw=red, thick] (14, 7) circle (.05);
\filldraw[black] (14, 3) circle (.05);
   \end{tikzpicture}
\vfill \subcaption{The black piecewise linear curves represent the
 random vectors $\vv H_0$ and $\vv H_n$ appearing in Definition~\ref{Def11}.}
        \label{fig:sub2}
    \end{subfigure}
    \caption{The subset $\mathbb{S}_N$ of the $\ZZ^2$ lattice for $N=6$.}
     \label{Fig:0}
\end{figure}
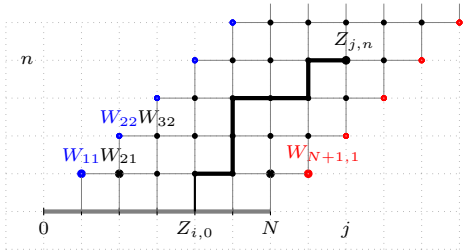
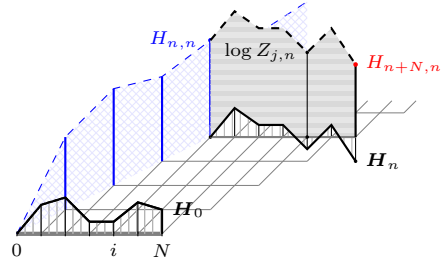

According to \cite[Theorem 1.6]{barraquand2024stationary}, the log-gamma polymer on a strip has  a unique stationary measure, which  is the %
law of the first vector component, which is \eqref{H0} below, of the pair
  of   reweighted log-gamma random walks.
  We rewrite their    representation in terms of the reweighted gamma densities as follows.
{  Let $(X_1\topp N,\dots,X_N\topp N,Y_1\topp N,\dots,Y_N\topp N)$ be a random vector in $(0,\infty)^{2N}$ with joint density}
    \begin{equation}
  \label{f-joint+}
  f_{LG}(\vv x,\vv y)=\frac{1}{\mathcal{Z}_N} \frac{ (x_1\dots x_N)^{\alpha+v -1}(y_1 \dots y_N)^{\alpha+u-1} }{ \left(\sum_{k=0}^{N-1} x_1\dots x_{k} y_{k+2}\dots y_N\right)^{u+v}} e^{-(x_1+\dots+x_N+y_1+\dots +y_N)},
\end{equation}
where $\mathcal{Z}_N=\mathcal{Z}_N(\alpha,u,v)$ is a  normalizing constant, which is known to be finite.
Then the law of %
\begin{multline}
    \label{H0} \vv L\topp N  := (\vv L\topp N_0,\vv L\topp N_1,\dots, \vv L\topp N_N)
    \\=\left(0,-\log X_1\topp N,-\log\big(X_1\topp N X_2\topp N\big),\dots,-\log\big(X_1\topp N\dots X_N\topp N\big)\right)
\end{multline}
is the
unique stationary law  of the log-gamma polymer \eqref{Log-G-Poly}.
  That is, with the initial condition $\vv H_0$   independent of $\{W_{i,j}\}_{(i,j)\in\mathbb{S}_N}$, such that  $\vv H_0 \stackrel{d}{=}  \vv L\topp N$,   we have $\vv H_n \stackrel{d}{=} \vv H_0$  for each $n=1,2,\ldots$.

 For typographical convenience, we write $\vv L$ in place of $\vv L\topp N$ whenever no confusion can arise.

Our first main result establishes the phase diagram for
 this stationary measure, see  Figure~\ref{Fig-PhD}.

\begin{theorem}\label{Thm:PhD}  The stationary measure of log-gamma polymer \eqref{H0}
 has the following scaling limits.
\begin{enumerate}
  [(i)]
  \item
 {      If $u,v\geq 0$ (maximal current), then $ \displaystyle \lim_{N\to\infty}\tfrac1N \vv L_N\topp N=
      -\psi(\alpha)$ in probability.
}
      \item If $u>v, v<0$ %
      (high
      density), then $  \displaystyle \lim_{N\to\infty}\tfrac1N \vv L_N\topp N=
      -\psi(\alpha+v)$ in probability.
      \item If $ u<v, u<0$ %
       (low  density), then $  \displaystyle\lim_{N\to\infty}\tfrac1N \vv L_N\topp N =
      -\psi(\alpha-u)$ in probability.
      \item If $u=v<0$ (coexistence line), then
      $ \displaystyle\tfrac1N \vv L_N\topp N \Rightarrow U$
      in distribution, where $U$ is  uniform  on the interval $[-\psi(\alpha-u), -\psi(\alpha+u)]$.
\end{enumerate}
\end{theorem}
 \begin{figure}[hbt]
  \BeginAccSupp{ActualText={Three regions of the phase diagram}}
  \begin{tikzpicture}[scale=.7]
\fill[pattern=dots, pattern color=gray!60] (5,5)--(5,10) -- (10,10)
-- (10,5) -- cycle;

\fill[pattern=horizontal lines, pattern color=gray!60] (1,1)--(5,5) -- (10,5)
-- (10,1) -- cycle;

\fill[pattern=vertical lines, pattern color=gray!60] (1,1)--(1,10) -- (5,10)
-- (5,5) -- cycle;

  \draw[->] (5,0) to (5,11);
 \draw[->] (0.,5) to (11,5);

 \draw[-,thick](1,1) to (11,1);
 \draw[-,thick](1,1) to (1,11);
 \draw[-,thick] (5,5) to (11,5);
 \draw[-,thick] (5,5) to (5,11);
 \draw[-,thick] (1,1) to (5,5);

  \node [below] at (8.5,8.5) {$\rho=-\psi(\alpha)$};
   \node at (8,2) {\footnotesize $\rho=-\psi(\alpha+v)$};
    \node  at (3,8) {\footnotesize $\rho=-\psi(\alpha-u)$};
  \node[below] at (11,5) {$u$};
    \node[left] at (5,11) {$v$};
        \node[left,below] at (1,1) {$(-\alpha,-\alpha)$};
\end{tikzpicture}
  \caption{Phase diagram for the stationary measure of the log-gamma polymer. The limit
  $\frac{1}{N}\vv  L_N\to \rho=\rho(\alpha,u ,v )$ depends on the boundary parameters $u,v> -\alpha$.
  The  line between the points $(-\alpha,-\alpha)$ and $(0,0)$ separates the {  regions with the low
  $\rho=-\psi(\alpha-u)$ and high $\rho=-\psi(\alpha+v)$ density, and for the parameters  $(u,v)$  on that line, the limit of  $\frac{1}{N}\vv  L_N$  is random. } }
  \label{Fig-PhD}
\end{figure}
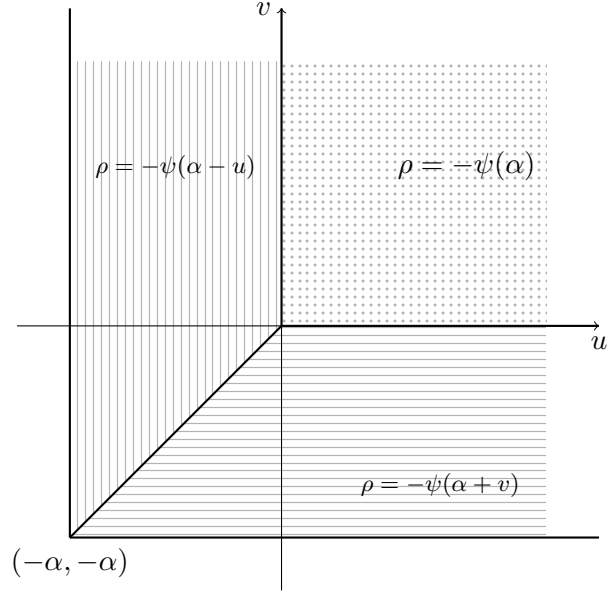

 We give two proofs of Theorem~\ref{Thm:PhD}, both based on the analysis of the Laplace transform. In Section~\ref{Proof of Thm:PHD}, we develop a probabilistic approach  based on its representations
  by independent  $\mathrm{Beta}_{II}$ random variables.  %
 An alternative analytic proof, based on Barraquand’s integral formula discussed in Section~\ref{Sec:BarInt}, is presented in Section~\ref{Sec:AltProof}.

\subsection{Stationary measures for  log-gamma polymer in a half-space}\label{Sec:Stat-meas-half}

For $\alpha>0$ and $u+\alpha>0$, the half-space log-gamma polymer partition function is defined by the recursion \eqref{one-step+} on the strip
$$\mathbb{S}_\infty:=\{(j,n):\; n\in \ZZ_{\ge 0}, \; j\in\{n, n+1, \dots\}\}\subset \ZZ_{\ge 0}^2.$$
 As previously, we use  independent random variables %
$\{W_{j,n}\}_{(j,n)\in\mathbb{S}_\infty}$ with  $1/W_{j,n}\simeq \mathrm{Gamma}(2\alpha)$ for $j>n$ and $1/W_{n,n}\simeq \mathrm{Gamma}(\alpha+u) $, $n=1,2,\dots$.

As in Definition~\ref{Def11}, a stationary measure is the law of a random sequence
$\{H_k\}_{k\in\ZZ_{\ge0}}$ with the property that, if the recursion \eqref{one-step+} on the half-infinite strip $\mathbb S_{\infty}$ is started from the initial condition
\[
Z_{k,0}=e^{H_k}, \qquad k\in\ZZ_{\ge0},
\]
then, for every $n\ge0$, the sequence
\[
\{\log Z_{n+k,n}-\log Z_{n,n}\}_{k\in\ZZ_{\ge0}}
\]
has the same distribution as $\{H_k\}_{k\in\ZZ_{\ge0}}$.

As already mentioned, for given $\alpha>0$ and $u>-\alpha$, Barraquand and Corwin
\cite{Barraquand-Corwin-AOP2024} constructed a one-parameter family of stationary measures indexed by a drift parameter $\tilde v$ satisfying $|\tilde v|<\alpha$ and $\tilde v\le u$.
To state their result, consider a random sequence $\{Z_k\}_{k\in\ZZ_{\geq 0}}$ %
given by
\begin{equation}\label{Half-plane Z}
  Z_k=\frac{1}{X_1\dots X_k}
  + \sum_{j=1}^k
  \frac{V}{Y_1\dots Y_j\, X_j\dots X_k},\quad k=1,2,\dots
\end{equation}
where
$
X_j \simeq \mathrm{Gamma}(\alpha-\tilde v)$,
$Y_j \simeq \mathrm{Gamma}(\alpha+\tilde v)$, $j=1,2,\dots$,
$V \simeq \mathrm{Gamma}(u-\tilde v)$,
and all random variables are independent. We take $V=0$ when $\tilde v=u$. We also define $Z_0=1$.
(Formula \eqref{Half-plane Z} is a reformulation of
\cite[Eq.~(1.6)]{Barraquand-Corwin-AOP2024}.)

 According to  \cite[Theorem 1.8]{Barraquand-Corwin-AOP2024}, for every
 $\alpha>0$ and $u,\tilde v>-\alpha$ such that $\tilde v\leq \min\{0,u\}$,
 the law of the sequence $\{\log Z_k\}_{k\in\ZZ_{\geq 0}}$ is a stationary measure for the log-gamma polymer in the half space.
The authors also note that  the distribution of the sequence $\{Z_k\}$ remains the same if
 \(\tilde v\) is replaced with \(-\tilde v\), see also \eqref{Z-int} below.

\medskip
It is natural to expect that, as $N\to\infty$, the stationary measures of log-gamma polymers on strips of width $N$ converge to a stationary measure of the half-space log-gamma polymer.
The following theorem, which is our second main result, establishes this convergence.

\begin{theorem}\label{Thm-half-lim}
Let $\{\vv L\topp N_k\}_{k\ge 0}$ be the  log-gamma polymer  \eqref{H0} on a strip of width $N$ distributed according to the stationary measure with parameters $\alpha,u,v$, and extended by appending infinitely many zeros.
As $N\to\infty$, %
\[
\left\{\vv L\topp N_k\right \}_{k\ge 0}
\Rightarrow
\left\{\log Z_k\right\}_{k\ge 0},
\]
where $\{Z_k\}_{k\ge 0}$ is   given by
\eqref{Half-plane Z} with parameters $\alpha,u$ and   $\tilde v=\min\{0,u,v\}$.
\end{theorem}

The proof of Theorem~\ref{Thm-half-lim} uses different methods in different regions of the phase diagram. Outside the maximal current region (Theorem~\ref{Thm:PhD}), it relies on $\mathrm{Beta}_{II}$ representations. For $u,v>0$, it instead uses Barraquand's integral representation of the multipoint Laplace transform, reproduced here as Theorem~\ref{BarraquandThm1.11}, together with a new contour integral representation for the multipoint Laplace transform of the stationary measure of the half-space log-gamma polymer, which constitutes our third main result.
\begin{theorem}\label{Thm:IntegralHS}
Fix $\alpha>0$, $u>-\alpha$,  $|\tilde v|<\alpha$  such that $\tilde v<u$ and $k\geq 1$.
Let $\{Z_k\}_{k\ge 0}$ be the stationary half-line log-gamma polymer \eqref{Half-plane Z} with drift parameter $\tilde v$.
Then     for
  $u>t_1>t_2>\dots>t_k> |\tilde v|$, with $t_{k+1}=0$.
\begin{multline}\label{Z-int}
 \EE\left[\prod_{r=1}^k Z_r^{2t_{r+1}-2t_r}\right]
 \\= \tfrac{1}{\mathbf{C}_k(\vv t)}
 \int_{(0,\infty)^{k}}
\Big(\prod_{j=1}^{k-1}
\left|
\Gamma(t_j-t_{j+1}+\i y_j+\i y_{j+1})
\Gamma(t_j-t_{j+1}+\i y_j-\i y_{j+1})
\right|^2\Big)
\\ \cdot
 \Big(
\prod_{j=1}^k
\tfrac{|\Gamma(\alpha+t_j+\i y_j)|^2}
     {|\Gamma(2\i y_j)|^2}
\Big)
\left|\Gamma(t_k+\tilde v+\i y_k)\right|^2 \left|\Gamma(t_k-\tilde v+\i y_k)\right|^2
 |\Gamma(u-t_1+\i y_1)|^2    %
\,\d \vv y ,
\end{multline}
where
\begin{equation*}
   \mathbf C_k(\vv t)=(2\pi)^k \Gamma(u+\tilde v)\Gamma(u-\tilde v) \Gamma(\alpha+\tilde v)^{k}
   \Gamma(\alpha-\tilde v)^{k}\prod_{j=1}^k \Gamma(2t_j-2t_{j+1}).
\end{equation*}
\end{theorem}

A version of Theorem~\ref{Thm:IntegralHS} permitting repeated parameters
$t_1 \ge t_2 \ge \dots\ge t_k$, in the spirit of Theorem~\ref{BarraquandThm1.11}, can also be established. The simplest instance, corresponding to equal parameters
{
$t_1=t_2=\dots=t_k=t$} and $\tilde v=0$
reads as follows.
\begin{multline*}
 \EE[Z_k^{-2t}]=\tfrac{1}{2\pi \Gamma(2t)\Gamma(u)^2\Gamma(\alpha)^{2k}}
 \\ \cdot
\int_0^\infty\frac{|\Gamma(\alpha+t+\i x)|^{2k}|\Gamma(t+\i x)|^4|\Gamma(u-t+\i x)|^2}{|\Gamma(2\i x)|^2}\,\d x.
\end{multline*}

\subsection{Overview of the paper}
The paper is organized as follows.

In Section~\ref{Sec:LapTr}, we study the one-point Laplace transform of the stationary   distribution of the log-gamma polymer on a strip. We first express  the Laplace transform in terms of the normalizing constant $\GG_N(\alpha,u,v)$ in \eqref{f-joint+}. Corollary~\ref{per1} { establishes } the asymptotic behavior of
$\GG_N(\alpha,u,v)$ %
in terms of $\mathrm{Beta}_{II}$ perpetuities (equivalently, exponential functionals of log-$\mathrm{Beta}_{II}$ random walks), yielding corresponding formulas for the Laplace transform. To the best of our knowledge, these representations are new; they play a key role in the proofs of Theorems~\ref{Thm:PhD} and~\ref{Thm-half-lim}.

In Section~\ref{Proof of Thm:PHD} we apply independent $\mathrm{Beta}_{II}$ random variables to give a probabilistic proof of Theorem~\ref{Thm:PhD}.

In Section~\ref{Sec:BarInt} we discuss the one-point version of Barraquand's contour integral formula { for the Laplace transform of the stationary measure of the log-gamma polymer on a strip} and give explicit formulas for its analytic continuation in Proposition~\ref{Prop2}. We then study the multipoint Laplace transform by deriving an integral representation in terms of independent $\mathrm{Beta}_{II}$ random variables and giving a self-contained proof of multipoint Barraquand's contour integral formula.

Section~\ref{Sec:AC} is devoted to the proof of Proposition~\ref{Prop2} and to a further study of the analytic continuation of Barraquand's integral representation for the one-point Laplace transform.

In Section~\ref{Sec:AltProof}, we present an alternative   proof of Theorem~\ref{Thm:PhD}, based on the analytic tools developed in Section~\ref{Sec:AC}.

Finally, in Section~\ref{Sec:H-HLGP}, we derive %
integral representations, in terms of independent $\mathrm{Beta}_{II}$ random variables, for the multipoint Laplace transform of the stationary measures of the half-space log-gamma polymer and use them { (together with the contour integral approach)} to prove Theorems~\ref{Thm-half-lim} and~\ref{Thm:IntegralHS}.

For the reader's convenience, the appendices collect several useful results and a proof of an integral identity stated in Lemma~\ref{Lem:H=L}, for which we have been unable to locate a reference in the literature.

\section{Normalizing constant $\GG_N(\alpha,u,v)$ and the Laplace transform}\label{Sec:LapTr}
In this section we discuss the univariate Laplace transform \(
\EE[e^{-2t\,\vv L_N}]
\), which  we use in the proof of Theorem~\ref{Thm:PhD}.
Since $ e^{-2t\,\vv L_N}=\left(X_1\topp N\dots X_N\topp N\right)^{2t}$,
it follows from \eqref{f-joint+} that the Laplace transform
can be expressed as the ratio of two integrals,
\begin{equation}\label{LG-xx}
\EE\bigl[e^{-2t\,\vv L_N}\bigr]
=
\frac{\GG_N(\alpha+t,u-t,v+t)}
     {\GG_N(\alpha,u,v)},
\end{equation}
where
\begin{multline}\label{II-1}
  \GG_N(\alpha,u,v)
  \\ =\int_{(0,\infty)^{2 N}}
  \frac{ (x_1\dots x_N)^{\alpha+v -1}(y_1 \dots y_N)^{\alpha+u-1} }{ \left(\sum_{k=0}^{N-1} x_1\dots x_{k} y_{k+2}\dots y_N\right)^{u+v}} e^{-(x_1+\dots+x_N+y_1+\dots +y_N)} \d \vv x \d \vv y
\end{multline}
is the normalizing constant from \eqref{f-joint+}.
 Recall that $\alpha,\alpha+u,\alpha+v>0$.
 By evaluating $(N+1)$ of the $2N$ integrals in \eqref{II-1}, we arrive at the following integral representation.
\begin{lemma} \label{Lem:Jacek-Beta}
    \begin{multline}
      \label{Jacek-Beta}
    \frac{\GG_N(\alpha,u,v)}{\Gamma(\alpha+v)\Gamma(\alpha+u)(\Gamma(2\alpha))^{N-1}}
    \\=\,\int_{(0,\infty)^{N-1}}\,\Big(\prod_{j=1}^{N-1}\,\frac{z_j^{\alpha+u-1}}{(1+z_j)^{2\alpha}}\Big)\,\frac{\mathrm dz_1\ldots\mathrm dz_{N-1}}{(1+\sum_{j=1}^{N-1}\,\prod_{k=1}^j\,z_k)^{u+v}}.
    \end{multline}

\end{lemma}
\begin{proof}
Using the formula \eqref{II-1} and integrating out $x_N$ and $y_1$ we get
 \begin{multline*}
 \GG_N(\alpha,u,v)
 =\Gamma(\alpha+v)\Gamma(\alpha+u)\\
  \cdot \int_{\RR_+^{2 (N-1)}}
  \tfrac{ (x_1\dots x_{N-1})^{\alpha+v -1}(y_2 \dots y_N)^{\alpha+u-1}e^{-(x_1+y_2%
  +\dots+x_{N-1}+y_N)} }{ \left(x_1\dots x_{N-1}\big(1+\frac{y_N}{x_{N-1}}+\frac{y_{N-1}}{x_{N-2}}\frac{y_N}{x_{N-1}}+\dots+\prod_{j=1}^{N-1}\,\frac{y_{j+1}}{x_j}\big)\right)^{u+v}}  \d \vv x \d \vv y.
 \end{multline*}
 Substituting %
 $z_j=\tfrac{y_{N-j+1}}{x_{N-j}} $
 for $j=1,\ldots,N-1$ while keeping $x_j$, $j=1,\ldots,N-1$, we see that the Jacobian is $x_1\ldots x_{N-1}$ and thus
  \begin{multline*}
\frac {\GG_N(\alpha,u,v)}{\Gamma(\alpha+v)\Gamma(\alpha+u)}
=
 \int_{(0,\infty)^{N-1}}
 \Big(\int_{(0,\infty)^{N-1}}\,(x_1\dots x_{N-1})^{2\alpha-1}
     \\ \cdot  e^{-x_1(1+z_{N-1})-x_2(1+z_{N-2})-\ldots-x_{N-1}(1+z_{1})}
   \d \vv x\Big)
   \tfrac{ (z_1 \dots z_{N-1})^{\alpha+u-1}}{ \left(1+z_1+z_1z_2+\ldots +z_1z_2\ldots z_{N-1}\right)^{u+v}}\,\d \vv z.
 \end{multline*}
 Integrating with respect to $\d \vv x $ w obtain  \eqref{Jacek-Beta}.
\end{proof}

 \arxiv
 {
Trivially, $\GG_1(\alpha,u,v)=\Gamma(\alpha+u)\Gamma(\alpha+v)$. One can check that
\begin{equation*}
    \label{GG2} \GG_2(\alpha,u,v)= \frac{\Gamma (2 \alpha ) \Gamma (u+\alpha )^2 \Gamma (v+\alpha
   )^2}{\Gamma (u+v+2 \alpha )}.
\end{equation*}
}
A positive random variable $\zeta$ is said to have the $\mathrm{Beta}_{II}(a,b)$ distribution, where $a,b>0$,
if it has density $$\frac{\Gamma(a+b)}{\Gamma(a)\Gamma(b)} \frac{x^{a-1}}{(1+x)^{a+b}}\; \mathbf{1} _{x>0}. $$
 Since $\GG_N(\alpha,u,v)=\GG_N(\alpha,v,u)$, see \eqref{II-1},
 Lemma \ref{Lem:Jacek-Beta} gives immediately  the following two probabilistic expressions for the normalizing constant.

\begin{corollary}\label{per1}   If $|u|<\alpha$ { and $v>-\alpha$}, then
        \begin{multline}\label{perpet}
        \frac{\GG_N(\alpha,u,v)}{\Gamma(\alpha+v)\Gamma(\alpha+u)^N\Gamma(\alpha-u)^{N-1}}
        \\=\mathbb E\,\left[(1+\zeta_1+\zeta_1\zeta_2+\ldots+\zeta_1\zeta_2\ldots \zeta_{N-1})^{-u-v}\right],
    \end{multline}
    where %
     $(\zeta_i)_{i\ge 1}$ is a sequence of
    i.i.d. $\mathrm{Beta}_{II}(\alpha+u,\alpha-u)$ random variables.

     If $|v|<\alpha$   and $u>-\alpha$, then
        \begin{multline}\label{perpet1}
        \frac{\GG_N(\alpha,u,v)}{\Gamma(\alpha+u)\Gamma(\alpha+v)^N\Gamma(\alpha-v)^{N-1}}
        \\ =\mathbb E\,\left[(1+  \zeta_1'+\  \zeta_1'\  \zeta_2'+\ldots+\  \zeta_1'\  \zeta_2'\ldots \  \zeta_{N-1}')^{-u-v}\right],
    \end{multline}
    where %
     $(\  \zeta_i')_{i\ge 1}$ is a sequence of
    i.i.d. $\mathrm{Beta}_{II}(\alpha+v,\alpha-v)$ random variables.
\end{corollary}

\begin{remark}
{The expression on the right-hand side of \eqref{perpet} (and, similarly, on the right-hand side of  \eqref{perpet1}) is a  moment of a random variable }
\begin{equation}\label{RRn}
    R_n:=1+\zeta_1+\zeta_1\zeta_2+\ldots+\zeta_1\zeta_2\dots\zeta_n,
\end{equation}
where $\zeta_j>0$ are i.i.d. Note that $R_n \stackrel{d}{=}\tilde R_{n}$, where  $\tilde R_k=1+\tilde R_{k-1} \zeta_k$ for $k=1,2,\ldots$,   $\tilde R_0=1$.
Such Markov sequences $\{\tilde R_n\}$ arise in a variety of contexts, including financial mathematics, where they are known as perpetuities.
Additional background on perpetuities and related stochastic recursions can be found in \cite{Kesten1973,goldie1996perpetuities,goldie2000perpetuities,buraczewski2016stochastic,grincevicius1975limit}.

 Random variable $R_n$ can also be expressed as an (additive) exponential functional of the random walk with increments
$\log\zeta_j$; { such functionals have been studied e.g. in }  \cite{Hirano-1997,lePage1997local,Xu-Wei:2023}.
\end{remark}

Probabilistic expressions in Corollary~\ref{per1} yield  asymptotics for the normalizing constant $\GG_N(\alpha,u,v)$. We write $a_n\sim b_n$ if $\lim_{n\to\infty} a_n/b_n=1$.

By the symmetry of \eqref{II-1} in $u$ and $v$, the asymptotics for $-\alpha<v<0$ and $u>v$ follows from the case $-\alpha<u<0$ and $v>u$, which we now consider.
If $0< p<-2u$  and $\zeta$ has $\mathrm{Beta}_{II}(\alpha+u,\alpha-u)$ law,  then
  \begin{equation}\label{beta-moment}
   \mathbb E\,[\zeta^p]= \tfrac{\Gamma(\alpha-u-p)\Gamma(\alpha+u+p)}{\Gamma(\alpha-u)\Gamma(\alpha+u)}  <1
  \end{equation}
since the moment is $1$ at $p=0$ and at $p=-2u$. Consequently, see e.g. \cite[Theorem 2.1]{goldie2000perpetuities}, the series
\begin{equation}\label{Rseries}
 R=1+\zeta_1+\zeta_1\zeta_2+\ldots+\zeta_1\zeta_2\ldots\zeta_n+\dots
\end{equation}
converges almost surely and in $p$-th moments. Denoting by $R_n$ the partial sum of the series, see \eqref{RRn}, we see that
$$\EE[R_n^{-u-v}]\to \EE[R^{-u-v}].$$ Indeed, if $u+v<0$ we can take $p=-u-v$, and if $u+v\geq 0$ then random variables
$R_n^{-u-v}$ are bounded by $1$.

This gives the  following asymptotics for the normalizing constant.
\begin{lemma}\label{Cor:from-perp}
   If $-\alpha<u<0$ and $v>u$ then,   as $N\to\infty$, we have
 \begin{equation}
     \label{Cor:from-perp-A}
     \GG_N(\alpha,u,v)\sim \Gamma(\alpha-u)^{N}\Gamma(\alpha+u)^N \tfrac{\Gamma(\alpha+v)}{\Gamma(\alpha-u)}\EE[R^{-u-v}],
 \end{equation}
 where $R$ is  given by \eqref{Rseries}.

If $-\alpha<v<0$ and $u>v$, then, as $N\to\infty$, we have
\begin{equation}\label{Cor:from-perp-B}
  \GG_N(\alpha,u,v)\sim
  \Gamma(\alpha+v)^N\Gamma(\alpha-v)^N
  \tfrac{\Gamma(\alpha+u)}{\Gamma(\alpha-v)}
  \EE[(R')^{-u-v}],
\end{equation}
where
\begin{equation}\label{Rseries2}
  R'  =1+\zeta_1'+\zeta_1'\zeta_2'
  +\ldots+\zeta_1'\dots\zeta_n'+\dots,
\end{equation}
and $(\zeta_i')_{i\ge1}$ are i.i.d.\ $\mathrm{Beta}_{II}(\alpha+v,\alpha-v)$ random variables.
\end{lemma}
 {We remark that this asymptotics can be made even more explicit. According to \cite[Example 9]{chamayou1991explicit}, $R-1\simeq \mathrm{Beta}_{II}(\alpha+u,-2u)$ so
 $$\EE[R^{-u-v}]=\tfrac{\Gamma(\alpha-v)\Gamma(v-u)}{\Gamma(\alpha+v)\Gamma(-2 u)}.$$
 This also follows by comparing \eqref{Cor:from-perp-A} with expression \eqref{C1-Mo} for $t=0$.
Similarly, comparing \eqref{Cor:from-perp-B} with  expression \eqref{C2-m0} for $t=0$, we obtain
$$\EE[ (R')^{-u-v}]=\tfrac{\Gamma(\alpha-v)\Gamma(u-v)}{\Gamma(\alpha+u)\Gamma(-2v)}.$$
   }

We will also need asymptotics of the normalizing constant under small perturbations of the parameters. Note that by Corollary~\ref{per1}, for every $N$ we have
\begin{multline*}
\tfrac{\GG_N\!\left(\alpha+\frac{t}{N},\,u-\frac{t}{N},\,v+\frac{t}{N}\right)}
{\Gamma\!\left(\alpha+v+\frac{2t}{N}\right)\Gamma(\alpha+u)^N
\Gamma\!\left(\alpha-u+\frac{2t}{N}\right)^{N-1}}
\\ =
\mathbb{E}\!\left[
\left(1+\zeta_{N,1}+\zeta_{N,1}\zeta_{N,2}
+\ldots+\zeta_{N,1}\zeta_{N,2}\dots \zeta_{N,N-1}\right)^{-u-v}
\right].
\end{multline*}
Here $\{\zeta_{N,1},\zeta_{N,2},\ldots\}_{N\ge1}$ is a triangular array of   random variables such that, for each fixed $N$, the random variables $\zeta_{N,1},\zeta_{N,2},\ldots$ are independent and identically distributed with  %
$
\mathrm{Beta}_{II}\!\left(\alpha+u,\;\alpha-u+\tfrac{2t}{N}\right)
$ distribution.

Interchanging the roles of $u$ and $v$, we also obtain
\begin{multline*}
\tfrac{\GG_N\left(\alpha+\frac tN,u
-\frac tN,v+\frac tN\right)}{\Gamma(\alpha+u)\Gamma(\alpha-v)^{N-1}\Gamma(\alpha+v+\frac{2t}{N})^N}
 \\=\mathbb E\,\left[(1+ \zeta_{N,1}'+ \zeta_{N,1}' \zeta_{N,2}'+\ldots+ \zeta_{N,1}' \zeta_{N,2}'\ldots \zeta_{N,N-1}')^{-u-v}\right],
\end{multline*}
where $\{\zeta'_{N,j}\}_{N,j\ge1}$ is a triangular array whose rows are i.i.d.\ with
$
\mathrm{Beta}_{II}\!\left(\alpha+v+\tfrac{2t}{N},\,\alpha-v\right)$ distribution.

The asymptotics of these expressions can be determined from the following lemma.
\begin{lemma}\label{triangle}
Let $-\alpha<u<0$ and $v>u$. %
Then , as $N\to\infty$,
\begin{equation}\label{limrn}
R_{N,N-1}:=1+\zeta_{N,1}+\zeta_{N,1}\zeta_{N,2}+\ldots+\zeta_{N,1}\zeta_{N,2}\ldots \zeta_{N,N-1}{\Rightarrow}%
R,
\end{equation}
where $R=\lim_{n\to\infty} R_n$   is   given by \eqref{Rseries}; furthermore,
\begin{equation}\label{limuv}
\lim_{N\to\infty}\,\mathbb E\,[R_{N,N-1}^{-u-v}]\,=\,\mathbb E\,[R^{-u-v}].
\end{equation}

Let $-\alpha<v<0$ and  $v<u$. %
Then, as $N\to\infty$, \begin{equation}\label{limrn2}
  R_{N,N-1}':=1+  \zeta_{N,1}'+  \zeta_{N,1}'  \zeta_{N,2}'+\ldots+
  \zeta_{N,1}'  \zeta_{N,2}'\ldots   \zeta_{N,N-1}'{\Rightarrow} %
  R',
\end{equation}
where $  R'=\lim_{n\to\infty}   R_n'$  is  given by \eqref{Rseries2}; furthermore,
\begin{equation}\label{limuv2}
\lim_{N\to\infty}\,\mathbb E\,[  (R_{N,N-1}')^{-u-v}]\,=\,\mathbb E\,[ (R')^{-u-v}].
\end{equation}
\end{lemma}
\begin{proof}
Fix $t\in\mathbb{R}$. Since the random variables $\zeta_{N,j}$ are i.i.d.\ (for each fixed $N$) and converge in distribution to $\zeta_j$, Skorohod's representation theorem implies that, without loss of generality, we may assume that all variables are defined on a common probability space and that
$
\zeta_{N,j} -\zeta_j \to 0$ almost surely and in $p$-th moments as $N\to\infty$ for any $j=1,2,\dots,$ and
 $0<p <\alpha-u$.

For example, one may construct such a coupling via distribution functions by setting
\[
\zeta_{N,j} = F_N^{-1}\bigl(F(\zeta_j)\bigr),
\]
where $F_N$ and $F$ denote the distribution functions of $\zeta_{N,j}$ and $\zeta_j$, respectively.

With this coupling, we will verify convergence in moments using the following $L^p$-type metric:
\begin{equation}\label{p-metric}
   \|X\|_p :=
\begin{cases}
\mathbb{E}\left[|X|^p\right], & 0 < p \le 1, \\[4pt]
\left(\mathbb{E}\left[|X|^p\right]\right)^{1/p}, & 1 \le p < \infty.
\end{cases}
\end{equation}
    The key observation is that if $-u-v\leq p<-2u$ and $p> 0$  then, as we noted in \eqref{beta-moment}, we have
   \begin{equation}\label{rho-bound}
    \lim_{N\to\infty} \mathbb E\,[\zeta_{N,1}^p]=\lim_{N\to\infty}\tfrac{\Gamma(\alpha-u+\frac{2t}{N}-p)\Gamma(\alpha+u+p)}{\Gamma(\alpha-u+\frac{2t}{N})\Gamma(\alpha+u)}
     =\mathbb E\,[\zeta^p]<\mathbb E\,[\zeta^{-2u}]=1.
  \end{equation}
  Hence, there exists a constant $\rho=\rho(p)\in(0,1)$ such that, for all sufficiently large $N$, we have $\|\zeta_{N,1}\|_p<\rho$.

 Given $\eps>0$, choose $K$ such that $\rho^K/(1-\rho)<\eps$.
 As %
 { in \eqref{RRn}}, let $R_k$ denote the partial sum of the first $k+1$ terms of the series \eqref{Rseries}, and define
 $$
R_{N,k}= 1+\zeta_{N,1}+\zeta_{N,1}\zeta_{N,2}+\ldots+\zeta_{N,1}\zeta_{N,2}\ldots \zeta_{N,k}.
$$

 Then by the triangle inequality and independence, we obtain
\begin{multline*}
    \limsup_{N\to\infty} \|R_{N,N-1}-R\|_p\leq \limsup_{N\to\infty}  \|R_{N,K}-R_{K}\|_p
   \\  +\limsup_{N\to\infty}\sum_{k=K}^\infty \|\zeta_{N,1}\|_p^k+\sum_{k=K}^\infty \|\zeta\|_p^k
    \\ \leq \limsup_{N\to\infty}  \|R_{N,K}-R_{K}\|_p
    +2\sum_{k=K}^\infty \rho^k <2\eps,
\end{multline*}
since $\|R_{N,K}-R_K\|_p\to 0$  for a fixed $K$.

If $u+v<0$, then  taking $p=-u-v$ in the above limit  establishes both \eqref{limrn} and \eqref{limuv}.  If $u+v\geq 0$, then \eqref{limrn} follows from the convergence of moments of order $p>0$, for instance with  $p=-u$; consequently, \eqref{limuv} holds since $R_{N,N-1}^{-u-v}\in(0,1]$. This completes the proof of  the first part of the lemma.

The second part of the lemma follows by a similar  argument, using moments of order $p$ with $-u-v\leq p<-2 v$.
\end{proof}
\arxiv
{  Here is the proof of the second part of the lemma.
Without loss of generality, we assume that $\  \zeta_{N,j}'-\  \zeta_j'\to 0$ as $N\to \infty$ almost surely and in moments.

Let $p>0$ be such that $-u-v\leq p<-2 v$.  Then we have the following version of \eqref{rho-bound}:
$$
\EE[\  (\zeta_{N,1}')^p]= \frac{\Gamma (-p-v+\alpha ) \Gamma
   \left(p+\frac{2 t}{N}+v+\alpha
   \right)}{\Gamma (u+v+\frac{2 t}{N}) \Gamma (\alpha -v)}\to \EE[\  (\zeta')^p]<\EE[\  (\zeta')^{-2v}]=1.
$$
 Thus, there exists a constant  $\rho\in(0,1)$ such that for all sufficiently large  $N$, we  have  $\|\  \zeta_{N,1}'\|_p<\rho$.

 As in the previous argument, given $\eps>0$, choose $K$ such that $\rho^K/(1-\rho)<\eps$. Then, with $R_K'$ denoting the sum of the first $K$ terms of the series \eqref{Rseries2}, we have
$$
     \| R_{N,N-1}'-  R'\|_p\leq  \|  R_{N,K}'-\  R_K'\|_p+
     \sum_{k=K}^\infty \|\ \zeta_{N,1}'|_p^k+  \sum_{k=K}^\infty \|\  \zeta'|_p^k
     \\  \leq  \|\  R_{N,K}'-\  R_K'\|_p+2\eps
$$
and convergence follows. As previously, if $u+v<0$ we take $p=-u-v$, and if $u+v\geq 0$ we  deduce \eqref{limrn2} from convergence of moments, say,
of order $p=-v$ and then infer \eqref{limuv2} from the fact that the sequence $  (R_{n}')^{-u-v}$ is bounded.
}

\section{Proof of Theorem \ref{Thm:PhD}}\label{Proof of Thm:PHD}

The proof   of Theorem \ref{Thm:PhD} proceeds by analyzing the asymptotic behavior of the logarithm of the Laplace transform \(
\EE[e^{-2t\,\vv L_N}]
\).
  In this section, we carry out this analysis using a probabilistic approach involving independent $\mathrm{Beta}_{II}$ random variables. An alternative analytic argument, based on Barraquand’s contour integral formulas, is presented in Section \ref{Sec:AltProof}.

\subsection{Proof of Theorem \ref{Thm:PhD}(i)}
 From  \eqref{Jacek-Beta}, for $t$ in a neighbourhood of zero (in particular for $t=0$), we have
\begin{multline*}
 \GG_N\left(\alpha+\tfrac tN,u-\tfrac tN,v+\tfrac tN\right)
 \\ =\tfrac{\Gamma\left(\alpha+v+\frac{2t}{N}\right)\Gamma(\alpha+u)\Gamma^{N-1}\left(2\alpha+\frac{2t}N\right)\,\Gamma^{2(N-1)}(\alpha)}{\Gamma^{N-1}(2\alpha)}\,\mathbb E\,[D_{N}\,e^{-\frac{2t}{N}\sum_{k=1}^{N-1}\,\log(1+\eta_k)}],
\end{multline*}
 where $(\eta_n)_{n\ge 1}$ is a sequence of i.i.d. $\mathrm{Beta}_{II}(\alpha,\alpha)$ r.v's and
 $$
 D_{N}:=D_N(\eta_1,\ldots,\eta_{N-1})= \tfrac{\prod_{k=1}^{N-1}\,\eta_k^u}{\left(1+\sum_{j=1}^{N-1}\,\prod_{i=1}^j\,\eta_i\right)^{u+v}}.
 $$
That is
\begin{equation}\label{ELG}
\EE[e^{-2 t \vv L_N/N}]=\tfrac{\GG_N\left(\alpha+\frac tN,u-\frac tN,v+\frac tN\right)}{\GG_N(\alpha,u,v)}=C_N\,\mathbb E\,[e^{-2t\widetilde X_N}],
\end{equation}
where $\widetilde X_N=\tfrac1{N}\sum_{k=1}^{N-1}\,\log(1+Y_{N,k})$ for $\vv Y_N=(Y_{N,1},\ldots,Y_{N,N-1})$, which has the law
\[
\mathbb P_{\vv Y_N}(\mathrm d\mathbf y)=\tfrac{D_N(y_1,\ldots,y_{N-1})}{\mathbb E\,[D_N]}\,\mathbb P_{(\eta_1,\ldots,\eta_{N-1})}(\mathrm d\mathbf y),\quad N=2,3,\ldots,
\]
and
\[
C_N=\tfrac{\Gamma(\alpha+v+\frac{2t}{N})}{\Gamma(\alpha+v)}\,\left(\tfrac{\Gamma(2\alpha+\frac{2t}N)}{\Gamma(2\alpha)}\right)^{N-1}\to\,e^{2t\psi(2\alpha)}.
\]
Thus, in view of \eqref{ELG}, it suffices to prove that $\widetilde X_N\stackrel{\mathbb P}{\to}\mu:=\mathbb E[\log(1+\eta)]=\psi(2\alpha)-\psi(\alpha).$

\bigskip
Let $S_k=\sum_{i=1}^k\,\log\eta_i$, $k=1,2,\ldots$. Since  $\mathbb E\,[\log\eta_i]=0$ and $\mathbb V\mathrm{ar}\,[\log\eta_i]=2\psi_1(\alpha)<\infty$, by the Kolmogorov maximal inequality for $\delta>0$ and $\kappa=\tfrac{1+\delta}2\in(\tfrac12,1)$ we have
$$
\mathbb P(\max_{1\le k\le N}\,|S_k|\ge N^{\kappa})\le \tfrac{\mathbb V\mathrm{ar}\,S_N}{N^{2\kappa}}=\tfrac{2\psi_1(\alpha)}{N^{\delta}}\to 0.
$$

Therefore for $B_N=\{\max_{1\le k\le N}\,|S_k|< N^{\kappa}\}$ we have $\mathbb P(B_N)\to 1$. On $B_N$ we have
\[
1+\sum_{j=1}^{N-1}\,\prod_{i=1}^j\,\eta_i\le N\,e^{N^{\kappa}}\quad\mbox{and}\quad \prod_{k=1}^N\,\eta_i=e^{S_N}\ge e^{-N^{\kappa}}.
\]
We obtain
\begin{equation}\label{Kol}
\mathbb E\,[D_N]\ge \mathbb E\,[D_N\mathbf 1_{B_N}]\ge \tfrac{e^{-uN^{\kappa}}}{\left(Ne^{N^{\kappa}}\right)^{u+v}}\,\mathbb P(B_N)=N^{-u-v}\,e^{-(2u+v)N^{\kappa}}\mathbb P(B_N).
\end{equation}

Set $X_N=\tfrac1{N}\sum_{k=1}^{N-1}\,\log(1+\eta_k)$. In view of the Chernoff inequality there exists $c>0$ such that $\mathbb P(|X_N-\mu|\ge\epsilon)<e^{-cN}$. Note that since $u,v\ge 0$ we have $D_N\in[0,1]$.
Together with \eqref{Kol}, this completes the proof, as for any $\epsilon>0$ we obtain
\begin{multline*}
  \mathbb P(|\widetilde X_N-\mu|\ge \epsilon)=\tfrac{\mathbb E\,\left[D_N\mathbf 1_{|X_N-\mu|\ge\epsilon}\right]}{\mathbb E\,[D_N]}\le \tfrac{\mathbb P(|X_N-\mu|>\epsilon)}{\mathbb E\,[D_N]}
  \\ \le \tfrac{e^{-cN}}{N^{-u-v}e^{-(2u+v)N^{\kappa}}\mathbb P(B_N)}\stackrel{\kappa<1}{\longrightarrow} 0.
\end{multline*}

\subsection{Proof of Theorem~\ref{Thm:PhD}(ii)}
In view of \eqref{Cor:from-perp-B} %
and  {  the second part of}  Lemma \ref{triangle} we see that
\begin{multline*}
\lim_{N\to\infty}  \log \EE[e^{-2 t \vv L_N/N}]=
\lim_{N\to\infty} \log \tfrac{ \GG_N(\alpha+t/N,u-t/N,v+t/N)}{ \GG_N(\alpha,u,v)}\\
=\lim_{N\to\infty} N \big( \log\Gamma(\alpha+v+\tfrac{2t}{N})- \log\Gamma(\alpha+v)\big) = 2 t \psi(\alpha+v).
\end{multline*}
\qed
 \subsection{Proof of Theorem~\ref{Thm:PhD}(iii)}
In view of \eqref{Cor:from-perp-A} %
and { the first part of} Lemma \ref{triangle} we see that
\begin{multline*}
\lim_{N\to\infty}  \log \EE[e^{-2 t \vv L_N/N}]= \lim_{N\to\infty} \log \tfrac{ \GG_N(\alpha+t/N,u-t/N,v+t/N)}{ \GG_N(\alpha,u,v)}\\
=\lim_{N\to\infty} N \left(\log\Gamma(\alpha-u+\tfrac{2t}{N})-\log\Gamma(\alpha-u)\right)=2t\,\psi(\alpha-u).
\end{multline*}
 \qed

\subsection{Proof of Theorem \ref{Thm:PhD}(iv)}
We again consider the asymptotics of the  denominator and numerator in \eqref{LG-xx}.
Denote  by $w\in(0,\alpha)$ the common value of $-u=-v$.
Take $s\ge 0$.  Then, by Corollary~\ref{per1},
\begin{multline*}
 \mathcal Z_N\!\left(\alpha+\tfrac s2,u-\tfrac s2,u+\tfrac s2\right)
\\ =
\Gamma \left(\alpha-w+s\right)
\Gamma^N(\alpha-w)
\Gamma^{N-1}\!\left(\alpha+w+s\right)\,\EE\left[(R_{N-1}(s))^{2w}\right],
\end{multline*}
where
\[
R_k(s)=1+\zeta_1(s)+\dots+\zeta_1(s)\dots\zeta_k(s), \quad k\ge 1,\quad R_0(s)=1,
\]
and random variables $\zeta_j(s)\simeq \mathrm{Beta}_{II}(\alpha-w,\alpha+w+s)$, $j\in \mathbb N$, are independent. We note that the moments \(M_k(s):=\EE[(R_k(s))^{2w}]\) satisfy the recurrence relation
\begin{equation}\label{E-rec'}
M_k(s)=\Delta_k(s)+\rho(s) M_{k-1}(s),\quad k\ge 1,\quad \mbox{with }\;M_0(s)=1,
\end{equation}
where
\begin{equation}\label{DeltaN'}
\Delta_k(s)
:=
\EE\!\left[(R_k(s))^{2w}
-(\zeta_k(s)R_{k-1}(s))^{2w}\right] \quad \mbox{(with $\Delta_0(s):=1$)}
\end{equation}
and
\begin{equation}\label{W_N'}
\rho(s):=\EE[(\zeta_k(s))^{2w}]=\tfrac{\Gamma(\alpha+w)}{\Gamma(\alpha+w+s)}\,\tfrac{\Gamma(\alpha-w+s)}{\Gamma(\alpha-w)}.
\end{equation}
The solution of \eqref{E-rec'} is
\begin{equation*}
M_k(s)
=
\sum_{j=0}^k \Delta_j(s)\rho^{k-j}(s).
\end{equation*}
In particular, for $s=0$, we omit the argument $(0)$ and write $M_k=\sum_{j=0}^k\,\Delta_j$.

Consequently, for $s=s_N:=\tfrac{2t}{N}$ %
we have
\[
\tfrac{\mathcal Z_N\!\left(\alpha+s_N,u-s_N,u+s_N\right)}{\mathcal Z_N\!\left(\alpha,u,u\right)}\,=J_1(N)J_2(N)J_3(N),
\]
where
\[
J_1(N)=\tfrac{\Gamma(\alpha-w+s_N)}{\Gamma(\alpha-w)}\to 1,\quad J_2(n)=\left(\tfrac{\Gamma(\alpha+w+s_N)}{\Gamma(\alpha+w)}\right)^{N-1}\to e^{2t\psi(\alpha+w)},
\]
and
\[
J_3(N)=\tfrac{M_{N-1}(s_N)}{M_{N-1}}=\tfrac{\frac 1N\left(M_{N-1}(s_N)-\Delta\sum_{j=0}^{N-1}\,\rho^j(s_N)\right)}{\tfrac 1N\,M_{N-1}}+\tfrac{\Delta}{\frac 1NM_{N-1}}\,\tfrac1N\frac{1-\rho^N(s_N)}{1-\rho(s_N)},
\]
where
\begin{equation}\label{Delta}
    \Delta=\tfrac{\Gamma(\alpha+w)}{\Gamma(2w)\Gamma(\alpha-w)}\left(\psi(\alpha+w)-\psi(\alpha-w)\right).
\end{equation}

We now apply Lemma~\ref{Lem:Kronecker'}, to be proved later.
\begin{lemma}\label{Lem:Kronecker'}
Let $s_N=\tfrac{2t}{N}$ for $t\ge 0$ and $N\ge 1$. Then
\begin{equation}
    \label{Cesaro} \lim_{N\to\infty}\,\tfrac{M_{N-1}(s_N)-\Delta\sum_{j=0}^{N-1}\,\rho^j(s_N)}{N}=0.
\end{equation}
\end{lemma}
For $t=0$ Lemma \ref{Lem:Kronecker'} implies  $\tfrac1N M_{N-1} \to \Delta$
and, again referring to Lemma \ref{Lem:Kronecker'}, this time for $t>0$, we conclude that %
\[
\lim_{N\to\infty}\,J_3(N)\,=\lim_{N\to\infty}\,\tfrac1N\tfrac{1-\rho^N(s_N)}{1-\rho(s_N)}=
\tfrac{1-e^{2t(\psi(\alpha-w)-\psi(\alpha+w))}}
{2t(\psi(\alpha+w)-\psi(\alpha-w))},
\]
since, by \eqref{W_N'},
\[
\rho(s_N)=1+\tfrac{2t(\psi(\alpha-w)-\psi(\alpha+w))}{N}+o(\tfrac1N).
\]
Multiplying the limits, we obtain
$$\lim_{N\to\infty}\EE[e^{-2 t \vv L_N/N}] =
 \frac{1}{2t}\;\tfrac{e^{2\psi(\alpha+w) t}-e^{2
   \psi(\alpha-w) t}}{\psi(\alpha+w)  - \psi(\alpha-w)}{=
   \int_{-\psi(\alpha+w)}^{-\psi(\alpha-w)}\,\tfrac{e^{-2t\theta}}{\psi(\alpha+w)-\psi(\alpha-w)}\,\mathrm d\theta.}$$
Hence, for $t> 0$, the limiting Laplace transform corresponds to the uniform distribution on $[-\psi(\alpha+w),-\psi(\alpha-w)]$.
To complete the proof, we invoke \cite[Theorem 2]{Mukherjea-Rao-Suen-2006}, %
{ which says that convergence of the Laplace transforms on an open interval implies weak convergence.}
\begin{proof}[Proof of Lemma \ref{Lem:Kronecker'}]
Recall \eqref{DeltaN'}. Since $\rho(s_N)\in(0,1]$ we have
\begin{multline*}
  \left|M_{N-1}(s_N)-\Delta\sum_{j=0}^{N-1}\,\rho^j(s_N)\right|\le \sum_{j=0}^{N-1}\,\left|\Delta_j(s_N)-\Delta\right|\,\rho^{N-1-j}(s_N)
  \\ \le \sum_{j=0}^{N-1}\,\left|\Delta_j(s_N)-\Delta\right|.
\end{multline*}

It suffices to prove that
\begin{equation}\label{jN}
K:=\sup_{j,N}\Delta_j(s_N)<\infty
\qquad\text{and}\qquad
\lim_{\substack{j\to\infty\\ N\to\infty}}\Delta_j(s_N)=\Delta.
\end{equation}

Indeed, \eqref{jN} implies that
\begin{equation}
    \label{AI-claim}
    \lim_{N\to\infty}\frac1N\sum_{j=0}^{N-1}|\Delta_j(s_N)-\Delta|=0,
\end{equation}
which  proves \eqref{Cesaro}.
\arxiv{
To verify \eqref{AI-claim}, let
$
K=\sup_{N,j}\Delta_j(s_N)<\infty
$.
Given $\varepsilon>0$, choose $M$ such that
\[
|\Delta_j(s_N)-\Delta|<\varepsilon,
\qquad N,j>M.
\]
Then, for $N>M$,
\[
\frac1N\sum_{j=0}^{N-1}|\Delta_j(s_N)-\Delta|
\le
\frac{M(K+\Delta)}{N}
+\frac{N-M}{N} \, \varepsilon,
\]
which  yields
\[
\limsup_{N\to\infty}
\frac1N\sum_{j=0}^{N-1}|\Delta_j(s_N)-\Delta|
\le \varepsilon.
\]
}

To prove \eqref{jN}, we define all random variables $\zeta_j(s_N)$ and $\zeta_j$ on the same probability space,
which, by  Skorohod's representation theorem,
  allows us to assume $\zeta_j(s_N)\to \zeta_j$ for all $j\ge 1$ almost surely as $N\to\infty$. We will first prove the %
double limit
\begin{equation}\label{Plim}
\lim_{\substack{j\to\infty\\ N\to\infty}}   %
\,R_j(s_N)= 1+\sum_{k=1}^\infty\,\prod_{r=1}^k\,\zeta_r=:R \quad \mbox{ in probability.}
\end{equation}
In view of \eqref{W_N'} we see that $\rho(s_N)\le \rho(0)=1$. Consequently, recall notation \eqref{p-metric}, for $p\in(0,2w)$ %
and $s_N\geq 0$,  we have $\|\zeta(s_N)\|_p\leq \|\zeta\|_p=:q<1$.

Given $\eps>0$, let $M_0$ be such that %
$q^{M_0}/(1-q)<\eps$.
Clearly, $R_{M_0}(s_N)\to R_{M_0}$ {   almost surely. } %
Since $\sup_N\,R_{M_0}(s_N)\le R_{M_0}$,  by the Lebesgue dominated convergence
it follows that there exists $N_0$ such that   for all %
{  $N>N_0$ } we have
$\|R_{M_0}(s_N)-R_{M_0}\|_p<\eps$.
Then, by the triangle inequality and independence, for all %
{  $j,N>M_0\vee N_0 $ }
we have
\begin{multline*}
 \|R_{j}(s_N) - R\|_p
 \leq \|R_{M_0}(s_N) - R_{M_0}\|_p
 + \sum_{k=M_0}^j \|\zeta(s_N)\|_p^k
 + \sum_{k=M_0}^\infty \|\zeta\|_p^k
 \\ \le \epsilon
+2\tfrac{q^{M_0}}{1-q}\le 3\epsilon.
\end{multline*}

Next, we introduce function $f$ defined by $f(x)=x^{2w}-(x-1)^{2w}$, $x\ge 1$, and note that
\begin{equation}
     \label{f-bound'}
     0\leq f(x)\leq (1+2w)x^{2w-1}.
   \end{equation}
\arxiv{Indeed, if $ 0< 2w\leq 1$ and $x\geq 1$ then  by triangle inequality,
$x^{2w}=((x-1)+1)^{2w}\leq (x-1)^{2w}+1$.
If $2w>1$, then $f(x)=2w\int_{x-1}^x t^{2w-1}\d t\leq 2w x^{2w-1}$.
}
   Consequently, for $0<2w\le 1$ we have $f(x)\le 2$, $x\ge 1$. Thus
   $$\sup_{j,N}\,\Delta_j(s_N)\le \mathbb E\,\sup_{j,N}\,f(R_j(s_N))\le 2.$$

For $2w>1$ choose $\delta>0$ such that $p:=(2w-1)(1+\delta)\in[1,2w)$. Then again  we have $q:=\|\zeta\|_p<1$.  In view of \eqref{f-bound'}, the triangle inequality for $\|\cdot  \|_p$  and independence  we get
\begin{multline*}
\sup_{N,j} \Delta_j(s_N)\le  \sup_{N,j} \EE\left[ \left(f(R_j(s_N))\right)^{1+\delta}\right]
\\ \leq {(1+2w)^{1+\delta}\,\sup_{N,j}\,} \left(\sum_{k=1}^j(\EE [\zeta_1(s_N)^{p}])^{k/p}\right)^{p}\\
\leq {(1+2w)^{1+\delta}\,\sup_{j}\,} \left(\sum_{k=1}^jq^k\right)^{p} \leq \tfrac{{(1+2w)^{1+\delta}} }{(1-q)^p}<\infty.
\end{multline*}
Thus the first relation in \eqref{jN} holds and we have the  moment estimate:
    \begin{equation}\label{uuii}
     C:=\sup_{N,j} \EE\left[ \left(f(R_j(s_N))\right)^{1+\delta}\right]<\infty.
    \end{equation}

To prove the second relation in \eqref{jN}, given \(A>0\), let \(\ell_A\) be a continuous piecewise linear function satisfying \(\ell_A(x)=x\) for \(x\in[-A,A]\) and \(\ell_A(x)=0\) for \(x\notin[-A-1,A+1]\), and set
\(
g_A=\ell_A\circ f.
\)
By weak convergence, recall \eqref{Plim},
\begin{equation}\label{ljN}
    \lim_{\substack{j\to\infty\\ N\to\infty}}\EE[g_A(R_j(s_N))]=\EE[g_A(R)].
\end{equation}

Note that because $R-1 \simeq \mathrm{Beta}_{II}(\alpha-w,2w)$, see \cite[Example 9]{chamayou1991explicit}, we have $\EE[f(R)]=\Delta$, as given by \eqref{Delta}.

Given $\varepsilon>0$, choose $A>0$ sufficiently large so that
$$
\max\{\tfrac C{A^\delta}, \EE[f(R) 1_{f(R)>A}]\}<\eps.
$$
Using H\"older's inequality and then applying Markov's inequality to the $(1+\delta)$-moment estimate in \eqref{uuii}, we obtain
  \begin{multline*}
        \sup_{N,j}\EE \left[f(R_j(s_N)) 1_{f(R_j(s_N))> A}\right]
        \\ \leq
    \sup_{N,j} \EE[ (f(R_j(s_N)))^{1+\delta}]^{1/(1+\delta)}\,\mathbb P[f(R_j(s_N))> A]^{\delta/(1+\delta)}
     \leq \tfrac{C}{A^\delta}<\eps.
  \end{multline*}

Since $0\le f(r)-g_A(r)\le f(r)\mathbf 1_{{f(r)>A}}$,
it follows that
\begin{multline*}
\left|\EE \bigl[f(R_j(s_N))\bigr]-\EE[f(R)]\right|
\le
\left|\EE\bigl[f(R_j(s_N))-g_A(R_j(s_N))\bigr]\right|
\\ +\left|\EE\bigl[f(R)-g_A(R)\bigr]\right|
+\left|\EE\bigl[g_A(R_j(s_N))\bigr]-\EE\bigl[g_A(R)\bigr]\right|
\\ \le
\EE\left[f(R_j(s_N))\mathbf 1_{{f(R_j(s_N))>A}}\right]
+\EE\left[f(R)\mathbf 1_{{f(R)>A}}\right]
\\ +\left|\EE\bigl[g_A(R_j(s_N))\bigr]-\EE\bigl[g_A(R)\bigr]\right|
\le
2\varepsilon
+\left|\EE\bigl[g_A(R_j(s_N))\bigr]-\EE\bigl[g_A(R)\bigr]\right|.
\end{multline*}
Thus, in view of \eqref{ljN}, the second relation of \eqref{jN} follows.
\end{proof}

\section{Barraquand's integral representation}\label{Sec:BarInt}
Barraquand \cite[Thm~1.11]{Barraquand-2024-integral} used Whittaker functions to determine  an interesting  $d$-fold integral representation for the $d$-point Laplace transform of the %
 increments of vector
$\vv L$ for positive $u,v$, and described the conditions required for its analytic continuation to negative $u,v$.
 In particular, Barraquand obtained the following identity, in which the multivariate integrals \eqref{II-1} appearing in the numerator and denominator of \eqref{LG-xx} are replaced by univariate integrals.
\begin{proposition}[{\cite[formula (1.12)]{Barraquand-2024-integral}}] \label{CL2} If $\alpha, u,v>0$, $-v<t<u$  and $t+\alpha>0$,
then %
  \begin{multline}
  \label{LG}
  \EE [ e^{- 2 t \vv L_N}]
  \\ =\tfrac{1}{\mathcal{Z}_N} \int_{\i \RR} \left(\Gamma(\alpha+t+z)\Gamma(\alpha+t-z)\right)^N
  \tfrac{\Gamma(u-t+z)\Gamma(u-t-z)\Gamma(v+t+z)\Gamma(v+t-z)}{ \Gamma(2z)\Gamma(-2z)} \d z . %
\end{multline}
\end{proposition}
(Here $\mathcal{Z}_N$ is the normalizing constant given by the same integral with $t=0$.)

The integral on the right-hand side may include an arbitrary nonzero multiplicative factor, which is subsequently canceled by the normalizing
constant $\mathcal{Z}_N$. We will work with a version of \eqref{LG} that incorporates such a multiplicative factor in order to match \eqref{II-1}.
   We write
 \begin{equation*}\label{LG-J}
    \EE [ e^{- 2 t \vv L_N}]=\tfrac{\II_N(\alpha+t,u-t,v+t)}{\II_N (\alpha,u,v)},
 \end{equation*}
where %
     \begin{equation}\label{JJ-1}
   \II_N (\alpha,u,v)=\tfrac{1}{
    2\pi \Gamma(u+v)} \int_{0}^\infty \left|\Gamma(\alpha+\i x)\right|^{2N} \tfrac{\Gamma(v+\i x)\Gamma(v-\i x)\Gamma(u +\i x )\Gamma(u -\i x )}{|\Gamma(2\i x)|^2}\d x.
 \end{equation}
 (Compare  \cite[Section 4.5]{Barraquand-2024-integral}.)

With this normalization, Proposition \ref{CL2} {  is a consequence of the following representation of $\mathcal Z_N$.}
  \begin{proposition}\label{Lem:I=J}
     Assume     $\alpha>0$, $u>0$ and  $\re(v)>0$. Then for any   $N=1,2,\dots$,
     \begin{equation}
       \label{I_N=J_N}
      \GG_N (\alpha,u,v)= \II_N (\alpha,u,v),
     \end{equation}
where $ \GG_N (\alpha,u,v)$ is given by \eqref{II-1}.  \end{proposition}
  \begin{proof}
    This is essentially   the case $d=1$ and $t_1=0$ of Theorem \ref{BarraquandThm1.11}.
    (Theorem \ref{BarraquandThm1.11}  states \eqref{GGII}  for real parameters, but the proof uses complex $v$, see \eqref{v-contstr}.)
  \end{proof}
 In this paper, we present a more elementary argument, independent of Whittaker functions,  that re-derives
Barraquand's  integral  formula for positive $u,v$ (see Corollary~\ref{Cor:Barraquand}).
The argument leads to a similar formula, see Lemma~\ref{Thm:half-space}, for the stationary measures of the log-gamma polymer in the half-space.

However, we begin with the discussion of analytic  continuation  of  $ \II_N (\alpha,u,v)$.
Although the existence of an analytic continuation is apparent from \cite{Barraquand-2024-integral}, in Section~\ref{Sec:AltProof} we require explicit formulas to establish the full phase diagram by analytic methods.
   \subsection{Extension of Proposition \ref{Lem:I=J} }
It is known that \eqref{II-1}  defines a function that is analytic in $u,v$.
 This result follows from \cite[Propositions 3.15 and 3.16]{barraquand2024stationary}. (See also Proposition~\ref{Cl:G-an}.)

\begin{lemma}\label{Lem:I-anal}
Fix $\alpha>0$ and  $N\geq 1$.
Then, for fixed $v\in \{z\in\mathbb C:\;\re(z)>-\alpha\}$,  \eqref{II-1}  defines an analytic  function of    variable $u\in \{z\in\mathbb C:\;\re(z)>-\alpha\}$; similarly,  for fixed $u\in \{z\in\mathbb C:\;\re(z)>-\alpha\}$,
\eqref{II-1}  defines an analytic  function of    variable   $v\in \{z\in\mathbb C:\;\re(z)>-\alpha\}$.
\end{lemma}

We remark that by Hartogs' theorem \cite[Section 4.2]{Vladimirov1966},  a function analytic in each variable is    a holomorphic function of  variables $u,v$ on the product of the half planes. In particular, $u\mapsto \GG_N(\alpha,u,u)$ is  holomorphic on the half-plane ${ \{z\in\mathbb C:\;\re(z)>-\alpha\}}$.

 { Our next goal is to present an explicit analytic continuation of  \(\II_N(\alpha,u,v)\), thereby extending the %
 representation \eqref{I_N=J_N} to a larger range of the parameters $u$ and $v$.}
   The key ingredient is the following integrability result, accompanied by an explicit estimate that will allow us to invoke the dominated convergence theorem in the proof of Proposition~\ref{Prop2}.

 \begin{lemma}\label{Lem:J-int}
   Assume {  $N\in\ZZ_{\geq 1}$},   $\alpha>0$ and $u,v\in\CC$ are such that $\re(u)\not\in\ZZ_{\leq 0}$ and $\re(v)\not\in\ZZ_{\leq 0}$.
   Then the integral
     \begin{equation*}\label{integral}
     \int_{0}^\infty \left|\Gamma(\alpha+\i x)\right|^{2N} \tfrac{\Gamma(v+\i x)\Gamma(v-\i x)\Gamma(u +\i x )\Gamma(u -\i x )}{|\Gamma(2\i x)|^2}\d x.
     \end{equation*}
       converges absolutely.

Furthermore, for   intervals  $[a_1,b_1],[a_2,b_2]\subset \RR\setminus\ZZ_{\leq 0}$, and $A\ge 1$, there are constants  $C,\theta>0$ such that
\begin{equation}\label{4DCT}
     |\Gamma(\alpha+\i x)|^{2N} \tfrac{\left|\Gamma(v+\i x)\Gamma(v-\i x)\Gamma(u +\i x )\Gamma(u -\i x )\right|}{|\Gamma(2\i x)|^2}\le     C \left(1+x^{\theta}\right) e^{- N\pi x},\quad x\ge 0,
  \end{equation}
for all   complex $u,v$ such that $\re(u)\in[a_1,b_1]$, $\re(v)\in[a_2,b_2]$ and $|\im(u)|,|\im(v)|\leq A$.
 \end{lemma}

\begin{proof} {The absolute convergence of the integral is an immediate consequence of the bound \eqref{4DCT}.
To prove that this bound holds true we} derive bounds of the form
$c \left(1+x^{a}\right) e^{- b x}$ for each factor { on the left-hand side of \eqref{4DCT}} and then combine them to obtain { the final inequality}.

The simplest factor is
\begin{equation}\label{Sinh}
 \tfrac{1}{|\Gamma(2 \i x)|^2}=\tfrac{2}{\pi} x \sinh(2\pi x) \leq \tfrac{2}{\pi} x e^{2\pi x}\leq C_1(1+x^{\theta_1})e^{2\pi x}.
\end{equation}
(Here, $C_1=2/\pi$ and $\theta_1=1$.)

For the remaining factors we use inequality \cite[\href{http://dlmf.nist.gov/5.6.E9}{(5.6.9)}]{NIST:DLMF} which says that for $z=a+\i b$ with $a\geq 0$ we have
\begin{equation}\label{(5.6.9)NIST}
  |\Gamma(z)|^2\leq 2\pi |z|^{2a-1}e^{-\pi |b|} \exp(\tfrac{1}{3|z|}).
\end{equation}
In case $\re(z)<0$ we refer to $\Gamma(z)=\Gamma(z+1)/z$ to apply \eqref{(5.6.9)NIST} to a  variable with positive real part.

\bigskip
In view of \eqref{(5.6.9)NIST} we have
\begin{multline*}
  |\Gamma(\alpha+\i x)|^{2N}\le (2\pi)^N |\alpha+\i x|^{(2\alpha-1)N}e^{-\pi N x}\,\exp(\tfrac N{3\sqrt{\alpha^2+x^2}})
  \\ \le \tfrac{(2\pi)^N\exp(\frac N{3\alpha})}{\alpha^N} (\alpha^2+x^2)^{\alpha N}\,e^{-\pi N x}.
\end{multline*}

The $c_r$-inequality  implies that for $c_r=2^{r-1}\vee 1\le 2^r$
\begin{equation}\label{cr}
(a+w)^r\le c_r(a^r+w^r)\le c_r(a^r\vee 1)(1+w^r)\le2^r(a\vee 1)^r(1+w^r)
\end{equation}
for positive $a,r$ and $w\geq 0$. Consequently,
\begin{equation}\label{G-alpha}
|\Gamma(\alpha+\i x)|^{2N}\le  C_2^N (1+x^{\theta_2})e^{-\pi N x},
\end{equation}
with $C_2:=C_2^{(\alpha)}=\tfrac{2^{1+\alpha}\pi\exp(\frac 1{3\alpha})}{\alpha}  (\alpha^2\vee 1)^{\alpha }$ and $\theta_2=2\alpha N$.

\bigskip
Next, we consider $ |\Gamma(u\pm \i x)|^2$.
In the argument we only consider $\re(u)<0$. (We omit the proof for $\re(u)>0$, which is similar to the argument for \eqref{G-alpha}.)
We write $u=a+\i b$ and, without loss of generality, we consider  $[a_1,b_1]= [-K-1+\delta,-K-\delta]$, where
$K\in\ZZ_{\geq 0}$ is such that $$1-\delta\ge \tilde\alpha:=a+K+1\ge\delta\in(0,\tfrac12)\quad\mbox{and}\quad a+j<-\delta\quad\mbox{for}\;j=0,1,\ldots,K.$$  We seek a bound of the form  \eqref{4DCT} that holds for all $x\geq 0$   with constants $C,\theta$  that depend only on parameters $ K,A, \delta$.
Clearly, (note that $|a+j|>\delta$, $j=0,1,\ldots,K$)
$$
    |\Gamma(u\pm \i x)|^2 =  \tfrac{ |\Gamma(u+K+1\pm \i x)|^2}{\prod_{j=0}^{K} |u+j\pm \i x|^2}\le    \tfrac{ |\Gamma(u+K+1\pm \i x)|^2}{\prod_{j=0}^{K} (a+j)^2}\leq   \tfrac{1}{\delta^{2K+2}}    |\Gamma(\tilde \alpha+\i(b\pm  x))|^2.
$$
Thus, see \eqref{G-alpha}  { with  $\alpha=\tilde\alpha$, and   $N=1$}, we obtain
\begin{equation*}\label{*tmp*}
  |\Gamma(\tilde\alpha+\i(b\pm  x))|^2\le C_2^{(\tilde\alpha)}(1+|b\pm x|^{\theta})e^{-\pi|b\pm x|}.
\end{equation*}
where $\theta=2\tilde\alpha$.
Recalling that $\delta\le \tilde\alpha<1$  we see that
$$
C_2^{(\tilde\alpha)}\le \tfrac{4\pi}{\delta}\exp(\tfrac{1}{3\delta}).
$$
Since $|b\pm x|\ge |x|-|b|\ge |x|-A$ we have $e^{-\pi |b\pm x|}\le e^{\pi A}\,e^{-\pi x}$.

Since  $A\geq 1$ and $\theta=2\tilde\alpha\le 2$  we have $1+(2A)^{\theta}\le  5A^2$ and $1+x^{\theta}\le  2(1+x^2)$. Thus, in view of \eqref{cr} we obtain
$$
1+|b\pm x|^{\theta}\le 1+2^{\theta}(1\vee|b|)^{\theta}(1+x^{\theta})\le (1+(2A)^{\theta})(1+x^{\theta}) \leq 10A^2(1+x^2).
$$

\arxiv{
Indeed,
$1+x^{\theta}=(1+x^\theta) 1_{x\geq 1}+(1+x^\theta)1_{x\leq 1}\leq (1+x^2) 1_{x\geq 1}+2\cdot  1_{x\leq 1}
\leq 2(1+x^2) 1_{x\geq 1}+ 2(1+x^2) 1_{x\leq 1}=2(1+x^2)$
and $1+(2A)^{\theta}\le A^2+(2A)^2= 5 A^2$.
}

Finally, we conclude that
\begin{equation}\label{G-u}
|\Gamma(u\pm \i x)|^2\le C_3(1+x^{\theta_3})e^{-\pi x}
\end{equation}
with $C_3=\tfrac{40 A^2\pi}{\delta^{2K+3}}\,e^{\pi A+\frac{1}{3\delta}}$ and $\theta_3=2$.

\bigskip
The same proof  applies to $v\in\CC$ such that $\re (v)\in[a_2,b_2]=[-J-1+\delta,-J-\delta]\subset(-\infty,0)$ for some
$J\in\ZZ_{\geq 0}$,  $\delta\in (0,1/2)$ and $|\im(v)|\leq A$ and  establishes   the  following bound
\begin{equation}
    \label{G-v}|\Gamma(v\pm \i x)|^2\leq C_4  (1+x^{\theta_4})e^{-\pi x},
\end{equation}
where $\theta_4=2$ and $C_4=C_4(J,A,\delta)$ does not depend on $x\geq 0 $.

\bigskip
Denote $\theta_{\max}=\max\{\theta_1,\theta_2,\theta_3,\theta_4\}=\max\{\theta_2,2\}$ and observe that
 $$
 (1+x^{\theta_1})(1+x^{\theta_2})(1+x^{\theta_3})(1+x^{\theta_4})\leq 16 +\left(1+x^{\theta_{\max}}\right)^4
 \leq 24 \left(1+x^{4 \theta_{\max}}\right).
 $$
 Indeed, for $x\in[0,1]$ the left-hand side is at most $2^4=16$ and for $x>1$ it is bounded by $\left(1+x^{\theta_{\max}}\right)^4$. Then we use $(1+a)^2\leq 2(1+a^2)$ twice  to get  $(1+a)^4\leq 8(1+a^4)$. So for all $x\geq 0$, the  left-hand side  is bounded by $ 16+8(1+x^{4\theta_{\max}}) =24+8 x^{4\theta_{\max}} \leq 24(1+x^{4\theta_{\max}})$. Thus,  combining the bounds \eqref{Sinh}, \eqref{G-alpha}, \eqref{G-u}, \eqref{G-v}
 we obtain \eqref{4DCT} with $C=24C_1 C_2 C_3 C_4$ and $\theta=4\theta_{\max}$.
\end{proof}
\arxiv{For completeness, we give  the omitted argument for the bound \eqref{G-v} when $\re(v)>0$. As previously, we write $v=a+\i b$ with $|b|<A$ and $a\in[a_2,b_2]$, where $a_2=\delta>0$.
Then, using \eqref{(5.6.9)NIST} with $z=a+1+\i(b\pm x)$ we obtain
$$
|\Gamma(v\pm \i x)|^2=\frac{|\Gamma(a+1+\i(b\pm x))|^2}{|a\pm \i x|^2}\le  \frac{|\Gamma(a+1+\i(b\pm x))|^2}{\delta^2}
\leq \frac{2\pi \sqrt[3]{e}}{\delta^2} e^{\pi A} \; |z|^{2 a+1} e^{-\pi x}.
$$
(Here we used $|z|>a+1>1$ and $|b|<A$.) Since $2a+1>1$, by the triangle inequality
\begin{multline*}
 |z|^{2a+1} %
\leq (b_2+1+A+|x|)^{2b_2+1}
\leq 2^{2b_2+1} \left((b_2+1+A)^{2b_2+1} + |x|^{2b_2+1} \right)
\\ \leq  2^{2b_2+1}  (b_2+1+A)^{2b_2+1} (1+|x|^{2b_2+1}) .
\end{multline*}
This proves \eqref{G-v} with $\theta_4=2b_2+1$ and constant $C_4=C_4(a_2,b_2,A)=\ {\pi }{a_2^{-2}} e^{1/3+\pi A}4^{b_2+1}  (1+b_2+A)^{2b_2+1}$.
}

\begin{remark}\label{Rem:J-int+}
    The bound \eqref{4DCT} remains valid if one of the parameters is an
    integer   in $\ZZ_{\leq 0}$ while the real part of the other parameter is not in $\ZZ_{\leq 0}$. For example, suppose $u=-m\in\ZZ_{\leq 0}$ while $\re(v)\not\in\ZZ_{\leq 0}$. Then,  in the proof of \eqref{4DCT},  we replace the two separate
    bounds \eqref{Sinh} and \eqref{G-u} with a single inequality:
\begin{equation}\label{sinh+G-u}
\tfrac{|\Gamma(-m+\i x)|^2}{|\Gamma(2\i x)|^2}=  \tfrac{\pi}{x\sinh (\pi x)} \tfrac{1}{\prod_{k=1}^m(k^2+x^2)}\, \tfrac{2x \sinh(2\pi x)}{\pi}
=2\tfrac{e^{\pi x}+e^{-\pi x}}{\prod_{k=1}^m(k^2+x^2)}\leq \tfrac{4}{m!^2}e^{\pi x}.
\end{equation}
Somewhat more generally, for a fixed $m\in\ZZ_{\leq 0}$, we need for bound \eqref{4DCT}  to hold for $u=-m+\eps$ uniformly in  $\eps\in[-1/2,1/2]$.
This can be seen as follows.
Using the recursion $z\Gamma(z)=\Gamma(1+z)$ we obtain %
\begin{multline}
    \label{u=-m bound}
\tfrac{|\Gamma(-m+\eps+\i x)|^2}{{|\Gamma(2\i x)|^2}}
\\= \tfrac{2}{\pi}\cdot \tfrac{|\Gamma(1+\eps+\i x)|^2}{\prod_{k=1}^m ((k-\eps)^2+x^2)} \cdot \tfrac{x \sinh(2\pi x)}{\eps^2+x^2}\leq \tfrac{2^{  2m+1}}{\pi}  |\Gamma(1+\eps+\i x)|^2
\,\tfrac{x \sinh(2\pi x)}{\eps^2+x^2}.
\end{multline}
\arxiv{$(k-\eps)^2\geq (k-1/2)^2\geq 1/4 $ for $k\geq 1$. (The actual constant is of no importance.)
}

Then, {  from \eqref{G-alpha} with $\alpha=1+\eps\in[1/2,3/2]$ and $N=1$ we get
$|\Gamma(1+\eps+\i x)|^2\leq C_2^{(1+\eps)} (1+x^{3})e^{-\pi x}$}.
In view of the elementary inequality $1-e^{-x}\le x$ the
 last factor in \eqref{u=-m bound} is bounded by
$$
\tfrac{x \sinh(2\pi x)}{\eps^2+x^2}\le e^{2\pi x}\tfrac{1-e^{-4\pi x}}{2x}\le 2\pi e^{2\pi x}.
$$
Thus \eqref{4DCT} holds uniformly over (real)
$u \in [-m-1/2,-m+1/2]$  and over all complex $v$ in a compact set  $\re(v)\in[a_2,b_2]\subset \RR\setminus \ZZ_{\le 0}$, $|\im(v)|\leq A$.
\end{remark}
By Lemma \ref{Lem:I-anal}, the left-hand side of \eqref{I_N=J_N} is well defined for all $u,v>-\alpha$. It is therefore natural to determine the appropriate form of {  its} right-hand side {  in the same domain of $u$ and $v$}.

The following result does not cover all cases of interest; in particular, it excludes the {\em coexistence line} with $u=v<0$. However, it applies to a broader range of parameters than Proposition~\ref{Lem:I=J}. In fact, formulas for all parameters $u,v>-\alpha$, including the case $u=v<0$, can be obtained by taking suitable limits of \eqref{I_N-x}. The resulting expressions are rather involved, and their precise form depends on the values of $u$ and $v$; therefore, we defer them to Section~\ref{Sec:Lims}.

Recall   the   Pochhammer symbol $ (a)_n=a(a+1)\dots(a+n-1)$.

\begin{proposition}\label{Prop2}
Assume  %
$\alpha,u,v\in\RR$ are such that $\alpha>0$, $\alpha+u>0$
  and  $\alpha+v>0$.
 Furthermore, we assume that
 $u,v\not\in \ZZ_{\leq 0}$ %
  and if $u<0$ and $v<0$, we assume that  $v-u\not\in\ZZ$.
Then %
  \begin{multline}\label{I_N-x}
     \GG_N(\alpha ,u ,v )
     =\tfrac{1}{2\pi \Gamma(u+v)}\int_{0}^\infty \left|\Gamma(\alpha +\i x)\right|^{2N}\,\tfrac{|\Gamma(v+\i x)|^2|\Gamma(u +\i x )|^2}{|\Gamma(2\i x)|^2}\d x \\ +\mathbb{D}_N(\alpha,u,v)+ \mathbb{D}_N(\alpha,v,u),
  \end{multline}
where
\begin{equation}\label{D(u,v)}
   \mathbb{D}_N(\alpha,u,v)=\sum_{j\in \ZZ\cap[0,-v )} m_j(u,v) \Gamma(\alpha+v+j)^{N}\Gamma(\alpha-v-j)^{N}\nonumber
\end{equation}
   with
    \begin{equation}
    \label{their-atoms}
    m_j(u,v)
    =\tfrac{-2(v+j )}{j!}\cdot \tfrac{ \Gamma(u-v)}{\Gamma (1-2 v ) }\cdot\tfrac{ (v+u)_j(2 v )_j }{ (v-u+1)_j},\; j\in \ZZ\cap[0,-v ).
  \end{equation}
  \end{proposition}
   This result gives the form of analytic continuation for Proposition \ref{Lem:I=J} as  indicated in \cite{Barraquand-2024-integral}.
 A rigorous argument is more involved than it may initially appear, so for completeness we include a proof of Proposition~\ref{Prop2} in Section~\ref{Sec:AC}.
 We thank Alexey Kuznetsov for his substantial contributions to this proof.

\subsection{Multivariate Laplace transform}\label{Sec:MLT}
For  $\vv t=(t_1,\dots,t_N)\in\RR^N$, the multi-point Laplace transform of the increments of $\vv L$ is
\begin{equation}\label{Lap-0}
     \mathbb E\left[e^{-2\sum_{i=1}^N\,t_i(\vv L_i-\vv L_{i-1})}\right]=\tfrac{\GG_{\vv t}(\alpha,u,v)}{\GG_{N}(\alpha,u,v)},
\end{equation}
 where $\GG_{\vv t}(\alpha,u,v)=\GG_{t_1,\dots,t_N}(\alpha,u,v)$ is
 \begin{multline}\label{GGt}
 \GG_{t_1,\dots,t_N}(\alpha,u,v)
 \\=\int\limits_{(0,\infty)^{2 N}}
  \tfrac{ \left(\prod_{k=1}^N\,x_k^{\alpha+v+2t_{k}-1}\right)\,\left(\prod_{j=1}^N\,y_j\right)^{\alpha+u-1} }{ \left(\sum_{k=0}^{N-1} x_1\dots x_{k} y_{k+2}\dots y_N\right)^{u+v}} e^{-\sum_{j=1}^N\,(x_j+ y_j)} \d \vv x \d \vv y,
 \end{multline}
 and the normalizing constant $\GG_N(\alpha,u,v)$  is given by \eqref{II-1}.

A multivariate version of Corollary~\ref{per1} takes the following form.
\begin{proposition}\label{Prop:Perp-more}
    If $|u|<\alpha$, $v>-\alpha$ and $\alpha-u+2t_j>0$, $j=1,\dots N-1$, and $\alpha+v+2t_N>0$, then
\begin{multline}\label{pert-MV}
     \GG_{t_1,\dots,t_N}(\alpha,u,v)
     =\Gamma(\alpha+v+2t_N) \Gamma(\alpha+u)^N \\ \cdot  \Biggl(\prod_{k=1}^{N-1}\Gamma(\alpha-u+2t_k)\Biggr)\, \EE [(1+\zeta_1+\zeta_1\zeta_2+\dots+\zeta_1\dots \zeta_{N-1})^{-u-v}]
     \end{multline}
with independent $\zeta_k\simeq \mathrm{Beta}_{II}(\alpha+u,\alpha-u+2t_{N-k})$, $k=1,\dots,N-1$.

    If $|v|<\alpha$, $u>-\alpha$ and $\alpha+v+2t_j>0$, $j=1,\dots,N$,  then
   \begin{multline}\label{pert-MU}
     \GG_{t_1,\dots,t_N}(\alpha,u,v)
     =\Gamma(\alpha+u) \Gamma(\alpha-v)^{N-1} \\ \cdot   \Biggl(\prod_{k=1}^N \Gamma(\alpha+v+2t_k)\Biggr)\, \EE [(1+\xi_1+\xi_1\xi_2+\dots+\xi_1\dots\xi_{N-1})^{-u-v}]
     \end{multline}
with independent $\xi_k\simeq \mathrm{Beta}_{II}(\alpha+v+2 t_k,\alpha-v)$, $k=1,\dots,N-1$.
\end{proposition}

\begin{proof}
  Factoring out $x_1\dots x_{N-1}$ from the denominator in   \eqref{GGt} we obtain
\begin{multline*}
   \GG_{t_1,\dots,t_N}(\alpha,u,v)
   \\ =  \int_{(0,\infty)^{2 N}}
  \tfrac{ \left(\prod_{k=1}^N\,x_k^{\alpha+v+2t_{k}-1}\right)\,\left(\prod_{j=1}^N\,y_j\right)^{\alpha+u-1} }{ \left(\sum_{k=0}^{N-1} x_1\dots x_{k} y_{k+2}\dots y_N\right)^{u+v}} e^{-\sum_{j=1}^N\,(x_j+ y_j)} \d \vv x \d \vv y
  \\=   \int_{(0,\infty)^{2 N}}
  \tfrac{ x_N^{\alpha+v+t_N-1}\left(\prod_{k=1}^{N-1}\,x_k^{\alpha-u+2t_{k}-1}\right)\,\left(\prod_{j=1}^N\,y_j\right)^{\alpha+u-1} }{ \left(1+\frac{y_N}{x_{N-1}}+\frac{y_{N-1}}{x_{N-2}}\frac{y_N}{x_{N-1}}+\dots+\prod_{j=1}^{N-1}\,\frac{y_{j+1}}{x_j}\right)^{u+v}} e^{-\sum_{j=1}^N\,(x_j+ y_j)} \d \vv x \d \vv y.
  \end{multline*}
 Integrating out $x_N$ and $y_1$ we get
  \begin{multline*}
      \tfrac{\GG_{t_1,\dots,t_N}(\alpha,u,v)}{\Gamma(\alpha+v+2t_N)\Gamma(\alpha+u) }\\
      =
  \int_{(0,\infty)^{2 N-2}}
  \tfrac{ \left(\prod_{k=1}^{N-1}\,x_k^{\alpha-u+2t_{k}-1}\right)\,\left(\prod_{j=2}^N\,y_j\right)^{\alpha+u-1} }{ \left(1+\frac{y_N}{x_{N-1}}+\frac{y_{N-1}}{x_{N-2}}\frac{y_N}{x_{N-1}}+\dots+\prod_{j=1}^{N-1}\,\frac{y_{j+1}}{x_j}\right)^{u+v}} e^{-\sum_{j=2}^N\,(x_{j-1}+ y_j)} \d \vv x \d \vv y
  \\=\int_{(0,\infty)^{2 N-2}}
  \tfrac{ \left(\prod_{k=1}^{N-1}\,x_k^{2\alpha +2t_{k}-1}\right)\,\left(\prod_{j=1}^{N-1}\,z_j\right)^{\alpha+u-1} }{ \left(1+z_1+z_1z_2+\dots+z_1\dots z_{N-1}\right)^{u+v}} e^{-\sum_{j=1}^{N-1}x_j(1+z_{N-j})}\, \d \vv x \d \vv z
\end{multline*}
where in the penultimate integral we performed a change of variables $z_j=\tfrac{y_{N-j+1}}{x_{N-j}}$, $j=1,\dots, N-1$. Finally, integrating out $x_j$, $j=1,\dots, N-1$ we obtain
\begin{multline*}
  \GG_{t_1,\dots,t_N}(\alpha,u,v)
 ={\Gamma(\alpha+v+2t_N)\Gamma(\alpha+u)\,\prod_{k=1}^{N-1}\Gamma(2\alpha+2t_k)}
 \\ \cdot  \int_{(0,\infty)^{N-1}}\,\Big(\prod_{k=1}^{N-1}\,\tfrac{z_k^{\alpha+u-1}}{(1+z_k)^{2\alpha+2 t_{N-k}}}\Big)\,\frac{ \d \vv z}{ \left(1+z_1+z_1z_2+\ldots +z_1z_2\ldots z_{N-1}\right)^{u+v}}.
\end{multline*}
\arxiv{For completeness, we prove of the second part:

 Factoring out $y_2\dots y_{N}$ from the denominator in   \eqref{GGt} we obtain
\begin{multline*}
   \GG_{t_1,\dots,t_N}(\alpha,u,v)=  \int_{(0,\infty)^{2 N}}
  \frac{ \left(\prod_{k=1}^N\,x_k^{\alpha+v+2t_{k}-1}\right)\,\left(\prod_{j=1}^N\,y_j\right)^{\alpha+u-1} }{ \left(\sum_{k=0}^{N-1} x_1\dots x_{k} y_{k+2}\dots y_N\right)^{u+v}} e^{-\sum_{j=1}^N\,(x_j+ y_j)} \d \vv x \,\d \vv y
   \\=   \int_{(0,\infty)^{2 N}}
   \frac{ \left(\prod_{k=1}^N\,x_k^{\alpha+v+2t_{k}-1}\right)\,y_1^{\alpha+u-1}\left(\prod_{j=2}^N\,y_j\right)^{\alpha-v-1} }
   { \left(1+\frac{x_1}{y_2}+\frac{x_1}{y_{2}}\frac{x_2}{y_3}+\dots+\prod_{j=1}^{N-1}\,\frac{x_{j}}{y_{j+1}}\right)^{u+v}} e^{-\sum_{j=1}^N\,(x_j+ y_j)} \d \vv x \,\d \vv y
   \\
   \mbox{ (integrating out $x_N$ and $y_1$)}
  \\
   =\Gamma(\alpha+v+2t_N)\Gamma(\alpha+u)
    \int_{(0,\infty)^{2 N-2}}
  \frac{ \left(\prod_{k=1}^{N-1}\,x_k^{\alpha+v+2t_{k}-1}\right)\,\left(\prod_{j=2}^N\,y_j\right)^{\alpha-v-1} }{\left(1+\frac{x_1}{y_2}+\frac{x_1}{y_{2}}\frac{x_2}{y_3}+\dots+\prod_{j=1}^{N-1}\,\frac{x_{j}}{y_{j+1}}\right)^{u+v}} e^{-\sum_{j=2}^N\,(x_{j-1}+ y_j)} \d \vv x \,\d \vv y
  \\
  \mbox{ (changing variable $z_j=\frac{x_j}{y_{j+1}}$, $j=1,\dots, N-1$)}
  \\=\Gamma(\alpha+v+2t_N)\Gamma(\alpha+u) \int_{(0,\infty)^{2 N-2}}
  \frac{ \left(\prod_{j=1}^{N-1}\,y_{j+1}^{2\alpha +2t_{j}-1}\right)\,\left(\prod_{k=1}^{N-1}\,z_k\right)^{\alpha+v+2 t_k-1} }{ \left(1+z_1+z_1z_2+\dots+z_1\dots z_{N-1}\right)^{u+v}} e^{-\sum_{j=1}^{N-1}y_{j+1}(1+z_{j})} \d \vv y\, \d \vv z
  \\=
  \Gamma(\alpha+v+2t_N)\Gamma(\alpha+u) \prod_{k=1}^{N-1}\Gamma(2\alpha+2t_k)
  \int_{(0,\infty)^{N-1}}\,\prod_{k=1}^{N-1}\,\tfrac{z_k^{\alpha+v+2t_k-1}}{(1+z_k)^{2\alpha+2 t_{k}}}\,\tfrac{ \d \vv z}{ \left(1+z_1+z_1z_2+\ldots +z_1z_2\ldots z_{N-1}\right)^{u+v}}.
  \\
  =\Gamma(\alpha+u)  \Gamma(\alpha-v)^{N-1} \prod_{k=1}^N \Gamma(\alpha+v+2t_k) \EE [(1+\xi_1+\xi_1\xi_2+\dots+\xi_1\dots\xi_{N-1})^{-u-v}].
\end{multline*}
}
\end{proof}

Our next goal is to present an elementary proof of \cite[Theorem 1.11]{Barraquand-2024-integral}.
 The one-dimensional version of this theorem is  stated as
 Proposition~\ref{CL2}.

Let $\vv t$ be given by
 \begin{equation}\label{tt_d}
\vv t=%
(\underbrace{t_1,\dots, t_1}_{n_1 \, \mbox{\tiny times}},\underbrace{t_2,\dots, t_2}_{n_2 \, \mbox{\tiny times}},\dots,\underbrace{t_d,\dots,t_d}_{n_d\,\mbox{\tiny times}}),
\end{equation}
where
$t_1>\dots>t_d$, $\vv n=(n_1,\dots,n_d)\in\ZZ_{\ge1}^d$, and
$n_1+\dots+n_d=N$.

To obtain a $d$-variate %
integral representation of $\GG_{\vv t}(\alpha,u,v)$, analogous to Barraquand's contour integral formula, we introduce the following multipoint analogue of $\II_N(\alpha+t,u-t,v+t)$, see  \eqref{JJ-1},
\begin{multline}\label{IIId}
\II_{\vv t}^{\vv n}(\alpha,u,v) =\II_{t_1,\dots,t_d}^{n_1,\dots,n_d}(\alpha,u,v)=
\tfrac{1}{
    (2\pi)^{d} \Gamma(u+v)
  \prod_{j=2}^d\Gamma(2t_{j-1} -2t_{j})}
\\ \cdot   \int_{(0,\infty)^d} \prod_{j=2}^d |\Gamma(t_{j-1}-t_j + \i {y_{j-1}}+\i {y_j})|^2|\Gamma( t_{j-1}-t_j + \i {y_{j-1}}-\i {y_j} )|^2
\\ \cdot  \left(\prod_{j=1}^d\tfrac{\left|\Gamma(\alpha+t_{j}+\i {y_j})\right|^{2n_j}}{|\Gamma(2\i {y_j})|^2} \right) \left|\Gamma(u-t_{1} +\i {y_1} )\right|^2\,\Gamma(v+t_{d}+\i {y_{ d }})\Gamma(v+t_{d}-\i {y_{ d }})\, \d \vv y.
\end{multline}
(We wrote this expression in a form that allows for complex values of $v$, which are needed in our proof.)
In particular,  with $d=1$ we have $\vv n=(N)$, $\vv t=(t)$ and then   $\II_N(\alpha+t,u-t,v+t)=\II_{t}^{N}(\alpha,u,v)$,  see \eqref{JJ-1}.
\begin{theorem}[Barraquand]\label{BarraquandThm1.11} %
Fix $\alpha, u,v\in\RR$, such that $u>-\min\{\alpha, v\}$.  Fix $n_1,n_2,\dots,n_d\in \ZZ_{\geq 1}$ such that  $n_1+n_2+\dots+n_d=N$ and let
$u>t_1>\dots>t_d>-\min\{\alpha, v\}$. %
Then with  $\vv t$ given by \eqref{tt_d}, we have
 \begin{equation}\label{GGII}
\GG_{\vv t}(\alpha,u,v)=\II_{t_1,\dots,t_d}^{\vv n}(\alpha,u,v).
\end{equation}
\end{theorem}
In expanded form,
 \begin{equation*}
\GG_{\underbrace{t_1,\dots, t_1}_{n_1 \mbox{\tiny -times}},\, \underbrace{t_2,\dots, t_2}_{n_2  \mbox{\tiny -times}},\;.\;.\;.\;,\underbrace{t_d,\dots,t_d}_{n_d\mbox{\tiny -times}}}(\alpha,u,v)=\II_{t_1,\dots,t_d}^{n_1,\dots,n_d}(\alpha,u,v).
\end{equation*}

\begin{corollary}[Barraquand]\label{Cor:Barraquand}
 Fix $\alpha,u,v>0$ and $0=m_0<m_1<\dots<m_d=N$. Let $u>t_1>\dots>t_d>-\min\{\alpha,v\}$. Then we have the
 following expression for the $d$-point Laplace transform
 \begin{equation}\label{GB-conc}
     \EE\left[\prod_{j=1}^d e^{- 2 t_j(\vv L_{m_j}-\vv L_{m_{j-1}})}\right]
     = \tfrac{\II_{t_1,\dots,t_d}^{n_1,\dots,n_d}(\alpha,u,v)}{\II_N(\alpha,u,v)},
 \end{equation}
 where $n_j=m_j-m_{j-1}$, $j=1,\dots,d$ and the normalizing constant  $\II_N(\alpha,u,v)$ is given by \eqref{JJ-1}.
\end{corollary}
Indeed, the left-hand side of \eqref{GB-conc}  is given by  \eqref{Lap-0} applied to  the sequence \eqref{tt_d}.
The numerator on the right-hand side of \eqref{GB-conc}  is given by \eqref{GGII}.   With  $u,v>0$   we   recover the normalizing constant $\GG_N(\alpha,u,v)$ in the denominator, see \eqref{II-1},  as $\GG_{0}(\alpha,u,v)$  from \eqref{GGII} with $d=1$, $\vv n=N$ and $\vv t=0$.

For technical reasons, we will work with  equation \eqref{GGII} for complex $v$. We need the following extension of Lemma \ref{Lem:I-anal}.
\begin{proposition}\label{Cl:G-an}  Fix $\alpha\in\RR$, $u>-\alpha$ and $t_1,\dots,t_N>-\alpha$. Let $t_*=\min\{ t_j\}$. Then integral \eqref{GGt} converges and defines $\GG_{\vv t}(\alpha,u,v)$ as an analytic function of $v\in\CC$ on the half-plane $\re(v)>-\alpha-2t_*$.

\end{proposition}
This result is known, see \cite[Propositions 3.15 and   3.16]{barraquand2024stationary}. For the convenience of the reader, we
include a short   self-contained  proof.
\begin{proof}
Recall that if $z>0$ then $|z^u|=z^{\re(u)}$. The main step is to establish absolute integrability.
To do so we multiply the numerator and the denominator of the fraction under the integral in \eqref{GGt} by
\begin{equation}
    \label{Jacka-trick}
\left(\sum_{k=0}^{N-1} z_k\right)^{2\alpha+2 t_*-2\eps}=\left(\sum_{k=0}^{N-1} x_1\dots x_{k} y_{k+2}\dots y_N\right)^{2\alpha+2 t_*-2\eps},
\end{equation}
 where   $z_k:=x_1\dots x_k y_{k+2}\dots y_N>0$,   $\eps>0$ is chosen so that    $\eps<\alpha+t_*$, $\eps<u+\alpha$ and $\eps<\re(v)+\alpha+2 t_*$.

Observe that with the  additional factor \eqref{Jacka-trick}, the modulus of the denominator is bounded from below as follows.
\begin{multline*}
   \big(\sum_{k=0}^{N-1} x_1\dots x_{k} y_{k+2}\dots y_N\big)^{u+\re(v)+2(\alpha+t_*-\eps)}
    =\left(\sum_{k=0}^{N-1} z_k\right)^{u+\re(v)+2(\alpha+t_*-\eps)}
   \\=
\left(\sum_{k=0}^{N-1} z_k\right)^{\re(v)+\alpha+2t_*-\eps}\left(\sum_{k=0}^{N-1} z_k\right)^{u+\alpha-\eps}
 \geq z_{N-1}^{\re(v)+\alpha+2t_*-\eps} z_0^{u+\alpha-\eps}\\ =
\left(  x_1\dots x_{N-1}\right)^{\re(v)+\alpha+2t_*-\eps}\left( y_{2}\dots y_N\right)^{u+\alpha-\eps}.
\end{multline*}
(We  used the elementary fact that   function $x\mapsto x^\theta$ is   increasing   on $(0,\infty)$ for $\theta>0$.) %
This shows that the modulus of the integrand in \eqref{GGt} is bounded   from above by
\begin{multline}
    \label{to-integrate}
   x_N^{\alpha+\re(v)+2 t_N-1}y_1^{\alpha+u-1} \left(\prod_{k=1}^{N-1} x_j^{2 (t_k-t_*)+\eps-1}\right)\left(\prod_{k=2}^N y_j^{\eps-1}\right)\\ \cdot   \left(\sum_{k=0}^{N-1} x_1\dots x_{k} y_{k+2}\dots y_N\right)^{2\alpha+2 t_*-2\eps} e^{-\sum_{j=1}^N (x_j+y_j)}.
\end{multline}

Using monotonicity of $x\mapsto x^\theta$ again, now with $\theta=2(\alpha+t_*-\eps)>0$, we obtain
$$\left(\sum_{k=0}^{N-1} z_k\right)^\theta
\leq (N \max_{k\in\{0,\dots,N-1\}}\, z_k)^\theta= N^\theta\max_{k\in\{0,\dots,N-1\}}\,z_k^\theta\leq N^\theta \sum_{k=0}^{N-1}\, z_k^\theta.$$
Thus the expression \eqref{to-integrate}   is integrable on $(0,\infty)^{2N}$ and hence the integral \eqref{GGt} converges  absolutely.

It is well known that this implies analyticity, see e.g.  \cite[page 154, Exercise 2]{Conway78}.
\end{proof}

\begin{remark}\label{RM:I-an} The integral \eqref{IIId} converges for   $t_1>\dots>t_d$ such that   $\alpha+t_d>0$,   $u>t_1$ and $\re(v)+t_d>0$.
  Thus, for fixed $u,\alpha,t_1,\dots,t_d$ as above, expression  $\II_{\vv t}^{\vv n}(\alpha,u,v)$ is an analytic function of variable $v$ on the half-plane $\re(v)>-t_d$. (This half-plane is  contained in  the half-plane  $\re(v)>-\alpha-2 t_d$
  on which  $\GG_{\vv t}(\alpha,u,v)$ is an analytic function of variable $v$.)
\end{remark}

\subsection{Proof of Theorem \ref{BarraquandThm1.11}}

We use notation \eqref{tt_d} for the repeated entries of $\vv t$. Also write $\vv n=(n_1,\dots,n_d)\in\ZZ_{\geq 1}^d$ with $n_1+\dots+n_d=N$.
Following the notation as in  \eqref{IIId}, we denote
$$\GG_{t_1,\dots,t_d}^{\vv n}(\alpha,u,v):= \GG_{\vv t}(\alpha,u,v).$$

It is convenient to work with complex $v$ and  prove that \eqref{GGII} holds for $v$ such that
\begin{equation}
    \label{v-contstr}
\re(v)+t_d>0,
\end{equation}
where we keep the remaining conditions:  %
$u>t_1>\dots>t_d>-\alpha$ intact.

Under the assumptions of Theorem \ref{BarraquandThm1.11}, both the left-hand side  and the right-hand side of \eqref{GGII} are analytic functions of variable  $v$ on the half-plane  $\{v\in\CC: \re(v)>-t_d\}$.  Therefore, to show that they are equal on  { this half-plane}, it is enough to verify that they are equal on a strip $\{v:\, -t_d<\re(v)<\alpha\}$, (note that since $\alpha+t_d>0$ this  strip is nonempty). We will prove the latter fact by induction with respect to $N$.

 If  $N=1$, then $d=1$ and \eqref{GGII} holds by direct calculation:
 from \eqref{II-1} we obtain trivially that $\GG_t(\alpha, u, v )=\Gamma(\alpha+v+2t)\Gamma(\alpha+u)$. On the other hand, \eqref{JJ-1} gives
 \begin{multline*}
     \II_t^1(\alpha, u, v)=\II_{1}(\alpha+t,u-t,v+t)
     \\ =
     \tfrac{1}{
    2\pi \Gamma(u+v)}
  \int_{0}^\infty   \tfrac{\Gamma(\alpha+t+\i x)\Gamma(\alpha+t-\i x)\Gamma( v+t+\i x)\Gamma(v+t-\i x)\Gamma(u-t +\i x)\Gamma( u-t-\i x)}{\Gamma(2\i x)\Gamma(-2\i x)}\d x.
 \end{multline*}
   For   $\alpha+t>0$, $u-t>0$  and real  $v+t>0$,  this integral can be evaluated using de Branges beta integral \eqref{GG} with $a=\alpha+t, b=u-t,c=v+t$. We obtain  $\II_t^1(\alpha, u, v)=\Gamma(\alpha+v+2t)\Gamma(\alpha+u)$, so \eqref{GGII} holds for
real  $\alpha+t>0$, $u-t>0$ and real $v+t>0$.

The induction step is based on the following two recurrence relations (proofs of which are given later):

{
\begin{claim}\label{Clm:G12}For $u>t_1>t_2>\dots>t_d>-\re(v)>-\alpha$  and $\eps>0$    satisfying
\begin{equation}
  \label{eps2}
 \re(v)+t_d>\eps
\end{equation} we have   %
   \begin{multline}\label{G-rec*}
        \GG_{t_1,\dots,t_d}^{\vv n}(\alpha,u,v)
         = \tfrac{\Gamma(\alpha+v+2 t_d)}{\Gamma(u+v)} \tfrac{1}{2\pi \i} \, \int_{-\eps+\i \RR}
   \Gamma(u+v+s)\Gamma(-s)\Gamma(\alpha-v-s)
  \hfill  \\ \cdot
   \begin{cases}
       \GG_{t_1,\dots,t_{d-1}}^{n_1,\dots,n_{d-1}}(\alpha,u,v+s)\, \d s, & n_d=1, d>1, \\
       \\
       \GG_{t_1,\dots,t_d }^{n_1,\dots,n_{d-1},n_d-1}(\alpha,u,v+s)\,\d s, & n_d=2,3,\dots
       \end{cases}
   \end{multline}
   and
      \begin{multline}\label{I-rec*}
        \II_{t_1,\dots,t_d}^{n_1,\dots,n_d}(\alpha,u,v)
         = \tfrac{\Gamma(\alpha+v+2 t_d)}{\Gamma(u+v)} \tfrac{1}{2\pi \i}
         \, \int_{-\eps+\i \RR}
   \Gamma(u+v+s)\Gamma(-s)\Gamma(\alpha-v-s)
   \hfill \\ \cdot  \begin{cases}
       \II_{t_1,\dots,t_{d-1}}^{n_1,\dots,n_{d-1}}(\alpha,u,v+s)\, \d s, & n_d=1, d>1,  \\
       \\
       \II_{t_1,\dots,t_d}^{n_1,\dots,n_d-1}(\alpha,u,v+s)\, \d s, & n_d=2,3,\dots
       \end{cases}
   \end{multline}
    \end{claim}

}

\bigskip
    {The induction hypothesis is that \eqref{GGII} holds for $N-1$, with either $n_d$ replaced by $n_d-1$ when $n_d>1$  or, when $n_d=1$, with $d$ replaced by $d-1$ and it holds  for all $v$ such that $-t_d<\re(v)<\alpha$  (note that then $-t_{d-1}<\re(v)$).} We fix  $\alpha\in\RR$, $u>-\alpha$, $n_1,\dots,n_d\in\ZZ_{\geq 1}$ such that $N=n_1+\dots+n_d$ for some $d=1,2,\dots$. We also fix  $u>t_1>t_2>\dots>t_d>-\alpha$ and $v\in\CC$ such that $-t_d<\re(v)<\alpha$. Choose $\eps>0$ such that {\eqref{eps2} holds true.}

For $s=-\eps+\mathrm{i}t$ we have $|\re(v+s)|=|\re(v)-\eps|$ and thus, by the induction hypothesis,
\[
\GG_{t_1,\dots,t_{d-1}}^{n_1,\dots,n_{d-1}}(\alpha,u,v+s)
=
\II_{t_1,\dots,t_{d-1}}^{n_1,\dots,n_{d-1}}(\alpha,u,v+s),
\]
if $n_d=1$, and
\[
\GG_{t_1,\dots,t_d}^{n_1,\dots,n_d-1}(\alpha,u,v+s)
=
\II_{t_1,\dots,t_d}^{n_1,\dots,n_d-1}(\alpha,u,v+s)
\]
if $n_d>1$.

Thus, comparing \eqref{G-rec*} and \eqref{I-rec*}, we see that \eqref{GGII} holds for $v\in\CC$ in the strip $\{v:\, -t_d<\re(v)<\alpha\}$ to  the entire half-plane  $\{v:\, \re(v)>-t_d\}$.

\bigskip

To complete the proof of \ref{BarraquandThm1.11}, it remains to show that recurrence relations \eqref{G-rec*} and \eqref{I-rec*} hold.
\bigskip
\begin{proof}[{Proof of \eqref{G-rec*}}] %
        We use \eqref{GGt} with $\vv t$ given by \eqref{tt_d}. With some abuse of notation, we relabel the components of $\vv t$ as $(t_1,\dots,t_N)$.
   We use the following gamma integral \cite[\href{http://dlmf.nist.gov/5.13.E1}{(5.13.1)}]{NIST:DLMF}
$$
\tfrac{1}{2\pi \i \Gamma(a+b)}\int_{c-\i \infty}^{c+\i \infty} \Gamma(s+a)\Gamma(b-s) z^{-s} \d s=\tfrac{z^a}{(1+z)^{a+b}},
$$
which holds under the assumptions  $\re(a+b)>0$, $-\re(a)<c<\re(b)$, $z>0$.  The integration path is a straight line parallel to the imaginary axis. (For an elementary proof and an application to the proof of \eqref{koek:1.6.3}, see \cite[(4.4)]{goldfeld2011kontorovich}.)

We put $a=u+v-\eps, b=\eps $ and $c=0$, where $\eps>0$ satisfies \eqref{eps2} so that $\re(a)=u+\re(v)-\eps> u-t_d-\eps>0 $.
Taking $z=A/B$ with $A,B>0$, we get
\begin{equation}\label{A+B}
\tfrac{1}{(A+B)^{u+v}}=\tfrac{1}{2\pi\Gamma(u+v)}\int_{-\infty}^\infty \Gamma(u+v-\eps +\i t)\Gamma(\eps-\i t) \tfrac{B^{-\eps+\i t}}{A^{u+v-\eps +\i t}}\d t.
\end{equation}

Using this formula with $A=\sum_{k=0}^{N-2}x_1\dots x_ky_{k+2}\dots y_{N}$, $B=x_1\dots x_{N-1}$,  we
rewrite the quotient under the integral in \eqref{GGt} as
\begin{multline*}
  \tfrac{ \left(\prod_{k=1}^N\,x_k^{\alpha+v+2t_{k}-1}\right)\,\left(\prod_{j=1}^N\,y_j\right)^{\alpha+u-1} }{ \left(A+B\right)^{u+v}}
  \\=
  \tfrac{1}{2\pi\i \Gamma(u+v)} \int_{-\eps-\i \infty}^{-\eps+\i \infty}
  \tfrac{ \Gamma(u+v+s)\Gamma(-s) B^s\left(\prod_{k=1}^N\,x_k^{\alpha+v+2t_{k}-1}\right)\,\left(\prod_{k=1}^N\,y_j\right)^{\alpha+u-1} }{  A^{u+v+s}}\d s
  \\
  =\tfrac{1}{2\pi\i \Gamma(u+v)}\,\int_{-\eps-\i \infty}^{-\eps+\i \infty}\,x_N^{\alpha+v+2t_N -1}\,y_N^{\alpha-v-s-1}
\\ \cdot    \tfrac{ \Gamma(u+v+s)\Gamma(-s)\left(\prod_{k=1}^{N-1}\,x_k^{\alpha+v+s+2t_{k}-1}\right)\,
  \left(\prod_{j=1}^{N-1}\,y_j\right)^{\alpha+u-1}}{ \left(\sum_{k=0}^{N-2}x_1\dots x_ky_{k+2}\dots y_{N-1}\right)^{u+v+s}}\d s.
\end{multline*}
 We insert the above into \eqref{GGt}, and integrate with respect to $\d  x_N$ and $\d y_N$. The {absolute} convergence of the latter integral is ensured by $\re(\alpha-v-s)=\alpha-\re(v)+\eps>0$. %
 Then we integrate with respect to the remaining variables in $\d \vv x$, $\d \vv y$ and switch the order of integrals.   We obtain
\begin{multline*}
  \GG_{t_1,\dots,t_N}(\alpha,u,v)
  =\tfrac{\Gamma(\alpha+v+2t_N)}{2\pi i\Gamma(u+v)}\\ \cdot  \int_{-\eps-\i \infty}^{-\eps+\i \infty}
\Gamma(u+v+s)\Gamma(-s) \Gamma(\alpha-v-s) \GG_{t_1,\dots,t_{N-1}}(\alpha,u,v+s) \d s.
\end{multline*}

Depending on the value of $n_d$, this is \eqref{G-rec*}.
Indeed, $\vv t=(t_1,\dots,t_{d-1},t_d,\dots,t_d)$ if $n_{d}>1$ and $\vv t=(t_1,\dots,t_{d-1},t_d)$ if $n_{d}=1$.
    \end{proof}

\bigskip\begin{proof}[{Proof of \eqref{I-rec*}}]%
We use the following identity:
 \normalsize %
   \begin{multline}\label{dualHahn}
    \tfrac{1}{2\pi}\int_{(0,\infty)} \tfrac{\left|\Gamma(t_{d-1}-t_d + \i  {y}+\i  {x} )\right|^2 \, \left|\Gamma(
t_{d-1}-t_d + \i  {y}-\i  {x} )\right|^2\Gamma(v+s+t_{d}+\i  {x})\Gamma(v+s+t_{d}-\i  {x})}{ |\Gamma(2\i  {x})|^2}\,\d x \\
= \Gamma(2t_{d-1}-2t_d)\Gamma(v+s+t_{d-1}+\i  {y}) \Gamma(v+s+t_{d-1}-\i  {y}),
\end{multline}
which holds for $\re(s)=-\eps$ under the assumptions $t_d<t_{d-1}$ and $\re(v+s)+t_d=\re(v)+t_d-\eps>0$, see \eqref{eps2}. To obtain this formula, we apply  \eqref{GG+}  with
 $a=v+s +t_d$, $b=t_{d-1}-t_d+\i  {y}$, $c=t_{d-1}-t_d-\i {y}$, which  gives
\begin{multline*}
 \frac{1}{2\pi}
 \int_{(0,\infty)}\,\tfrac{\left|\Gamma(t_{d-1}-t_d + \i {y}+\i {x} )\right|^2 \, \left|\Gamma(
t_{d-1}-t_d + \i {y}-\i {x} )\right|^2\Gamma(v+s+t_{d}+\i {x})\Gamma(v+s+t_{d}-\i {x})}{ |\Gamma(2\i {x})|^2}\,\d x
\\ =\Gamma(a+b)\Gamma(a+c)\Gamma(b+c)
\\ =
\Gamma(v+s+t_{d-1}+\i {y}) \Gamma(v+s+t_{d-1}-\i {y})\Gamma(2t_{d-1}-2t_d),
\end{multline*}
and  \eqref{dualHahn} follows.

Expression \eqref{IIId}  can be rewritten as
\begin{multline}
     \II_{t_1,\dots,t_d}^{n_1,\dots,n_d}(\alpha,u,v)=\tfrac{1}{
    (2\pi)^{d-1} \Gamma(u+v)\prod_{j=2}^{d-1}\Gamma(2t_{j-1}-2t_j)}\\ \cdot  \int_{(0,\infty)^{d-1}}\,\Big(\prod_{j=1}^{d-1}\tfrac{\left|\Gamma(\alpha+t_{j}+\i {y_j})\right|^{2n_j}}{ |\Gamma(2\i {y_j})|^2} \Big)
    \left|\Gamma(u-t_{1} +\i {y_1} )\right|^2\label{IN1}\\
    \cdot \, \Big(\prod_{j=2}^{d-1}\left|\Gamma(t_{j-1}-t_j + \i {y_{j-1}}+\i {y_j} )\right|^2 \, \left|\Gamma(
t_{j-1}-t_j + \i {y_{j-1}}-\i {y_j} )\right|^2\Big)
\\ \cdot  J(y_{d-1})\,\d y_1\dots\d y_{d-1},
\end{multline}
where $J(y_{d-1})$ is the   integral with respect to $\d y_d$, given by: %
\begin{multline}\label{Jacka-J}
J(y)=\tfrac1{2\pi\,\Gamma(2t_{d-1}-2t_d)}\\ \cdot  \int_{(0,\infty)}\,\left|\Gamma(\alpha+t_{d}+\i {x})\right|^{2(n_d-1)}\,\tfrac{\left|\Gamma(t_{d-1}-t_d + \i {y}+\i {x} )\right|^2 \, \left|\Gamma(
t_{d-1}-t_d + \i {y}-\i {x} )\right|^2}{ |\Gamma(2\i {x})|^2}\\
\cdot \,\left(\left|\Gamma(\alpha+t_{d}+\i {x})\right|^2\Gamma(v+t_{d}+\i {x})\Gamma(v+t_{d}-\i {x})\right)\,\d x.
\end{multline}
\normalsize
In the above expression, we factored $\left|\Gamma(\alpha+t_{d}+\i {x})\right|^2$ out of   $\left|\Gamma(\alpha+t_{d}+\i {x})\right|^{2n_d}$, keeping it as one of the product in the parentheses at the end of the integrand in \eqref{Jacka-J}.
Referring to \eqref{koek:1.6.3} with  parameters $a=v+t_d -\eps + \i {x}$, $b=v+t_d -\eps -\i {x}$, $c=\eps $, $d=\alpha+\eps-v$
  which   have positive real part, see e.g. \eqref{eps2}, we can express the last factor {of  the integrand in} \eqref{Jacka-J} %
  as
\begin{multline*}
    \left|\Gamma(\alpha+t_{d}+\i {x})\right|^2\Gamma(v+t_{d}+\i {x})\Gamma(v+t_{d}-\i {x})
\\= \tfrac{\Gamma(\alpha+v+2 t_d)}{2\pi \i}  \int_{-\eps-\i \infty}^{-\eps+\i \infty}
\Gamma(-s) \Gamma(\alpha-v-s) \Gamma(v+s+t_d+\i {x}) \Gamma(v+s+t_d-\i {x}) \d s.
\end{multline*}
Therefore,
\begin{multline*}
J(y)=\tfrac1{2\pi\,\Gamma(2t_{d-1}-2t_d)}\\ \cdot  \int_{(0,\infty)}\,\left|\Gamma(\alpha+t_{d}+\i {x})\right|^{2(n_d-1)}\,
\tfrac{\left|\Gamma(t_{d-1}-t_d + \i {y}+\i {x} )\right|^2 \, \left|\Gamma(
t_{d-1}-t_d + \i {y}-\i {x} )\right|^2}{|\Gamma(2\i {x})|^2}\\
\cdot \,\tfrac{\Gamma(\alpha+v+2 t_d)}{2\pi \i}  \int_{-\eps-\i \infty}^{-\eps+\i \infty}
\Gamma(-s) \Gamma(\alpha-v-s) \Gamma(v+s+t_d+\i {{x}}) \Gamma(v+s+t_d-\i {{x}}) \d s\,\d x.
\end{multline*}
Changing the order of integration we get 
\begin{multline}\label{J-step}
J(y)=\tfrac{\Gamma(\alpha+v+2 t_d)}{2\pi \i}  \int_{-\eps-\i \infty}^{-\eps+\i \infty}
\Gamma(u+v+s)\Gamma(-s) \Gamma(\alpha-v-s)\,\\
\cdot \,\tfrac1{2\pi\,\Gamma(2t_{d-1}-2t_d)\Gamma(u+v+s)}\,\int_{(0,\infty)}\,\left|\Gamma(\alpha+t_{d}+\i {x})\right|^{2(n_d-1)}\\
\cdot  \,\tfrac{\left|\Gamma(t_{d-1}-t_d + \i {y}+\i {x} )\right|^2 \, \left|\Gamma(
t_{d-1}-t_d + \i {y}-\i {x} )\right|^2\Gamma(v+s+t_d+\i {x}) \Gamma(v+s+t_d-\i {x})}{|\Gamma(2\i {x})|^2} \,\d x\, \d s.
\end{multline}

If $n_d>1$,  we use the above expression for $J(y_{d-1})$ in \eqref{IN1}. After changing the order of integration we obtain the second part of formula \eqref{I-rec*}.

If $n_d=1$ then we use \eqref{dualHahn} to compute the inner integral   with respect to $\d x$ in \eqref{J-step}. We obtain
\begin{multline*}
J(y)=\tfrac{\Gamma(\alpha+v+2 t_d)}{2\pi \i} \\ \cdot   \int_{-\eps-\i \infty}^{-\eps+\i \infty}
\,\Gamma(u+v+s)\Gamma(-s) \Gamma(\alpha-v-s)\tfrac{\Gamma(v+s+t_{d-1}+\i {y}) \Gamma(v+s+t_{d-1}-\i {y})}{\Gamma(u+v+s)}\,\, \d s.
\end{multline*}

Using this  expression for $J(y_{d-1})$ in \eqref{IN1}, after changing the order of integration we obtain the first  part of formula \eqref{I-rec*}.
\end{proof}

\section{Proof of Proposition \ref{Prop2}}\label{Sec:AC}

 By Lemma~\ref{Lem:I-anal} (or Proposition~\ref{Cl:G-an}), the left-hand side of \eqref{I_N-x} is analytic for $u,v\in\CC$ with $\re(u)>-\alpha$ and $\re(v)>-\alpha$. Proposition~\ref{Lem:I=J} therefore reduces the problem to analyzing the analytic continuation of the integral \eqref{JJ-1}.

  For complex $u,v$ with $\re(u),\re(v) >0$ we rewrite the integral  \eqref{JJ-1} as
  \begin{equation}\label{step0}
   \II_N(\alpha,u,v)=
\tfrac{1}{
   4 \pi \i \Gamma(u+v)}\int_{\i \RR}\,\tfrac{ \Gamma(\alpha+z)^N\Gamma(\alpha-z)^N\Gamma(v+z)\Gamma(v-z)\Gamma(u+z)\Gamma(u-z )}{\Gamma(2z)\Gamma(-2z)}\,\d z.
 \end{equation}

If $\re(u)<0$ or $\re(v)<0$, then the integrand has singularities that affect the analytic continuation.

It is convenient to decouple the repeated occurrences of $u,v$ in this integral.
Denote ${\bf s}=(s_1,s_2) \in \CC^2$ and ${\bf t}=(t_1,t_2)\in \CC^2$. We define
\begin{equation*}\label{def_G}
H({\bf s},{ \bf t}; z):=
\tfrac{\Gamma(\alpha+z)^N\Gamma(\alpha-z)^N\Gamma(z-s_1)\Gamma(z-s_2)\Gamma(t_1-z)\Gamma(t_2-z)}{\Gamma(2z)\Gamma(-2z)},
\end{equation*}
where $\alpha>0$ is a fixed parameter.
We assume that
$\re(s_i)<0$ (for all $i=1,2$) and $\re(t_j)>0$ (for all $j=1,2$) and we define
\begin{equation}\label{def_F}
F({\bf s},{ \bf t})=\tfrac{1}{2\pi \i} \int_{\i \RR}
 H({\bf s},{ \bf t}; z) \d z.
\end{equation}
 Due to condition  $\re(s_i)<0$ and $\re(t_j)>0$, the integrand $ H({\bf s},{ \bf t}; z)$ is analytic function of $z$ on the contour $\i \RR$. The  integrand decays exponentially fast as $|\im(z)|\to \infty$ and is a continuous function of $z$ on $\i \RR$. Thus the integral converges and $F({\bf s},{ \bf t})$ is well-defined and it is analytic function of arguments $\vv s,\vv t$ on the domain
 $$
 D_0:=\bigl\{ \re(s_i)<0\; \text{for all } i=1,2,\;\text{and } \re(t_j)>0\;\text{for all } j=1,2\bigr\}.
 $$

 Note that  \eqref{step0} is
\begin{equation}
    \label{I2F}
    \II_N(\alpha,u,v)=\tfrac{F(-u,-v,u,v)}{2\Gamma(u+v)}.
\end{equation}

We denote by ${\mathcal C}$ any piecewise-smooth simple   curve in $\CC$ which coincides with the vertical line $\i \RR$
in the neighborhood of $\infty$.  This contour divides $\CC$ into two open sets: those points that lie to the right of this contour and those points that lie to the left of the contour.
 We will denote these parts by $\CC^+_{\mathcal C}$ and $\CC^-_{\mathcal C}$.
 See Figure \ref{Fig1+}  for an illustration.

\begin{figure}%
 \BeginAccSupp{ActualText={A contour that divides complex plane into two unbounded open sets}}
  \begin{tikzpicture}[scale=1]
  \fill[pattern=horizontal lines, pattern color=gray!60] (1,1.5) -- plot[domain=-3.755:3.585]  ({5.3-sign(\x)+cos(57.29577*\x+90-sign(\x)*90)},{5 +(2.5+sign(\x))*sin(57.29577*\x)/(abs(\x)-4)})
-- (1,8.6)-- (1,10) -- cycle;

 \fill[pattern=north east lines, pattern color=gray!60]   plot[domain=-3.755:3.58]  ({5.3-sign(\x)+cos(57.29577*\x+90-sign(\x)*90)},{5 +(2.5+sign(\x))*sin(57.29577*\x)/(abs(\x)-4)})
-- (8.5,8.53)-- (8.5,1.48 ) -- cycle;

 \draw[->,thick] (5,1) to (5,9.3);
 \draw[->] (0,5) to (9,5);

\draw[scale = 1,domain=-3.76:3.6,smooth,variable=\x,thick,blue] plot ({5.3-sign(\x)+cos(57.29577*\x+90-sign(\x)*90)},{5 +(2.5+sign(\x))*sin(57.29577*\x)/(abs(\x)-4)});

\draw[->, thick,blue] (7.2,3) to (7.22, 3.5);

\node [below] at (7,8) {$\CC_{\mathcal{C}}^+$};

\node [below] at (2,3) {$\CC_{\mathcal{C}}^-$};

\node [right] at (3.5,8) {\color{blue}$ \mathcal{C}$};

\node [right] at (5,9) {$\i \RR$};
 \end{tikzpicture}
\caption{ Contour $\mathcal{C}$ divides $\CC$ into two open sets  $\CC^+_{\mathcal C}$ and $\CC^-_{\mathcal C}$.}
\label{Fig1+}
\end{figure}
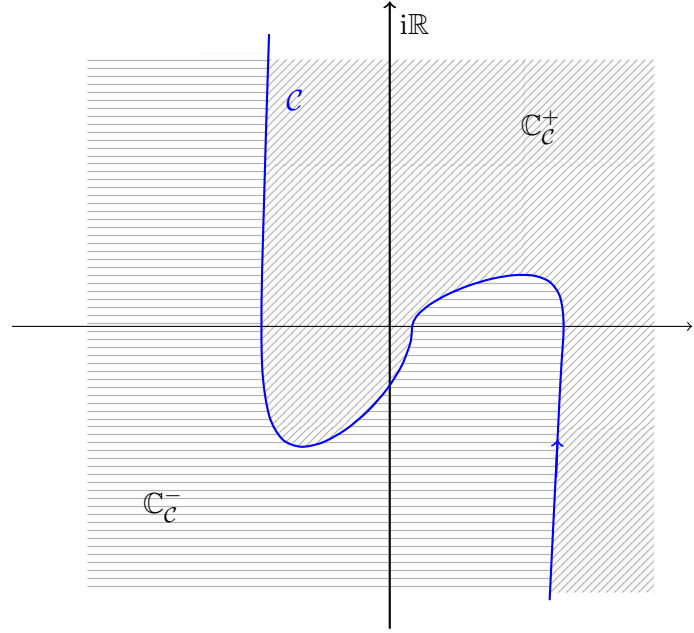

\begin{definition}
We call a contour ${\mathcal C}$ (of the type defined above) a \emph{good contour}
if it satisfies the following properties:
\begin{itemize}
\item if $z\in  \CC^+_{\mathcal C}$   then the whole half-line
$\{z+t \; : \; t \in [0,\infty) \}\subset \CC^+_{\mathcal C}$;
\item if $z\in  \CC^-_{\mathcal C}$ then the whole half-line
$\{z+t \; : \;t\in (-\infty,0]\} \subset \CC^-_{\mathcal C}$.
\end{itemize}
\end{definition}
For example, the contour ${\mathcal C}$ in Figure \ref{Fig1+} is  not  a good contour, whereas the  contours in Figures  \ref{Fig3+}  and \ref{Fig4} are good.

We will say that a good contour {\it separates} ${\vv s}$ and ${\vv t}$ if
$s_i \in \CC^-_{\mathcal C}$ (for all $i=1,2$) and
$t_j \in \CC^+_{\mathcal C}$ (for all $j=1,2$). See Figure  \ref{Fig3+} for an example. Thus domain $D_0$ is the precisely the domain of those $({\vv s},{\vv t})$ which are separated by the contour $\i \RR$.

\begin{figure}[H]
 \BeginAccSupp{ActualText={A good contour that separates s-variable poles that stay on the left from t-variable poles that stay on the right }}
 \begin{tikzpicture}[scale=.67]
 \draw[->] (5,0) to (5,11);
 \draw[->] (-2,5) to (13,5);
 \draw[-,dotted] (-1.,0) to (-1,11);
  \draw[-,dotted] (12.,0) to (12,11);
   \draw[-] (12.,4.9) to (12,5.05);
    \draw[-] (-1.,4.9) to (-1,5.05);
  \node [below] at (-1.2,5) {$-\alpha$};
  \node [below] at (12.2,5) {$\alpha$};
   \node [right] at (5,11) {$\i \RR$};
   \draw[fill,red] (0.3,7.8) circle (.05);
    \node [below] at (0.3,7.6) {\scriptsize$t_1$};
     \draw[fill,red] (2.3,7.8) circle (.05);
    \node [below] at (2.3,7.6) {\scriptsize$t_1+1$};
     \draw[fill,red] (4.3,7.8) circle (.05);
    \node [below] at (4.3,7.6) {\scriptsize$t_1+2$};
         \draw[ ] (6.3,7.8) circle (.05);
    \node [below] at (6.3,7.6) {\scriptsize$t_1+3$};
             \draw[ ] (8.3,7.8) circle (.05);
    \node [below] at (8.3,7.6) {\scriptsize$t_1+4$};
                \draw[ ] (10.3,7.8) circle (.05);
    \node [below] at (10.3,7.6) {\scriptsize$t_1+5$};
            \draw[ ] (12.3,7.8) circle (.05);
    \node [below] at (12.3,7.6) {\scriptsize$t_1+6$};

     \draw[fill,red] (1.6,4) circle (.05);
    \node [below] at (1.6,3.8) {\scriptsize$t_2$};
     \draw[fill,red] (3.6,4) circle (.05);
    \node [below] at (3.6,3.8) {\scriptsize$t_2+1$};
     \draw[] (5.6,4) circle (.05);
    \node [below] at (5.6,3.8) {\scriptsize$t_2+2$};
         \draw[ ] (7.6,4) circle (.05);
    \node [below] at (7.6,3.8) {\scriptsize$t_2+3$};
             \draw[ ] (9.6,4) circle (.05);
              \node [below] at (9.6,3.8) {\scriptsize$t_2+4$};
    \node [below] at (11.6,3.8) {\scriptsize$t_2+5$};
            \draw[ ] (11.6,4) circle (.05);
            \draw[ ] (13.6,4) circle (.05);
    \node [below] at (13.6,3.8) {\scriptsize$t_2+6$};

       \draw[] (-1.2,8.2) circle (.05);
    \node [above] at (-1.2,8.4) {\scriptsize $s_2-5$};
     \draw[] (-3.2,8.2) circle (.05);
    \node [above] at (-3.2,8.4) {\scriptsize$s_2-6$};

   \draw[] (0.8,8.2) circle (.05);
    \node [above] at (0.8,8.4) {\scriptsize$s_2-4$};
     \draw[] (2.8,8.2) circle (.05);
    \node [above] at (2.6,8.4) {\scriptsize$s_2-3$};
     \draw[] (4.8,8.2) circle (.05);
    \node [above] at (4.8,8.4) {\scriptsize$s_2-2$};
         \draw[fill,blue] (6.8,8.2) circle (.05);
    \node [above] at (6.8,8.4) {\scriptsize$s_2-1$};
             \draw[fill,blue] (8.8,8.2) circle (.05);
    \node [above] at (8.8,8.4) {\scriptsize $s_2$};

     \draw[] (-1.8,6) circle (.05);
    \node [above] at (-1.8,6.2) {\scriptsize $s_1-5$};

   \draw[] (0.05,6) circle (.05);
    \node [above] at (0.05,6.2) {\scriptsize $s_1-4$};
     \draw[] (2.05,6) circle (.05);
    \node [above] at (2.05,6.2) {\scriptsize$s_1-3$};
     \draw[] (4.05,6) circle (.05);
    \node [above] at (4.05,6.2) {\scriptsize$s_1-2$};
         \draw[fill,blue] (6.05,6) circle (.05);
    \node [above] at (6.05,6.2) {\scriptsize$s_1-1$};
             \draw[fill,blue ] (8.05,6) circle (.05);
    \node [above] at (8.05,6.2) {\scriptsize$s_1$};
 \draw[-,thick,blue] (5,9.2) to (5,11);
 \draw[-,thick,blue] (5.2,9) to (9.,9);
  \draw[->,thick,blue] (5,0) to (5,1);
  \node [right] at (5.2,1) {\color{blue}\scriptsize $\mathcal C$};
 \draw[-,thick,blue] (5,1) to (5,1.8);
 \draw[-,thick,blue] (1.2,2) to (4.8,2);
  \draw[-,thick,blue] (1,2.2) to (1,5.3);
    \draw[-,thick,blue] (8.3,5.5) to (1.2,5.5);
    \draw[-,thick,blue] (8.5,5.7) to (8.5,6.6);
      \draw[-,thick,blue] (-.6,6.8) to (8.3,6.8);
  \draw[-,thick,blue] (9.2,8.8) to (9.2,8.2);

    \draw[-,thick,blue] (-.8,7.8) to (-.8,7);

 \draw[-,thick,blue] (-.6,8) to (9,8);

 \draw[-,thick,blue] (5, 1.8) arc (0:90:.2);
\draw[-,thick,blue] (5, 9.2) arc (-180:-90:.2);
\draw[-,thick,blue] (-.8, 7) arc (-180:-90:.2);
\draw[-,thick,blue] (-.8, 7.8) arc (180:90:.2);
\draw[-,thick,blue] (1, 5.3) arc (180:90:.2);
\draw[-,thick,blue] (1, 2.2) arc (-180:-90:.2);
\draw[-,thick,blue] (9.2, 8.2) arc (0:-90:.2);
\draw[-,thick,blue] (8.5, 5.7) arc (0:-90:.2);
\draw[-,thick,blue] (9.2, 8.8) arc (0:90:.2);
 \draw[-,thick,blue] (8.5, 6.6) arc (0:90:.2);

\end{tikzpicture}
\caption{ A good contour that separates $\vv s=(s_1,s_2)$ and $\vv t=(t_1,t_2)$, with the locations of the poles of $H(\vv s,\vv t;z)$ marked.
With $s_1-s_2\not\in\ZZ$ and $t_1-t_2\not\in\ZZ$, all poles are simple poles. When the curve $\mathcal{C}$ is   replaced   back by $\i \RR$, the poles marked in blue contribute  positive residue, and the poles marked in red contribute negative residue.
}
\label{Fig3+}
\end{figure}
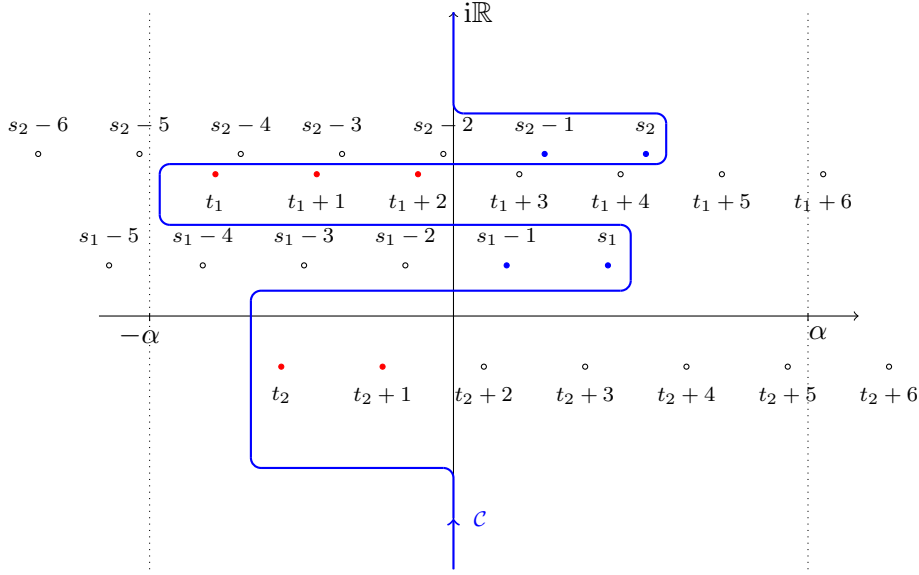

\begin{proposition}\label{Prop-A1}${}$
\begin{itemize}
\item[(i)] For any fixed $({\bf s}, {\bf t}) \in D_0$   the value of the integral in \eqref{def_F}   remains the same if the contour of integration is replaced by any good contour ${\mathcal C}$ that separates
${\bf s}$ and ${\bf t}$ and lies  in the strip  $-\alpha<\re(z)<\alpha$.
\item[(ii)] Given a fixed good contour  ${\mathcal C}$ that lies in the strip  $-\alpha<\re(z)<\alpha$, the integral
\begin{equation}\label{def_F2}
\frac{1}{2\pi \i} \int_{{\mathcal C}}
 H({\bf s},{ \bf t}; z) \d z
\end{equation}
defines a function analytic in the domain
$$
 D_{\mathcal C}:=\bigl\{ (\vv s,\vv t):  s_i \in \CC^-_{\mathcal C}, i=1,2  \mbox{ and }
t_j \in \CC^+_{\mathcal C}, j=1,2\bigr\}.
$$
This function is an analytic continuation of $F({\bf s},{ \bf t})$ defined in \eqref{def_F} into domain $D_{\mathcal C}$.
\end{itemize}
\end{proposition}
\begin{proof}
The proof of part (i) follows from the Cauchy Integral Theorem: we can modify the contour of integration in any way we want, as long as it goes to infinity along a vertical line (so that we have exponential decay of the integrand) and the modification avoids singularities of the integrand.
 If  $\re(s_j)\geq 0$ or $\re(t_j)\le 0$, then  $H(\vv s, \vv t; z)$, as a function of argument $z$ restricted to the strip $|\re(z)|<\alpha$ has a finite number of singularities at points $z\in\mathcal{P}$, where
\begin{equation}\label{defP}
{\mathcal P}={\mathcal P}_{\vv s}\cup{\mathcal P}_{\vv t},
\end{equation}
where
\begin{align*}
    {\mathcal P}_{\vv s}&=\{ s_i-k_i :   \; k_i \in {\mathbb Z}_{\ge 0}, \; \re(s_i)-k_i>0,\; i=1,2\}, \\
{\mathcal P}_{\vv t}&=\{ t_j+m_j :  \;  m_j \in {\mathbb Z}_{\ge 0}, \; \re(t_j)+m_j<0,\; j=1,2\}.
\end{align*}

Note that since  contour  $\mathcal C$ lies in the strip $\{z: -\alpha<\re(z)<\alpha\}$, it automatically avoids singularities that appear in the expression  $\Gamma(\alpha+z)\Gamma(\alpha-z)$.

Part (ii) is fairly obvious, since both integrals in \eqref{def_F} and \eqref{def_F2} have the same values
if
 $\re(s_i)<-A$ (for all $i=1,2$) and $\re(t_j)>A$ (for all $j=1,2$) for some $A$ large enough. (We can take $A=\max_{z\in\mathcal C}|\re(z)|$.)
See Figure \ref{Fig4}.
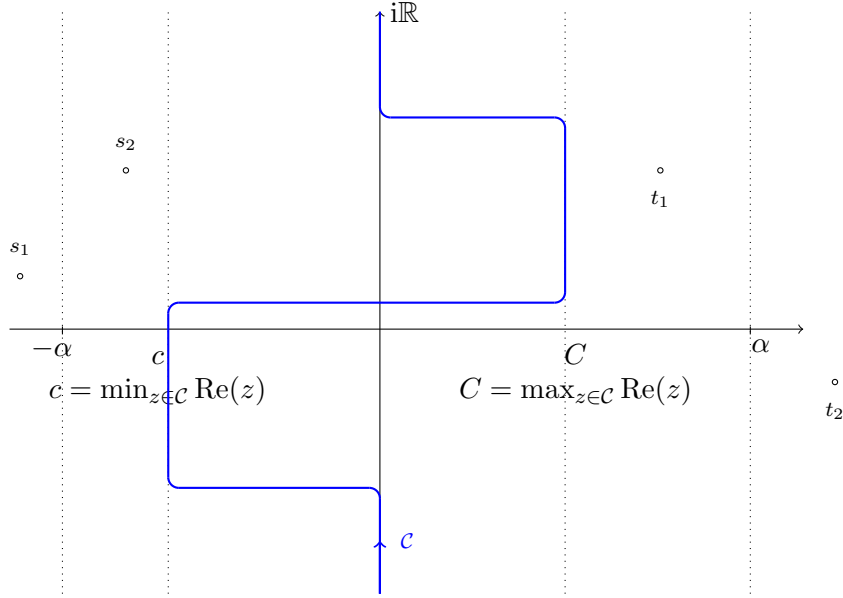
\begin{figure}[H]
 \BeginAccSupp{ActualText={A good contour illustrating the definition of c as the minimum of the real parts and C and the maximum of the real parts of point on the contour }}
 \begin{tikzpicture}[scale=.7]
 \draw[->] (5,0) to (5,11);
 \draw[->] (-2,5) to (13,5);
 \draw[-,dotted] (-1.,0) to (-1,11);
  \draw[-,dotted] (12.,0) to (12,11);
   \draw[-] (12.,4.9) to (12,5.05);
    \draw[-] (-1.,4.9) to (-1,5.05);
  \node [below] at (-1.2,5) {$-\alpha$};
  \node [below] at (12.2,5) {$\alpha$};
   \node [right] at (5,11) {$\i \RR$};

    \draw[-,dotted] (8.5,0) to (8.5,11);
     \node [below] at (8.7,5) {$\begin{matrix}C\\ C=\max_{z\in \mathcal C}\re(z)\end{matrix}$};
        \draw[-,dotted] (1,0) to (1,11);
     \node [below] at (.8,5) {$\begin{matrix}c\\ c=\min_{z\in \mathcal C}\re(z)\end{matrix}$};
                 \draw[ ] (10.3,8) circle (.05);
    \node [below] at (10.3,7.8) {\scriptsize$t_1$};
             \draw[ ] (13.6,4) circle (.05);
     \node [below] at (13.6,3.8) {\scriptsize$t_2$};

       \draw[] (0.2,8) circle (.05);
    \node [above] at (0.2,8.2) {\scriptsize $s_2$};

     \draw[] (-1.8,6) circle (.05);
    \node [above] at (-1.8,6.2) {\scriptsize $s_1$};

 \draw[-,thick,blue] (5,9.2) to (5,11);
 \draw[-,thick,blue] (5, 9.2) arc (-180:-90:.2);

 \draw[-,thick,blue] (5.2,9) to (8.3,9);
 \draw[-,thick,blue] (8.3, 9) arc (90:0:.2);

    \draw[-,thick,blue] (8.5,5.7) to (8.5,8.8);
 \draw[-,thick,blue] (8.5, 5.7) arc (0:-90:.2);

      \draw[-,thick,blue] (8.3,5.5) to (1.2,5.5);
    \draw[-,thick,blue] (1, 5.3) arc (180:90:.2);

  \draw[->,thick,blue] (5.,0) to (5,1);
  \node [right] at (5.2,1) {\color{blue}\scriptsize $\mathcal C$};
 \draw[-,thick,blue] (5,0) to (5,1.8);
 \draw[-,thick,blue] (1.2,2) to (4.8,2);

  \draw[-,thick,blue] (1, 2.2) arc (-180:-90:.2);

  \draw[-,thick,blue] (1,2.2) to (1,5.3);

    \draw[-,thick,blue] (5, 1.8) arc (0:90:.2);
\end{tikzpicture}
\caption{In the proof of Proposition \ref{Prop-A1}, we can take $A=\max\{-c, C\}$. }
\label{Fig4}
\end{figure}
\end{proof}
This proposition leads to the following result:

\begin{theorem}\label{Thm:Alexey}
The function $F({\bf s},{ \bf t})$ can be analytically continued to a function analytic in the domain
\begin{multline*}
  {\mathfrak D}_\alpha:=\Big\{ {\vv s} \in \CC^2, {\vv t} \in \CC^2 \; :\;  s_i - t_j \not\in [0,\infty), \; \re(s_i)<\alpha, \;\re(t_j)>-\alpha
\\ \textnormal{ for all $i,j=1,2$}\Big\}.
\end{multline*}
\end{theorem}
\begin{proof}
If $s_i - t_j \not\in [0,\infty)$  then there exists a good contour ${\mathcal C}$ that separates ${\bf s}$ and ${\bf t}$, and in fact separates the singularities in $\mathcal{P}_{\vv s}$ from the singularities in $\mathcal{P}_{\vv t}$, see Fig. \ref{Fig3+}.
Furthermore, since  $\re(s_j)<\alpha$ and $\re(t_j)>-\alpha$, this  contour can be chosen to lie in the strip $\{z: -\alpha<\re(z)<\alpha\}$.
Then we apply Proposition \ref{Prop-A1}.
\end{proof}
Now we have analytically continued $F({\bf s},{ \bf t})$ into ${\mathfrak D}_\alpha$. So we fix $({\bf s},{ \bf t}) \in {\mathfrak D}_\alpha$. There is some good contour ${\mathcal C}\subset\{|\re(z)|<\alpha\}$ that separates ${\bf s}$ and ${ \bf t}$, so that \eqref{def_F2} holds.
We can shift the contour of integration back from ${\mathcal C}$ to $\i \RR$. Of course, we need to ensure that ${\mathcal P} \cap \i \RR = \emptyset$ (see \eqref{defP} for definition of ${\mathcal P}$), so that no poles of the integrand lie on $\i \RR$. This can be done provided $\re(s_i)\not\in\ZZ_{\geq 0}$ and $\re(t_j)\not\in \ZZ_{\leq 0}$.

We encounter singularities arising from $t_1,t_2$ in the left half-plane $\re(z)<0$ and singularities from $s_1,s_2$ in the right half-plane $\re(z)>0$.
These singularities are simple poles if $t_1-t_2$ is not an integer, and $s_1-s_2$ is not an integer. (Note that the poles are
not simple  if $s_1,s_2>0$ and $s_1-s_2\in\ZZ$ or if $t_1,t_2<0$ and $t_1-t_2\in\ZZ$.)  Since $\mathcal{C}\subset\{|\re(z)|<\alpha\}$,
when we  change    the contour of integration from ${\mathcal C}$ to $\i \RR$, we get
 \begin{multline*}
 F({\bf s},{ \bf t})=\frac{1}{2\pi \i} \int_{{\mathcal C}} H({\bf s},{ \bf t}; z) \d z
 =\frac{1}{2\pi \i} \int_{\i \RR}
  H({\bf s},{ \bf t}; z) \d z \\ +  \sum_{\{k:\; \re(s_1)-k\geq 0\}}  \Res_{ z=s_1-k} H(\vv s,\vv t;z) +  \sum_{\{j:\; \re(s_2)-j\geq 0\}}  \Res_{z=s_2-j} H(\vv s,\vv t;z)\\
  -  \sum_{\{k:\; \re(t_1)+k\leq 0\}} \Res_{z=t_1+k} H(\vv s,\vv t;z) -   \sum_{\{j:\; \re(t_2)+j\leq 0\}}  \Res_{ z=t_2+j}H(\vv s,\vv t;z). \nonumber
 \end{multline*}
The above formula holds under the following assumptions on $\vv s$, $\vv t$:
 \begin{multline}\label{s-t-restrict1}
 \re(s_i)\not\in\ZZ_{\geq 0} \mbox{  and } \re(t_j)\not\in \ZZ_{\leq 0}; \\
 s_i-t_j\not\in[0,\infty), \re(s_i)<\alpha, \re(t_j)>-\alpha;
 \\ \mbox{if both $\re(s_j)>0$ we require } s_1-s_2\not\in\ZZ,  \\ \mbox{if both $\re(t_j)<0$ we require }  t_1-t_2\not\in\ZZ.
 \end{multline}
 We will be interested in  the case $\vv s=-\vv t $, in which case condition \eqref{s-t-restrict1}
 becomes
 \begin{multline}
\label{s-t-restrict2}
\re(t_1)\not\in\ZZ_{\leq 0},\; \re(t_2)\not\in\ZZ_{\leq 0},\;
\re(t_1)>-\alpha,\;\re(t_2)>-\alpha ; \\
\quad t_1\not\in(-\alpha, 0],\; t_2\not\in(-\alpha,0], \; t_1+t_2\not\in(-\infty,0];
\\ \mbox{if both $\re(t_1),\re(t_2)<0$ then we also require that   $t_1-t_2\not\in \ZZ$}.
 \end{multline}
 With $\vv s=-\vv t$,
 $$\Res_{ z=s_j-k} H(\vv s,\vv t;z)= \Res_{ z=-t_j-k}H(\vv s,\vv t;z)= -\Res_{ z=t_j+k} H(\vv s,\vv t;z),  k\in\ZZ_{\geq 0}$$ so under conditions \eqref{s-t-restrict2}, we obtain
   \begin{multline} \label{Res}
 F({\bf s},{ \bf t})=\frac{1}{2\pi \i} \int_{{\mathcal C}} H({\bf s},{ \bf t}; z) \d z
 =\frac{1}{2\pi \i} \int_{\i \RR} H({\bf s},{ \bf t}; z) \d z  \\  - 2  \sum_{\{k\in\ZZ_{\geq 0}:\; \re(t_1)+k\leq 0\}} \Res_{ z=t_1+k} H(\vv s,\vv t;z)  \\ -   2 \sum_{\{j\in\ZZ_{\geq 0}:\; \re(t_2)+j\leq 0\}}
 \Res_{ z=t_2+j}  H(\vv s,\vv t;z).
 \end{multline}
\subsection{Conclusion of the proof of Proposition \ref{Prop2}}
Now we take $u,v\in\RR$ as in Proposition \ref{Prop2}. In particular, $u,v\not\in\ZZ_{\leq 0}$. Consider
$t_1=u+\i \eps_1$, $t_2=v+\i \eps_2$, $\vv s=-\vv t$ with $\eps_1\ne \eps_2\in(0,1)$. It is  clear that condition \eqref{s-t-restrict2} is then satisfied,
with $t_1+t_2=u+v+\i(\eps_1+\eps_2)\not\in(-\infty, 0]$, so
from \eqref{Res} and \eqref{I2F} we obtain %
\begin{multline}\label{I_N-eps}
    \II_N(\alpha,u+\i \eps_1,v+\i\eps_2)
    \\=\frac{1}{
   4 \pi \i \Gamma(u+v)}\int_{\i \RR}\tfrac{ \Gamma(\alpha+z)^N\Gamma(\alpha-z)^N\Gamma(v+\i\eps_2+z)
   \Gamma(v+\i \eps_2-z)\Gamma(u+\i\eps_1+z)\Gamma(u+\i \eps_1-z )}{\Gamma(2z)\Gamma(-2z)}\d z
    \\+\sum_{\{j\geq 0:\; v+j\leq 0\}}m_j \Gamma(\alpha+v+\i \eps_2+j)^N\Gamma(\alpha -v -j -\i\eps_2)^N
   \\+\sum_{\{k\geq 0:\; u+k\leq 0\}}M_k \Gamma(\alpha+u+\i \eps_1+k)^N\Gamma(\alpha -u -k -\i\eps_1)^N,
\end{multline}
where
\begin{multline*}%
    m_j =-\tfrac{1}{\Gamma\left(u+v+\i (\eps_1+\eps_2)\right)} \\ \cdot  \Res_{z=v+\i\eps_2+j}\left(\tfrac{\Gamma(v+\i\eps_2+z)\Gamma(v+\i \eps_2-z)\Gamma(u+\i\eps_1+z)\Gamma(u+\i \eps_1-z )}{\Gamma(2z)\Gamma(-2z)} \right)
    \\=-\tfrac{(-1)^j}{j!}\, \tfrac {\Gamma(2v+2\i\eps_2+j)\Gamma(u+v+\i(\eps_1+\eps_2)+j)\Gamma(u-v+\i (\eps_1-\eps_2)-j )}{\Gamma\left(u+v+\i (\eps_1+\eps_2)\right)\Gamma(2(v+\i\eps_2+j))\Gamma(-2(v+\i\eps_2+j))}
    \\=\tfrac{\left(j+v+i \eps _2\right) \left(2 \left(v+i \eps _2\right)\right)_j \Gamma \left(u-v+i \eps _1-i
   \eps _2\right) \left(u+v+i \eps _1+i \eps _2\right)_j}{j! \left(v+i \eps _2\right) \Gamma
   \left(-2 \left(v+i \eps _2\right)\right) \left(-u+v-i \eps _1+i \eps _2+1\right)_j} %
\end{multline*}
and
\begin{multline*}%
    M_k=-\tfrac{1}{\Gamma\left(u+v+\i (\eps_1+\eps_2)\right)} \\ \cdot  \Res_{ z=u+\i\eps_1+k}\left(\tfrac{\Gamma(v+\i\eps_2+z)\Gamma(v+\i \eps_2-z)\Gamma(u+\i\eps_1+z)\Gamma(u+\i \eps_1-z )}{\Gamma(2z)\Gamma(-2z)}
   \right)
    \\=-\tfrac{(-1)^k}{k!} \,\tfrac {\Gamma(2u+2\i\eps_1+k)\Gamma(u+v+\i(\eps_1+\eps_2)+k)\Gamma(v-u-\i (\eps_1-\eps_2)-k )}{\Gamma\left(u+v+\i (\eps_1+\eps_2)\right)\Gamma(2(u+\i\eps_1+k))\Gamma(-2(u+\i\eps_1+k))}
    \\= \tfrac{\left(k+u+i \eps _1\right) \left(2 \left(u+i \eps _1\right)\right)_k \Gamma \left(-u+v-i \eps _1+i
   \eps _2\right) \left(u+v+i \eps _1+i \eps _2\right)_k}{k! \left(u+i \eps _1\right) \Gamma
   \left(-2 \left(u+i \eps _1\right)\right) \left(u-v+i \eps _1-i \eps _2+1\right)_k}.  %
\end{multline*}
It is natural to set $m_j=0$ if $v+j>0$ and $M_k=0$ if $u+k>0$.

We note that under the assumptions of Proposition \ref{Prop2}, we have $u,v\ne 0$ and $v-u\not\in\ZZ$ if both $u,v<0$.
So we can take the limits, and we obtain
$\lim_{\eps_1,\eps_2\to 0}m_j=m_j(u,v)$, and
$\lim_{\eps_1,\eps_2\to 0}M_k=m_k(v,u)$, where $m_j(u,v)$ is given by \eqref{their-atoms}.
 Note that taking the limit is also justified when one of the parameters $u,v$ is positive. For example,
if $v<0$ but $u>0$, then %
$M_k=0$ and the
expression $v-u + j$ in the limit of the denominator of $m_j$ is nonzero, as  $v-u+j<v+j\leq 0$.

Lemma \ref{Lem:J-int}, see \eqref{4DCT}, shows that we can use dominated convergence theorem
and  take the limit under the integral sign in \eqref{I_N-eps}.
Thus, passing to the limit as $\eps_1,\eps_2\to 0$ in \eqref{I_N-eps} we obtain \eqref{I_N-x}.

\begin{remark}
    If $u+v\in\ZZ_{\leq 0}$ then the integral in  \eqref{I_N-x} does not %
    {  appear in the formula.}

    {By Remark \ref{Rem:J-int+}, formula \eqref{I_N-x}  extends to the case when one of the parameters $u,v$ is a negative integer while the real part of the other one is not in $\ZZ_{\leq 0}$.}
\end{remark}

\subsection{Extensions of  Proposition~\ref{Prop2}}  \label{Sec:Lims}
When the assumptions of Proposition~\ref{Prop2} are not satisfied, the formula depends on whether $u \le v $ or $ u \ge v $, and we focus on the latter case.

We first consider the case $u-v\in\ZZ$. %

 \begin{proposition}\label{CL:u-v=m} %
     Suppose $v\in(-\alpha,0)\setminus \ZZ$ and $m\in\ZZ_{\geq 0}$ is such $v+m<0$. Let  $u=v+m$.
 \begin{enumerate}[(A)]
  \item    If $2v\not\in\ZZ_{\leq 0}$ %
     then
         \begin{multline}\label{G(v+m,v)}
         \GG_N(\alpha,v+m,v)
        \\   = \tfrac{1}{
    2\pi \Gamma(2v+m)} \int_{0}^\infty \left|\Gamma(\alpha+\i x)\right|^{2N} \tfrac{|\Gamma(v+\i x)|^2|\Gamma(v+m +\i x )|^2}{|\Gamma(2\i x)|^2}\d x
    -2S_1+2S_2,
     \end{multline}
     where
     \begin{multline*}
         S_1=\tfrac{1}{ \Gamma
   (1-2 v)} \sum_{j=0}^{m-1} \tfrac{(-1)^{j}  (j+v)  (m-j-1)! (2 v)_j
   (m+2 v)_j}{j! }  \Gamma(\alpha +j+v)^N \Gamma (\alpha-j-v )^N
     \end{multline*}
     and
     \begin{multline*}
         S_2= {(-1)^m} \sum _{i=0}^{\lfloor -v\rfloor -m} \tfrac{(i+m+v) \Gamma
   (-m-2 v+1) }{i! (i+m)! \Gamma (-i-2 m-2 v+1) \Gamma
   (-i-m-2 v+1)}
   \\ \cdot  \Big[N(\psi
   (i+m+v+\alpha )-\psi (-i-m-v+\alpha
   ))+\frac{1}{i+m+v}+\psi (-i-2 m-2 v+1)\\ +\psi
   (-i-m-2 v+1)-\psi (i+m+1)-\psi
   (i+1)\Big]\Gamma (\alpha-i-m-v )^N \Gamma (\alpha+i+m+v)^N.
     \end{multline*}
  \item   If  $v\in(-\alpha,0)\setminus \ZZ$ is such that $2v\in\ZZ_{\leq 0}$, then  formula \eqref{G(v+m,v)} still holds, except that the integral term vanishes:
  \begin{equation}
      \label{G(v+m,v)-D}
       \GG_N(\alpha,v+m,v)=-2S_1+2S_2.
  \end{equation}
\end{enumerate}
 \end{proposition}
 \begin{proof}

 Clearly,  $v<u<0$ and $u\not\in\ZZ$.
     With $u+\eps=v+m+\eps$, for small enough $\eps>0$, we can apply \eqref{I_N-x} with $u$ replaced by $v+m+\eps$ and take the limit as $\eps\to 0^+$.
We omit the details.
\arxiv{
For completeness, we provide omitted argument. In view of \eqref{4DCT}, we can pass to the limit under the  integral sign.   If $2v$ is an odd integer, then the integral converges, but  the normalizing constant $1/\Gamma(2v+m+\eps)$ converges to $1/\Gamma(2v+m)=0$,   as  $2v+m\in \ZZ_{\leq 0}$. Thus   if $2v$ is an odd integer, the integral term vanishes in the limit.
(Note that under the assumptions of this proposition, $2v$ cannot be an even integer.)  %

     The sum over $j$ is from $0$ to $J:=\lfloor-v\rfloor$ and by assumption we have $J\geq m$.
     We split this sum into  the (possibly empty) sum over $j\in\{0,\dots,  m-1\}$  and then over $j\in\{m,\dots,J\}$.
   For $j\leq m-1$ , we compute the  contribution of each term by taking the limit of
     \begin{multline*}
         m_j(v+m+\eps,v)\,\left(\Gamma(\alpha+v+j)\Gamma(\alpha-v-j)\right)^N
         \\ =
        -2 \frac{(j+v) (2 v)_j \Gamma (m+\eps ) (m+2 v+\eps )_j}{ j!
   \Gamma (1-2 v) (-m-\eps +1)_j}\Gamma(\alpha+v+j)^N\Gamma(\alpha-v-j)^N
    \\=-2\frac{(-1)^{j} (j+v) (2 v)_j \Gamma (-j+m+\eps ) (m+2
   v+\eps )_j}{ j! \Gamma (1-2 v)}\Gamma(\alpha+v+j)^N\Gamma(\alpha-v-j)^N
  \\
  =-2\frac{(-1)^{j} (j+v) \Gamma (m-j) (2 v)_j \Gamma
   (j+m+2 v) \Gamma(\alpha +j+v)^N \Gamma (-j-v+\alpha )^N}{ j! \Gamma
   (1-2 v) \Gamma (m+2 v)}+O\left(\eps \right).
     \end{multline*}
(Here we used \eqref{InvPochh} to rewrite the expression before taking the limit.) This gives the first sum in \eqref{G(v+m,v)}.

     Since $i+u+\eps=i+v+m+\eps<0$, the sum over $i$ is from $0$ to $J-m$, provided $\eps>0$ is small enough so that $\eps<-J-v$.  To compute the limit as $\eps\to 0^+$, we pair each  term in the  sum over $i$ with
     the term $j=m+i$ in the sum over $j\in\{m,\dots, J\}$. We compute the  contribution of each such pair for $i\in\{0,\dots,J-m\}$ by taking the limit of
\begin{equation}
    \label{sum-ryba}
    m_{i+m}(u+\eps,v)\Gamma(\alpha+v+i+m)^N\Gamma(\alpha-v-i-m)^N +M_i(u+\eps,v) \Gamma(\alpha+u+\eps+i)^N\Gamma(\alpha-u-\eps-i)^N.
\end{equation}
  To simplify
   $$m_{i+m}(u+\eps,v)=-\frac{2 (i+m+v) \Gamma (-m-\eps +1) \Gamma (m+\eps )
   }{\Gamma (1-2 v)
    (i+m)! \, \Gamma (i-\eps +1)  } \cdot \frac{ \Gamma (i+m+2 v)}{\Gamma (2 v)} \cdot \frac{\Gamma (i+2 m+2 v+\eps )}{\Gamma (m+2
   v+\eps )},$$
   we use the identity
  \begin{equation}\label{EuRef1}
      \frac{\Gamma(j+2v)}{\Gamma(2v)}=(-1)^j \frac{\Gamma(1-2v)}{\Gamma(1-2v-j)}
  \end{equation}
   and
   \begin{equation}\label{EuRef2}
       \Gamma(m+\eps)\Gamma(1-m-\eps)=(-1)^m \frac{\pi}{\sin(\pi \eps)}.
   \end{equation}
   (Both are consequences  of   the  Euler's reflection formula $\Gamma(z)\Gamma(1-z)=\tfrac{\pi}{\sin(\pi z)}$ \cite[\href{http://dlmf.nist.gov/5.5.E3}{(5.5.3)}]{NIST:DLMF}.)
   We obtain
   \begin{multline}\label{m(i+m)}
       m_{i+m}(u+\eps,v)=-\frac{2 (i+m+v) (-1)^m \pi
   }{\Gamma (1-2 v)
    (i+m)! \, \Gamma (i-\eps +1)  \sin(\pi \eps) } \cdot  \frac{\Gamma(1-2v)}{\Gamma(1-2v-i-m)}\cdot \frac{\Gamma(1-m-2v-\eps)}{\Gamma(1-2v -2m -i-\eps)}
    \\=-(-1)^{m}\frac{2 (i+m+v)
   }{
    (i+m)! \, i!  } \cdot  \frac{1}{\Gamma(1-2v-i-m)}\cdot \frac{\Gamma(1-m-2v)}{\Gamma(1-2v -2m -i)} \left(\tfrac{1}{\eps}-O(\eps^3)\right)
    \\ \cdot  \left(1-\eps \psi(1-m-2v)+O(\eps^2)\right)\left(1+\eps \psi(i+1) +O(\eps^2)\right)\left(1+\eps \psi(1-2v-2m-i) +O(\eps^2)\right)
   \\=(-1)^{m}\tfrac{2 (i+m+v)
   }{\Gamma (1-2 v-m)
   }
   \tfrac{(-2v-m)_i}{i!} \tfrac{(-2v-m)_{i+m}}{(i+m)!} \left(-\tfrac{1}{\eps}-  \psi(i+1)+ \psi(1-m-2v) -\psi(1-2v-2m-i)  \right) +O(\eps).
   \end{multline}
   Here we used repeatedly
   \begin{equation}\label{1/eps}
       \left(  \tfrac{  1}{\eps} +a+O(\eps)\right)(1+\eps b+O(\eps^2)) =   \tfrac{1}{\eps} + a + b+O(\eps).
   \end{equation}
Next, we analyze the second term in \eqref{sum-ryba}. Recall that  $i\in\{0,\dots,J-m\}$. Using Euler's reflection based identities as before, we have
  \begin{multline}\label{Mi}
      m_i(v,u+\eps)  = \frac{\Gamma (-m-\eps ) (i+m+v+\eps ) (2 (m+v+\eps
   ))_i (m+2 v+\eps )_i}{i! (m+v+\eps ) (m+\eps +1)_i
   \Gamma (-2 (m+v+\eps ))}
   \\= (-2)\frac{\Gamma (-m-\eps ) \Gamma (m+\eps +1)
   (i+m+v+\eps ) \Gamma (i+m+2 v+\eps ) \Gamma (i+2
   (m+v+\eps ))}{i!   \Gamma (i+m+\eps +1)
   \Gamma (1-2 (m+v+\eps )) \Gamma (2 (m+v+\eps )) \Gamma
   (m+2 v+\eps )}
   \\=\frac{2 (-1)^m (i+m+v+\eps )}{i!   \Gamma (i+m+\eps +1) \Gamma (1-2 (m+v+\eps ))} \cdot  \frac{ \pi }{\sin (\pi \eps)}
   \cdot \frac{\Gamma (i+m+2 v+\eps )}{\Gamma
   (m+2 v+\eps )} \cdot \frac{ \Gamma (i+2
   (m+v+\eps ))}{ \Gamma (2 (m+v+\eps) )}
   \\=\frac{2 (-1)^m (i+m+v+\eps )}{ i!\, \Gamma(i+m+\eps )} \cdot  \frac{ \pi }{\sin (\pi \eps)}
   \cdot
\frac{1}{\Gamma (1-2 (m+v+\eps )-i)} \frac{\Gamma
   (1-2v-m- \eps )}{\Gamma
   (1-2v-m- \eps -i)}
   \\= \frac{2 (-1)^m (i+m+v  )}{ i!(i+m)!}
   \cdot
\frac{1}{\Gamma (1-2 (m+v)-i)} \frac{\Gamma
   (1-2v-m )}{\Gamma
   (1-2v-m  -i)}
   \\ \cdot  \left(1+\tfrac{\eps}{i+m+v}\right)\left(1-\eps\psi(1+i+m)+O(\eps^2)\right)\left(\tfrac{1}{\eps}+O(\eps^3)\right)
   \left(1+2\eps \psi(1-2v-2m-i) +O(\eps^2)\right)  \\\cdot  \left(1+\eps \psi(1-2v-m-i) +O(\eps^2)\right)
   \left(1-\eps\psi(1-2v-m)+O(\eps^2)\right)
  \\ = (-1)^m\frac{2 (i+m+v  )}{  \Gamma(1-2v-m)} \frac{(-2v-m)_i}{i!}\frac{(-2v-m)_{m+i}}{(m+i)!}
  \big(\tfrac{1}{\eps}-\psi(1-2v-m)+\tfrac{1}{i+m+v} -\psi(1+i+m)
  \\+2\psi(1-2v-2m-i)+ \psi(1-2v-m-i) \big)+O(\eps).
  \end{multline}
  Since
  \begin{multline*}
  \Gamma(\alpha+u+\eps+i)^N\Gamma(\alpha-u-\eps-i)^N=\Gamma(\alpha+v+m+\eps+i)^N\Gamma(\alpha-v-m-\eps-i)^N
  \\ = \Gamma(\alpha+v+m)^N\Gamma(\alpha-v-m-i)^N\left(1+ \eps N \left[\psi (\alpha+v+m+\eps+i)-\psi(\alpha-v-m-\eps-i)\right]\right)+O(\eps^2),
  \end{multline*}
  combining this with \eqref{m(i+m)} and \eqref{Mi}, after cancelling the singular terms   in the sum of  \eqref{sum-ryba}, we can  take the limit as $\eps\to 0^+$. We obtain
  the second sum in \eqref{G(v+m,v)}.
  }
 \end{proof}

  \begin{corollary} \label{Prop2-sing}
      Assume $v=u<0$ and $\alpha+v>0$.  If $2v\not\in\ZZ_{\leq 0}$, then %
       \begin{multline}\label{I_N-x-sing-alt}
     \GG_N(\alpha,v,v)=
    \tfrac{1}{
    2\pi \Gamma(2v)}\int_{0}^\infty \left|\Gamma(\alpha+\i x)\right|^{2N}\tfrac{|\Gamma(v+\i x )|^4}{ |\Gamma(2\i x)|^2}\d x
    \\+
    \tfrac{2}{{ \Gamma(1-2v)}}\,\sum_{\{j\geq 0:\, j+v<0\}} \,(j+v) \tfrac{(-2v)_j^2}{j!^2} \, \left[\Gamma (-j-v+\alpha ) \Gamma
   (j+v+\alpha )\right]^N   \\ \cdot  \Big(N \left[\psi
    (j+v+\alpha )-\psi  (-j-v+\alpha )\right]+2 \left[\psi(1-j-2 v)- \psi  (j+1)\right]+\tfrac1{j+v}\Big).
    \end{multline}
  \end{corollary}

\begin{proof}
    The formula follows from \eqref{G(v+m,v)} with $m=0$.
\end{proof}
We remark that if $2v$ is a negative  odd integer then the same formula holds, but without the integral term. However if
$2v$ is an even integer then the integral in \eqref{I_N-x-sing-alt} is replaced by an additional term, see \eqref{I_N-x-sing+}.

Next,   we consider the case $u,v\in\ZZ_{\leq 0}$.
 \begin{proposition} \label{Cl:uv=-m,-k} %
 Fix two non-negative integers $m,k\in\ZZ_{\geq 0}$  such that $m\le k<\alpha$. Then with $u=-m$, $v=-k$ we have
  \begin{multline}\label{G(-m,-k)}
     \GG_N(\alpha,-m,-k)=  (-1)^{k+m} \tfrac{(k+m)!}{k!^2m!^2} \Gamma (\alpha )^{2 N}\\+
     \sum_{j=0}^{k-m-1}
   (-1)^j {\binom{m+k}{j}}\, \tfrac{ 2 (k-j)(k-m-j-1)!}{(2k-j)!} \Gamma (\alpha +j-k)^N \Gamma (\alpha -j+k )^N
       \\+
      (-1)^{k+m} \sum_{i=0}^{m-1} {\binom{m+k}{i}}\
     \tfrac{2 (m- i)} {(2m- i)!  ( i+k-m)!  }
   \Big[N\left(\psi(\alpha+m-i)-\psi(\alpha+i-m)\right)
    \\ +\psi(i+1)+\psi(1+i+k-m)- \psi(1+k+m-i) -\psi(1+2m-i) +\tfrac{1}{m-i}\Big] \\ \cdot  \Gamma (\alpha+ i-m )^N \Gamma (\alpha- i+m)^N.
 \end{multline}
\end{proposition}

 \begin{proof}We apply \eqref{I_N-x} with $u=-m-\eps$ and $v=-k+\eps$ where $\eps\in(0,1/2)$ and take a limit as $\eps\to 0$.
Contributions to the limit may come from the integral and from the discrete part.  We omit the details.
 \arxiv{ For completeness, we provide omitted argument.

We first remark that the  integral term does not contribute. By Lemma \ref{Lem:J-int},  for $\eps\in(0,1/2)$  the integral is finite, as then  $u=-m-\eps, v=-k+\eps\not \in \ZZ_{\leq 0}$. Since $z\mapsto 1/\Gamma(z)$ is analytic on the  complex plane $\CC$,  and $1/\Gamma(-k)=0$,  with the normalizing constant  $1/\Gamma(u+v)=1/\Gamma(-k-m)=0$, the  contribution from the integral term vanishes.

Since $\eps\in(0,1/2)$,  in \eqref{I_N-x} parameter $j$ varies from $0$ to $k$ and parameter $i$ varies from $0$ to $m-1$. We split  the sum over $j$, into the three ranges:  $j=0,\dots,k-m-1$ and $j=m-k,\dots, k-1$, and then we consider separately $j=k$.
For $j\leq k-m-1$,  we compute the limit of $m_j(u,v)\Gamma(\alpha+v+j)^N\Gamma(\alpha-v-j)^N$
by noting that
\begin{multline*}
   m_j(u,v)=\frac{(j-k+\eps ) (-k-m)_j (2 (\eps -k))_j \Gamma (k-m-2\eps ) }{j! (\eps -k) \Gamma (-2(\eps -k)) (-k+m+2 \eps +1)_j}\to\frac{(k-j)(-k-m)_j(-2k)_j(k-m-1)!}{j!k(2k-1)!(-k+m+1)_j}
   \\=2(-1)^j (k-j)\tbinom{m+k}{j} \tfrac{(k-m-j-1)!}{(2k-j)!}.
\end{multline*}
This gives the first sum in \eqref{G(-m,-k)}.

For  $i =0,\dots,m-1$, we pair every term in the sum over $i$ with the term $j=k-m+i$ in the sum over $j$. Together with the previous sum, this
covers $j=0,\dots,m-k,\dots,k-1 $. (We leave $j=k$ for later.)    So we  need to compute the limit as $\eps\to 0$ of  the following
\begin{multline}\label{m+M2}
    m_{k-m+i}(u,v)\Gamma(\alpha+v+k-m+i)^N\Gamma(\alpha-v-k+m-i)^N+M_{i}(u,v)\Gamma(\alpha+u+i)^N\Gamma(\alpha-u-i)^N
     \\
    = m_{k-m+i}(-m-\eps,-k+\eps)\Gamma (\alpha+i-m +\eps )^N \Gamma (\alpha-i+m
    -\eps )^N
    \\+M_{i}(-m-\eps,-k+\eps)\Gamma (\alpha+i-m -\eps )^N \Gamma (\alpha-i+m
    +\eps )^N.
\end{multline}
As previously,
\begin{multline}\label{GammaN-GammaN}
  \Gamma (\alpha+i-m \pm \eps )^N \Gamma (\alpha-i+m
   \mp \eps )^N
   \\=\Gamma (\alpha+i-m )^N \Gamma (\alpha-i+m )^N
   \left(1\mp N \eps \left[\psi (\alpha+m-i )-\psi (\alpha+i-m ) \right] \right)+O(\eps^2).
\end{multline}
As previously, we use Euler's reflection formulas \eqref{EuRef1} and \eqref{EuRef2} and also \eqref{InvPochh}  to rewrite the expressions,
we use the  low-order expansions { for the factors} and we use \eqref{1/eps}. We obtain
\begin{multline*}
  m_{k-m+i}(-m-\eps,-k+\eps)
  =  \frac{-2(i-m+\eps ) (-k-m)_{i+k-m} \Gamma (k-m-2 \eps ) (2
   (\eps -k))_{i+k-m}}{  \Gamma (1-2 (\eps -k))
   (i+k-m)! (-k+m+2 \eps +1)_{i+k-m}}
   \\ =\frac{(-1)^{k-m} (i-m+\eps ) (m+k)!}{(i+k-m)!\,(2m-i)!}
   \frac{ 2\pi  }{   \Gamma (i+2
   \eps +1) \sin(2\pi
   \eps  )
   }\frac{ 1}{\Gamma (1+k+m-i -2\eps)}
   \\ =-\frac{(-1)^{k-m} (m-i ) (m+k)!}{(i+k-m)!(2m-i)!i!(k+m-i)!}
  \Big(\frac{1}{\eps}- \frac{1}{m-i} -2\psi(1+i)+2\psi(1+k +m-i)\Big)+ O(\eps).
\end{multline*}
Thus, from \eqref{GammaN-GammaN} we obtain  the first term in \eqref{m+M2}, which becomes
\begin{multline}\label{2.30a}
\frac{(-1)^{k-m}(m-i)}{(m+k)!}\binom{m+k}{i}\binom{m+k}{2m-i} \Big( -\frac{1}{\eps}+ \frac{1}{m-i} +2\psi(1+i)-2\psi(1+k+m-i)
\\+N   \left[\psi (\alpha+i-m ) -\psi (\alpha-i+m )\right] \Big).
\end{multline}

Similarly, for $i\in\{0,\dots,m-1\}$   we obtain
\begin{multline*}
  m_{i}(-k+\eps,-m-\eps) = -2\frac{(i-m-\eps ) (-k-m)_i (-2m-2\eps )_i \Gamma (-k+m+2
   \eps )}{i! \Gamma (1-2 (-m-\eps )) (k-m-2
   \eps +1)_i}
   \\=
  2 \binom{k+m}{i} \frac{ (m-i+\eps )   \Gamma (k-m-2 \eps +1) \Gamma (-k+m+2
   \eps )}{
     \Gamma (2m -i+1+2 \eps) \Gamma (i+k-m-2
   \eps +1)}
    \\=
   (-1)^{k-m}\binom{k+m}{i} \frac{ (m-i+\eps )  }{
     \Gamma (2m -i+1+2 \eps) \Gamma (i+k-m-2
   \eps +1)} \frac{2\pi}{\sin(2 \pi\eps)}
  \\=\frac{(-1)^{k-m}(m-i )}{(m+k)!}\binom{k+m}{i} \binom{k+m}{2m-i} \left(\tfrac{1}{\eps}+2 \psi(1+i+k-m)-2\psi(1+2m-i)+\tfrac{1}{m-i}\right)+O(\eps).
\end{multline*}
In view of \eqref{GammaN-GammaN},  the second term in \eqref{m+M2} is
\begin{multline}\label{2.30b}
 \frac{(-1)^{k-m}(m-i)}{(m+k)!}\binom{m+k}{i}\binom{m+k}{i+k-m} \Big(\frac{1}{\eps}+ \frac{1}{m-i}+
2\psi(1+i+k-m)-2\psi(1+2m-i)
\\+N   \left[\psi (\alpha+i-m ) -\psi (\alpha-i+m )\right] \Big).
\end{multline}
In  the sum of  expressions \eqref{2.30a} and \eqref{2.30b},
  the terms at $1/\eps$ cancel out.  Passing to the limit as $\eps\to 0$,  we obtain the second sum in \eqref{G(-m,-k)}.

It remains to consider the last unpaired value $i=k$. We need to compute the limit as $\eps\to 0$ of %
\begin{multline*}
      m_k(u,v)\Gamma(\alpha -\eps)^{N} \Gamma(\alpha +\eps)^{N} =    \frac{-2\eps  (-k-m)_k (-2k+ 2\eps)_k \Gamma (k-m-2 \eps ) \Gamma
   (\alpha -\eps )^N \Gamma (\alpha +\eps )^N}{k!   \Gamma
   (1+2k-2 \eps) (-k+m+2 \eps +1)_k}
   \\=\frac{(-2 \eps)  (k+m)!\, \Gamma (k-m-2 \eps ) \Gamma (\alpha
   -\eps )^N \Gamma (\alpha +\eps )^N}{k!\, m!\, \Gamma (k-2
   \eps +1) (-k+m+2 \eps +1)_k}.
\end{multline*}

Since
$$
  (-k+m+2 \eps +1)_k= (-k+m+2 \eps +1)_{k-m} (1+2\eps)_m=(-1)^{k-m} \frac{\Gamma(k-m-2\eps)}{\Gamma(-2\eps)} \cdot \frac{\Gamma(1+m+2\eps)}{\Gamma(1+2\eps)},
$$
where we used \eqref{InvPochh}, and since $(-2\eps)\Gamma(-2\eps)\to 1$, we obtain
$$  m_k(u,v)\Gamma(\alpha -\eps)^{N} \Gamma(\alpha +\eps)^{N}
  = (-1)^{k-m} \frac{(k+m)!}{k!^2m!^2} \Gamma (\alpha )^{2 N} +O(\eps).$$
  This gives the first term in \eqref{G(-m,-k)}.
  }
 \end{proof}

  \begin{corollary} \label{Prop2-sing+}
      Assume $v=u<0$ and $\alpha+v>0$.  If {  $2v=-k  \in   \ZZ_{\leq -1}$,} then
   \begin{multline}\label{I_N-x-sing+}
     \GG_N(\alpha,v,v)
 =   \tfrac{1}{k!}\binom{k}{k/2}^2\Gamma (\alpha )^{2 N}\; \mathbf{ 1}_{k/2\in \ZZ_{\geq 1} }
 \\+
       \sum_{{j\in[0,k/2)\cap\mathbb Z}}   \tfrac{k-2j}{k!}\binom{k}{j}^2
       \left(\Gamma
   \left(\alpha+j-\tfrac{k}{2} \right)  \Gamma
   \left(\alpha-j+\tfrac{k}{2}\right)\right)^N
   \\ \cdot  \Big(
    N \left[\psi
   \left(\alpha-j+\tfrac{k}{2} \right)-\psi
   \left(\alpha+j-\tfrac{k}{2} \right)\right]
   \\ +  2\left[\psi(1+j)-\psi(1+k-j)+\tfrac{1}{k-2 j}\right]
   \Big).
   \end{multline}
  \end{corollary}

\begin{proof}
    If $k/2\not\in\ZZ_{\geq 1}$, this is  formula  \eqref{G(v+m,v)-D} with $m=0$.
     If $k/2\in\ZZ_{\geq 1}$, this is  formula  \eqref{G(-m,-k)} with  both parameters there equal and
     assigned the value $k/2$.
\end{proof}

 There are several other cases to consider. To save space, we point out just one more  case that is not covered by Proposition \ref{Prop2}.
 \begin{claim}\label{Cl:u=-m} %
     Suppose $v\in(-\alpha,0)\setminus \ZZ$ and $u=-m$, where $m\in\ZZ_{\geq 0}$ is such $v+m<0$. Then \eqref{I_N-x} holds, with the simplified integrand.
    That is,
\begin{multline*}
   \GG_N(\alpha,-m,v)= \tfrac{2}{
    \pi \Gamma(v-m)} \int_{0}^\infty \left|\Gamma(\alpha+\i x)\right|^{2N} |\Gamma(v+\i x)|^2\tfrac{\cosh(\pi x)}{\prod_{k=1}^m(k^2+x^2)} \d x\\
     +\tfrac{2\Gamma (v+m) }{\Gamma (1+2 m)}\sum _{i=0}^{m-1} \tfrac{(m-i) (-2 m)_i (v-m)_i }{
   i!  (-m-v+1)_i}\Gamma
   (\alpha+i-m )^{N} \Gamma (\alpha-i+m)^{N}
    \\ -
  \tfrac{2\Gamma (-m-v)}{\Gamma (1-2 v) }\sum_{j\in \ZZ\cap [0,-v)}\tfrac{(j+v) (2 v)_j  (v-m)_j }{
   j! (m+v+1)_j} \Gamma (\alpha+v+j )^{N} \Gamma(\alpha-v-j )^{N} .
\end{multline*}
 \end{claim}
 \arxiv{  \begin{proof}
     In this case, Lemma \ref{Lem:J-int} does not apply, but the convergence of the integral is ensured by Remark \ref{Rem:J-int+}, with the %
   {  integrand rewritten using}   Euler's reflection formula as %
   in \eqref{sinh+G-u}. The details are omitted.
 \end{proof}
 }

\section{Alternative proof of Theorem \ref{Thm:PhD}}\label{Sec:AltProof}
In this section we present  an alternative   proof  of Theorem \ref{Thm:PhD}
 based on  Proposition \ref{Prop2} %
 and its extensions.
\subsection{Analytic proof of Theorem \ref{Thm:PhD}(i)}
{
We present only a formal argument, without deriving an $N$-uniform bound on the remainder, and omit the boundary cases $u=0$ and $v=0$, which require additional modifications. The key Lemma~\ref{Lem:Extr} is also used in another proof.
}

    From \eqref{LG-xx}  and $t\in(-v,u)$  we obtain
\begin{equation*}\label{logE}
   \log \EE[e^{-2 t \vv L_N/N}]=\log \GG_N(\alpha+t/N,u-t/N,v+t/N)- \log \GG_N(\alpha,u,v) .
\end{equation*}
Thus in view of Lemma \ref{Lem:I-anal},   we obtain linear approximation %
\begin{equation*}
   \log \EE[e^{-2 t \vv L_N/N}] =
   \tfrac{t}{N{\GG_N(\alpha,u,v)}}  \left( \tfrac{\partial\GG_N}{\partial \alpha}-\tfrac{\partial\GG_N}{\partial u}+\tfrac{\partial\GG_N}{\partial v} \right) + o(t/N).
   \end{equation*}
Since  $t/N\in (-v,u)$,  the discrete sums $\mathbb{D}_N$ in \eqref{I_N-x} vanish and the linear approximation  becomes
    \begin{equation}\label{D[I]}
     \log \EE[e^{-2 t \vv L_N/N}]\sim
   \tfrac{t}{N{\II_N(\alpha,u,v)}}  \left( \tfrac{\partial\II_N}{\partial \alpha}-\tfrac{\partial\II_N}{\partial u}+\tfrac{\partial\II_N}{\partial v} \right)
\mbox{ as $N\to\infty$},
   \end{equation}
where $\II_N(\alpha,u,v)$ is given by the integral \eqref{JJ-1}.

Applying the linear approximation \eqref{D[I]} and differentiating under the integral sign, we see that the dominant term is
\begin{equation}
  \label{dominant}   \log \EE[e^{-2 t \vv L_N/N}]\sim 2t \tfrac{\int_0^\infty  \re(\psi(\alpha+\i x)) |\Gamma(\alpha+ \i x)|^{2N}
 |\Gamma(u+\i x)\Gamma(v+\i x)|^2|\Gamma(2\i x)|^{-2}\d x}
  {\int_0^\infty  |\Gamma(\alpha+ \i x)|^{2N}
 |\Gamma(u+\i x)\Gamma(v+\i x)|^2|\Gamma(2\i x)|^{-2}\d x},
\end{equation}
where we used \eqref{digamma}, i.e., the fact that for real $\alpha,  x$ we have
\begin{multline*}
\tfrac{\partial |\Gamma(\alpha +\i x)|^{2N}}{\partial \alpha}
=\tfrac{\partial \Gamma(\alpha +\i x)^N\Gamma(\alpha -\i x)^N}{\partial \alpha}
\\=N\left(\psi(\alpha +\i x)+\psi(\alpha -\i x) \right)|\Gamma(\alpha +\i x)|^{2N}
=2N \re\psi(\alpha +\i x)|\Gamma(\alpha +\i x)|^{2N}.
\end{multline*}
\arxiv{
Differentiation under the integral sign is justified by $|\psi(\alpha+\i x)|\sim  \tfrac12\log (\alpha^2+x^2)$ as $x\to \infty$
and bound \eqref{4DCT}.
}
To end the proof, we verify that for $t\in(-v,u)$ the Laplace transform converges:
\begin{equation}
    \label{Conclusion}
    \log \EE[e^{-2 t \vv L_N/N}]\to 2t \psi(\alpha).
\end{equation}
Let $\Psi(x):= \re(\psi(\alpha+\i x))$. Since $\phi(x):=|\Gamma(\alpha+\i x)\Gamma( u+\i x)\Gamma( v+\i x)|^2/|\Gamma(2\i x)|^2$ is integrable and positive on $(0,\infty)$, and $\Psi \phi$ is integrable,  \eqref{Conclusion} follows from \eqref{dominant}
 and  the following lemma,  which for $\phi$ continuous at $0$ is a consequence of the fact that
 probability measures $|\Gamma(\alpha+ \i x)|^{2N}/\int_0^\infty    |\Gamma(\alpha+ \i x)|^{2N}
  \d x$ converge weakly to $\delta_0(\d x)$,
\begin{lemma}\label{Lem:Extr}
   Suppose that  $\Psi:[0,\infty)\to\RR$ is continuous at $x=0$, $\Psi \phi$ and $\phi$ are integrable on $(0,\infty)$, $\phi\geq 0$ and $\phi(x)>0$ near $x=0$. Then
   \begin{equation}\label{Psi(0)}
     \lim_{N\to\infty} \tfrac{\int_0^\infty  \Psi(x) |\Gamma(\alpha+ \i x)|^{2N}
\phi(x)\d x}
  {\int_0^\infty    |\Gamma(\alpha+ \i x)|^{2N}
\phi(x)\d x}= \Psi(0).
   \end{equation}
\end{lemma}

\arxiv{
\begin{proof}[Proof of Lemma \ref{Lem:Extr}] Subtracting $\Psi(0)$ from both sides of \eqref{Psi(0)}, without loss of generality we may assume $\Psi(0)=0$. Given $\eps>0$, let $\delta>0$ be such that $|\Psi(x)|<\eps $ for $x\in[0,\delta)$.
Denote
$$\mathsf A_N= \int_0^\delta  \Psi(x) |\Gamma(\alpha+ \i x)|^{2N}
\phi(x)\d x, \, \quad \mathsf B_N= \int_\delta^\infty  \Psi(x) |\Gamma(\alpha+ \i x)|^{2N}
\phi(x)\d x, $$
$$
\mathsf C_N= \int_0^\delta    |\Gamma(\alpha+ \i x)|^{2N}
\phi(x)\d x, \quad \mathsf D_N=\int_\delta^\infty    |\Gamma(\alpha+ \i x)|^{2N}
\phi(x)\d x.
$$
Noting that $\mathsf C_N>0$, and $\mathsf D_N\geq 0$, we obtain
$$
\left|\frac{\int_0^\infty  \Psi(x) |\Gamma(\alpha+ \i x)|^{2N}
\phi(x)\d x}
   {\int_0^\infty    |\Gamma(\alpha+ \i x)|^{2N}
\phi(x)\d x}\right|\leq \frac{|\mathsf A_N|+|\mathsf B_N|}{\mathsf C_N+\mathsf D_N}\leq \frac{|\mathsf A_N|}{\mathsf C_N}+\frac{|\mathsf B_N|}{\mathsf C_N}\leq \eps+\frac{|\mathsf B_N|}{\mathsf C_N}.
$$
From \eqref{G-prod}, we see that $x\mapsto |\Gamma(\alpha+\i x)|^2$ is a strictly decreasing function.
Therefore
$$|\mathsf B_N|\leq C |\Gamma(\alpha+\i \delta)|^{2N},$$
where $C=\int_0^\infty |\Psi(x)|\phi(x)\d x<\infty$. %
Similarly,
$$\mathsf C_N\ge \int_0^{\delta/2}  |\Gamma(\alpha+ \i x)|^{2N}
\phi(x)\d x \geq c |\Gamma(\alpha+ \i \tfrac{\delta}{2})|^{2N},$$
where $c=\int_0^{\delta/2} \phi(x)\d x>0$. Therefore,
$$\frac{|\mathsf B_N|}{\mathsf C_N}\leq  \frac{C}{c} \left(\frac{|\Gamma(\alpha+\i \delta)|}{|\Gamma(\alpha+\i \tfrac{\delta}2)|}\right)^{2N} \to 0. $$
Since $\eps>0$ is arbitrary, this  proves \eqref{Psi(0)}.
\end{proof}
}
\subsection{Analytic Proof of Theorem \ref{Thm:PhD}(ii)}

We first consider the proof for $u,v$ that satisfy the assumptions of Proposition \ref{Prop2}.
That is, we assume $u>v$, $-\alpha<v<0$, $ v\not \in \ZZ_{\leq 0}$  and if $u<0$ we also assume that $u\not\in\ZZ_{\leq 0}$ and
$u-v\not\in\ZZ$.
In view of Remark \ref{Rem:J-int+}, the same proof works    if  $u\in\ZZ_{\leq 0}$ while $v\not\in\ZZ_{\leq 0}$ or if $v\in\ZZ_{\leq 0}$ while $u\not \in\ZZ_{\leq 0}$.

\begin{claim}\label{Claim 2} Suppose that $v<0$ and $u>v$. Then for fixed  $t\ne 0$, we have
\begin{multline}\label{C2-m0}
 \GG_N(\alpha+t/N,u-t/N,v+t/N) \\ \sim m_0(u,v)  \Gamma(\alpha+ 2 t/N+v)^N \Gamma(\alpha-v)^N \quad \mbox{as $N\to\infty$},
\end{multline}
where $m_0(u,v)={\Gamma (u-v)}/{\Gamma (-2 v)}>0$ is given by \eqref{their-atoms} with $j=0$.

Furthermore,  \eqref{C2-m0} holds also for $t=0$ if $u-v\not\in\ZZ$.
\end{claim}

\begin{proof}   We have $u-t/N<v+t/N$ and for large enough $N$ their difference is not an integer, so we can use \eqref{I_N-x} with $\alpha, u,v$ replaced by $\alpha+t/N$, $u-t/N$ and $v+t/N$.
Note that if $v-u\not\in\ZZ_{\leq0}$ or if $u>0$ then  for $N$ large enough
$$m_0(u-t/N,v+t/N)= \tfrac{\Gamma (u-v-2t/N)}{\Gamma (-2 v-2t/N)}\to \tfrac{\Gamma (u-v)}{\Gamma (-2 v)}=m_0(u,v)>0$$
is well defined. %
So the right-hand side of \eqref{C2-m0} arises from the first term $j=0$ of the discrete sum $\mathbb{D}_N(\alpha+t/n,u-t/N,v+t/N)$ in \eqref{I_N-x}, and we need to verify that this is
the dominant contribution.

We first confirm that  the contribution from the integral  term  in \eqref{I_N-x}
is of lower order. Indeed,
since $|\Gamma(\alpha+t/N+\i x)|\leq \Gamma(\alpha+t/N)$,  the contribution of the integral is bounded by
\begin{multline*}
  \GG_1(\alpha+t/N,u-t/N,v+t/N) \Gamma(\alpha+t/N)^{2N-2}
  \\ =\Gamma(\alpha+u)\Gamma(\alpha+v+2t/N) \Gamma(\alpha+t/N)^{2N-2}.
\end{multline*}
We observe that there is $\delta>0$ such that for $t/N<\delta$ this  contribution is of lower order than the right-hand side of \eqref{C2-m0}. Indeed, function
\begin{equation}\label{x-function} x\mapsto \tfrac{\Gamma(\alpha+x)^2}{\Gamma(\alpha+ 2 x +v ) \Gamma(\alpha-v)}
\end{equation}
is continuous at $x=0$ and by \eqref{C-S} its value at $x=0$ is smaller than $1$.

It remains to verify that the other terms in \eqref{I_N-x} also have a lower contribution.

By the second part of Lemma \ref{Lem:prop-Gamma}(iv), for $j\ge 1$ such that $v+t/N+j<0$ we have
\begin{multline*}
    \Gamma(\alpha+ 2 \tfrac tN+v+j) \Gamma(\alpha-v-j)\leq
\Gamma(\alpha+ 2 \tfrac tN+v+1) \Gamma(\alpha-v-1)
\\= \tfrac{\alpha +2t/N+v}{\alpha -v-1} \Gamma(\alpha+ 2 \tfrac tN+v) \Gamma(\alpha-v).
\end{multline*}
For $N$ large enough, $\tfrac{\alpha +2t/N+v}{\alpha -v-1}$ is arbitrarily close to $\tfrac{\alpha  +v}{\alpha -v-1}$.
Note that  in order for the second atom to exist,
w must have $-\alpha<v<-1$ so $\alpha-(v+1)>\alpha>0$ and $\alpha+v>0$.
Thus
$$0<\tfrac{\alpha +v}{\alpha   -(v+1)}< \tfrac{\alpha +v}{\alpha  -v}<\tfrac{\alpha-1}{\alpha+1}:=q<1$$
and we see that the contributions of higher atoms are negligible.

Similar reasoning shows that the terms in the discrete sum $\mathbb{D}_N(\alpha+t/N,v+t/N,u-t/N)$  in  \eqref{I_N-x} %
have geometrically lower contribution than \eqref{C2-m0}.
\end{proof}

 We are now ready to prove Theorem \ref{Thm:PhD}(ii)
  under additional conditions that $ v\not \in \ZZ_{\leq 0}$  and if $u<0$ we also assume that $u\not\in\ZZ_{\leq 0}$ and
$u-v\not\in\ZZ$.
Applying  \eqref{C2-m0} to the numerator and the denominator in \eqref{LG-xx},
 we obtain
\begin{multline*}
   \lim_{N\to\infty} \log \EE[e^{-2 t \vv L_N/N}]
   =\lim_{N\to\infty} \log \tfrac{ \GG_N(\alpha+t/N,u-t/N,v+t/N)}{ \GG_N(\alpha,u,v)}
   \\
 =
  \lim_{N\to \infty}  N \log \tfrac{\Gamma(\alpha+2 t/N+v)}{\Gamma(\alpha+v)}
=
2t \lim_{N\to \infty}  \tfrac{\log \Gamma(\alpha+v+2 t/N)-\log \Gamma(\alpha+v)}{2 t /N}
\\ =2 t \psi(\alpha+v).
\end{multline*}

This argument covers also the case $u\in\ZZ_{\leq 0}$   while $v\in (-\alpha,0)\setminus\ZZ$, where we invoke  Claim \ref{Cl:u=-m}.
Similar reasoning also applies if $v\in\ZZ_{\leq 0}$ while $u\not\in\ZZ_{\leq 0}$, based on an analogous version of   Claim \ref{Cl:u=-m}.

\medskip
There remain two more cases, where  the argument needs modifications.

The first case is  $ v=-k\in\ZZ_{\leq -1}$, $u=-m  \in \ZZ_{\leq 0} $ and $m<k$. Then we use Proposition  \ref{Cl:uv=-m,-k}.
The dominant contribution to the denominator in \eqref{LG-xx} is again the term corresponding to $j=0$, given by
$$\tfrac{(k-m-1)!}{(2k-1)!} \Gamma(\alpha+k)^N\Gamma(\alpha-k)^N.$$
With $t\ne 0$, in  the numerator we  use \eqref{I_N-x} with $\alpha,u,v$ replaced by $\alpha+t/N$, $u=-m-t/N$ and $v=-k+t/N$, and again it is clear that for large enough $N$ only the term with $j=0$ contributes.
Thus,   as $N\to\infty$,
\begin{multline*}
    \log \EE[e^{-2 t \vv L_N/N}]
   =  \log \tfrac{ \GG_N(\alpha+t/N,u-t/N,v+t/N)}{ \GG_N(\alpha,u,v)}
   \\
  \sim \log\tfrac{m_0(-m-t/N,-k+t/N)}{2k (k-m-1)!/{(2k)!}}+N \log \tfrac{\Gamma(\alpha+2t/N-k)\Gamma(\alpha +k)}{ \Gamma(\alpha+k) \Gamma(\alpha-k) }
  \\=\log\tfrac{m_0(-m-t/N,-k+t/N)}{2k (k-m-1)!/{(2k)!}}+2t \tfrac{\log \Gamma(\alpha+2t/N-k) - \log \Gamma(\alpha-k)}{2t/N}
  \to 2t \psi(\alpha+v).
   \end{multline*}
   Here, we used  the fact, established in the proof of Proposition  \ref{Cl:uv=-m,-k}, that with $k\geq 1$, we have
   $$\lim_{\eps\to 0} m_0(-m-\eps,-k+\eps)=\tfrac{(k-m-1)!}{(2k-1)!}.$$

  The second case is $ v\not \in \ZZ_{\leq 0}$  and $u<0$ but $v-u\in\ZZ_{\leq -1}$ (recall that $u>v$).
 To analyze the  numerator in  \eqref{LG-xx}, we use  \eqref{I_N-x}  with $\alpha,u,v$ replaced by $\alpha+t/N,u-t/N,v+t/N$.  As in the previous case, the dominating terms come from $j=0$. To analyze the denominator in \eqref{LG-xx} we apply
  Proposition \ref{CL:u-v=m}. The argument used in the proof of
Proposition \ref{CL:u-v=m}  can be modified to show that for $m\in\ZZ_{\geq 1}$ we have
$\lim_{\eps\to 0}m_0(v+m-\eps,v+\eps) = -2v(m-1)!$. The later appears in \eqref{G(v+m,v)-D} or \eqref{G(v+m,v)} for $j=0$.
We omit the details.

\arxiv
{To summarize, we have split this proof into the following five cases:

\begin{enumerate}[($A$)]
    \item $u>v$, $-\alpha<v<0$, $ v\not \in \ZZ_{\leq 0}$  and if $u<0$ we also assume that $u\not\in\ZZ_{\leq 0}$ and
$u-v\not\in\ZZ$
\item[($A_1$)]  $u>v$, $-\alpha<v<0$, $u\in\ZZ_{\leq 0}$ while $v\not\in\ZZ_{\leq 0}$
\item[($A_2$)]  $u>v$, $-\alpha<v<0$, $v \in\ZZ_{\leq 0}$ while $u\not\in\ZZ_{\leq 0}$
\item $u>v$, $-\alpha<v<0$, $u\in\ZZ_{\leq 0}$, $v\in\ZZ_{\leq -1}$
\item $u>v$, $-\alpha<v<0$, $ v\not \in \ZZ_{\leq 0}$   and
$u-v\in\ZZ_{\geq 1}$
\end{enumerate}
}

\subsection{Analytic proof of Proof of Theorem \ref{Thm:PhD}(iii)}

In this proof we fix $t$ and   take  $N$ large enough so that $v+t/N>u$ and $u-t/N<0$.

 \begin{claim}\label{Claim 1} Suppose that $u<0$ and  $v>u$  { satisfy the assumptions of Proposition \ref{Prop2}.}
Then, for fixed $t\ne 0$,
   \begin{equation}
      \label{C1-Mo}
      \GG_N(\alpha+t/N,u-t/N,v+t/N)\sim
      m_0(v,u) \Gamma(\alpha+ 2 t/N-u)^N \Gamma(\alpha+u)^N %
    \end{equation}
    as $N\to\infty$,  where $m_0(v,u)= \Gamma(v-u)/\Gamma(-2u)>0$ is given by \eqref{their-atoms} with $u,v$ swapped.
    Furthermore,  \eqref{C1-Mo} holds also for $t=0$ if $u-v\not\in\ZZ$.
\end{claim}

\begin{proof}

For completeness, we include the argument.

We have $u-t/N<v+t/N$ and for large enough $N$ their difference is not an integer, so we can use \eqref{I_N-x} with $u,v$ replaced by $u-t/N$ and $v+t/N$.
The right-hand side of \eqref{C1-Mo} arises from the first term $j=0$
of the discrete sum $\mathbb{D}_N(\alpha+t/N,v+t/N,u-t/N)$ in \eqref{I_N-x}, and we need to verify that this is
the dominant contribution.

We first confirm that  the contribution from the integral term  in \eqref{I_N-x}
is of lower order. Indeed,
since $|\Gamma(\alpha+t/N+\i x)|\leq \Gamma(\alpha+t/N)$,  the contribution of the integral is bounded by
$$\GG_1(\alpha+t/N,u-t/N,v+t/N) \Gamma(\alpha+t/N)^{2N-2}.$$
We observe that there is $\delta>0$ such that for $t/N<\delta$ this  contribution is of lower order than the right-hand side of \eqref{C1-Mo}. Indeed, as in  \eqref{x-function}, function
$$x\mapsto \tfrac{\Gamma(\alpha+x)^2}{\Gamma(\alpha+ 2 x -u) \Gamma(\alpha+u)}$$
is continuous at $x=0$ and
its value at $x=0$ is smaller than $1$.

It remains to verify that the other terms in \eqref{I_N-x} also have a lower contribution.
By the second part of Lemma \ref{Lem:prop-Gamma}(iv), for $k\ge 1$ such that $u-t/N+k<0$ we have $ \Gamma(\alpha+ 2 t/N-u-k) \Gamma(\alpha+u+k)\leq
\Gamma(\alpha+ 2 t/N-u-1) \Gamma(\alpha+u+1)
= \tfrac{\alpha +u}{\alpha +2 t/N -u-1} \Gamma(\alpha+ 2 t/N-u) \Gamma(\alpha+u)$. Note that  in order for the second atom to exist,
w must have $-\alpha<u<-1$ so $\alpha-(u+1)>0$ and $\alpha+u>0$.
Thus
$$0<\tfrac{\alpha +u}{\alpha +2 t/N -(u+1)}< \tfrac{\alpha +u}{\alpha  -u}:=q<1$$
and we see that the contributions of higher atoms are negligible.

 Similar reasoning shows that the terms of the discrete sum  $\mathbb{D}_N(\alpha+t/N,u-t/N,v+t/N)$ in \eqref{I_N-x} with $m_j(u-t/N,v+t/N)>0$ have a lower contribution than the term $m_0(v+t/N,u-t/N)$ which corresponds to $j=0$ in the sum  $\mathbb{D}_N(\alpha+t/N,v+t/N,u-t/N)$.
\end{proof}
\arxiv
{
If   $u-v\not\in\ZZ_{\leq 0}$ then the above proof shows that \eqref{C1-Mo} holds also for $t=0$.
Thus, under additional conditions that $u-v,u,v\not\in\ZZ_{\leq 0}$,
 we can apply \eqref{C1-Mo} to the numerator and the denominator in \eqref{LG-xx}.
 We obtain
\begin{multline*}
   \lim_{N\to\infty} \log \EE[e^{-2 t \vv L_N/N}]
   =\lim_{N\to\infty} \log \frac{ \GG_N(\alpha+t/N,u-t/N,v+t/N)}{ \GG_N(\alpha,u,v)}
   \\
   =%
   \lim_{N\to \infty} \log \frac{m_0(v+t/N, u-t/N)}{m_0(v,u)}+
  \lim_{N\to \infty}  N \log \frac{\Gamma(\alpha+ 2 t/N-u)}{\Gamma(\alpha-u)}
\\=
2t \lim_{N\to \infty}  \frac{\log \Gamma(\alpha-u+ 2 t/N )-\log \Gamma(\alpha-u)}{2 t /N} =2 t \psi(\alpha-u).
\end{multline*}

We omit the proof for the cases $u-v\in\ZZ_{\leq 0}$ or $u,v\in\ZZ_{\leq 0}$, as we did not provide the corresponding versions of
Propositions \ref{CL:u-v=m} and \ref{Cl:uv=-m,-k} in this paper.
  }

\subsection{Analytic proof of Theorem \ref{Thm:PhD}(iv)}
For  fixed $t\ne 0$ and large enough $N$, we have $u-v-2t/N=-2t/N\not\in\ZZ$ and $ v-t/N,v+t/N \not \in \ZZ_{\leq -1}$. In this case, we can analyze the contributions of the two dominant terms in \eqref{I_N-x} that correspond to $j=0$ in  $\mathbb{D}_N(\alpha+t/N,v+t/N,v-t/N)$ and $j=0$
 in  $\mathbb{D}_N(\alpha+t/N,v-t/N,v+t/N)$.
 For $N\to\infty$, we obtain
\begin{multline*}%
  \GG_N(\alpha+t/N,v+t/N,v-t/N)
  \\ \sim
 \tfrac{\Gamma \left(\frac{2 t}{N}\right)
   \left(\Gamma (\alpha+v ) \Gamma
   \left(\alpha+\frac{2 t}{N}-v
   \right)\right)^N}{\Gamma \left(2
   \left(\frac{t}{N}-v\right)\right)}
   + \tfrac{\Gamma \left(-\frac{2
   t}{N}\right) \left(\Gamma (\alpha
   -v) \Gamma \left(\alpha+\frac{2
   t}{N}+v
   \right)\right)^N}{\Gamma \left(-2
   \left(\frac{t}{N}+v\right)\right)}
   \\\sim \tfrac{N}{2 t} \left(\tfrac{\left(\Gamma
   (\alpha+v ) \Gamma \left(\alpha +\frac{2
   t}{N}-v
   \right)\right)^N}{\Gamma \left(\frac{2
   t}{N}-2 v\right)}-\tfrac{\left(\Gamma
   (\alpha -v) \Gamma \left(\alpha+\frac{2
   t}{N}+v
   \right)\right)^N}{\Gamma \left(-
   \frac{2t}{N}-2v\right)}\right),
\end{multline*}
where we used   $\eps \Gamma(\eps)\sim 1$ as $\eps=\pm 2t/N\to 0$.

Next we consider the normalizing constant. For  $2v\not\in\ZZ_{\leq -1}$ we use \eqref{I_N-x-sing-alt}, where the dominant term
comes from $j=0$. We obtain
\begin{multline*}%
  \mathcal{Z}_N=\GG_N(\alpha,v,v)\sim
  \tfrac{N \psi (\alpha
   -v)-N \psi (\alpha+v )-2 \psi
   (-2 v)-2 \gamma-1/v }{\Gamma (-2
   v)}\Gamma(\alpha+v)^N\Gamma(\alpha-v)^N
  \\ \sim N \tfrac{\psi (\alpha -v)-\psi
   (\alpha+v )}{\Gamma (-2 v)} \Gamma(\alpha+v)^N\Gamma(\alpha-v)^N \quad \mbox{  as $N\to\infty$}.
\end{multline*}
If $2v=-k\in\ZZ_{\leq -1}$, we deduce the same asymptotics from \eqref{I_N-x-sing+}, where the factor $k/k!=1/(k-1)!$ is $1/\Gamma(-2v)$.
 Therefore, using \eqref{LG-xx} we obtain
 \begin{multline*}%
   \EE [ e^{- 2 t \vv L_N/N}]=\tfrac{\GG_N(\alpha+t/N,v-t/N,v+t/N)}{\GG_N(\alpha,v,v)}
 \\  \sim \tfrac{1}{2 t \left(\psi(\alpha-v)-\psi(\alpha+v)\right)} \tfrac{{ \left(\frac{\left(\Gamma
   (\alpha+v) \Gamma \left(\alpha+\frac{2
   t}{N}-v
   \right)\right)^N}{\Gamma \left(\frac{2
   t}{N}-2 v\right)}-\frac{\left(\Gamma
   (\alpha -v) \Gamma \left(\alpha+\frac{2
   t}{N}+v
   \right)\right)^N}{\Gamma \left(-\frac{2t}{N}-2v \right)}\right)}}{  \Gamma(\alpha+v)^N\Gamma(\alpha-v)^N}
   \\
   =\tfrac{1}{2 t \left(\psi(\alpha-v)-\psi(\alpha+v)\right)}
   {{ \left(\tfrac{   \Gamma \left(\alpha+\frac{2
   t}{N}-v
   \right)^N\Gamma(-2 v)}{ \Gamma(\alpha-v)^N\Gamma \left(\frac{2
   t}{N}-2 v\right)}-\tfrac{  \Gamma \left(\alpha+\frac{2
   t}{N}+v
   \right)^N\Gamma(-2 v)}{ \Gamma(\alpha+v)^N\Gamma \left(-\frac{2t}{N}-2v\right)}\right)}}.
 \end{multline*}
 As previously,
 $$\log  \tfrac{   \Gamma \left(\alpha+\frac{2
   t}{N}-v
   \right)^N}{ \Gamma(\alpha-v)^N} = \tfrac{\log \Gamma \left(\alpha+\frac{2
   t}{N}-v
   \right)-\log\Gamma(\alpha-v)}{1/N}\to 2 t \psi(\alpha-v) $$
   and similarly
   $$\log  \tfrac{   \Gamma \left(\alpha+\frac{2
   t}{N}+v
   \right)^N}{ \Gamma(\alpha+v)^N}\to 2 t \psi(\alpha+v).$$
   Therefore,
   $$\lim_{N\to\infty}\EE [ e^{- 2 t \vv L_N/N}]= \tfrac{1}{2t}\;\tfrac{e^{2\psi(\alpha-v) t}-e^{2
   \psi(\alpha+v) t}}{\psi(\alpha-v)  - \psi(\alpha+v)}.
   $$
 To conclude the proof, we  note that the last expression is indeed the Laplace transform of the asserted limit.

\section{Half-space log-gamma polymer}\label{Sec:H-HLGP}
We consider the stationary measures of the half-space log-gamma polymer defined as {   the law of the sequence} \eqref{Half-plane Z}. The change of variables that appears in the proof of Lemma~\ref{Lem:Jacek-Beta}
extends to this setting and yields the following formula for the multipoint Laplace transform of the increments of $\{\log Z_k\}_{k\ge0}$.

\begin{proposition}\label{P:half+}
Fix $\alpha>0$ and $u>-\alpha$. %
{ If $u\ge\tilde v$ and $-\alpha< \tilde v<\alpha$, then   for
  $t_1,\dots, t_k> (\tilde v-\alpha)/2$, with $t_{k+1}=0$ we have}
\begin{multline}\label{half-space-beta+}
 \EE\left[\prod_{r=1}^k Z_r^{2t_{r+1}-2t_r}\right]
\\ = \biggl(\prod_{r=1}^k \tfrac{\Gamma(\alpha-\tilde v+2t_r)}{\Gamma(\alpha-\tilde v)}\biggr)\, \EE\left[
 \prod_{r=1}^k(1+\zeta_0+\zeta_0\zeta_1+\dots+\zeta_0\zeta_1\dots \zeta_{r-1})^{2t_{r+1}-2t_r}\right],
\end{multline}
where $\{\zeta_j\}$ are independent { random variables,} $\zeta_0\simeq \mathrm{Beta}_{II}(u-\tilde v,\alpha+\tilde v)$ and
$\zeta_r\simeq \mathrm{Beta}_{II}(\alpha-\tilde v+2 t_r,\alpha+\tilde v)$ for $r=1,\dots,k-1$,
with a natural convention $\zeta_0=0$ for $\tilde v=u$ (in this case  $\{Z_k\}$ is an inverse-gamma multiplicative random walk).
\end{proposition}
We note that a relation of Beta-type laws to  stationary measures in the half-space is mentioned  in
\cite[Section 1.2.3]{Barraquand-Corwin-AOP2024}.
\begin{proof}
  In view of \eqref{Half-plane Z}, the left-hand side of \eqref{half-space-beta+} is
 \begin{multline*}%
    \tfrac{1}{\Gamma(u-\tilde v) \Gamma(\alpha-\tilde v)^k \Gamma(\alpha+\tilde v)^k}
    \int _{(0,\infty)^{2k+1}} z^{u-\tilde v-1}
  \biggl(\prod_{j=1}^k x_j\biggr)^{\alpha-\tilde v-1}\biggl(\prod_{j=1}^ky_j\biggr)^{\alpha+\tilde v-1}   \\ \cdot  \Biggl(\prod_{r=1}^k  \biggl(\tfrac{1}{x_1\dots x_r}+\sum_{j=1}^r \tfrac{z}{y_1\dots y_j x_j\dots x_r}\biggr)^{2t_{r+1}-2t_r}\Biggr)e^{-z-\sum_{j=1}^k x_j-\sum_{j=1}^k y_j} \d z\, \d \vv x\, \d \vv y.
 \end{multline*}
 Note that
$$
\tfrac{1}{x_1\dots x_r}+z\sum_{j=1}^r\tfrac{1}{y_1\dots y_j x_j\dots x_r}=
\tfrac{1}{x_1\dots x_r}\biggl(1+ \tfrac{z}{y_1}+  \tfrac{z}{y_1}\tfrac{x_1}{y_{2}}+ \dots+ \tfrac{z}{y_1} \prod_{j=2}^{r}\tfrac{x_{j-1}}{y_j}\biggr).
$$

Therefore,
    \begin{multline*}
 \EE\left[\prod_{r=1}^k Z_r^{2t_{r+1}-2t_r}\right]
\\  = \tfrac{1}{\Gamma(u-\tilde v) \Gamma(\alpha-\tilde v)^k \Gamma(\alpha+\tilde v)^k}
    \int _{(0,\infty)^{2k+1}}
    \left(\prod_{j=1}^k x_j^{\alpha-\tilde v+2t_j-1}\right)\left(\prod_{j=1}^ky_j\right)^{\alpha+\tilde v-1}
     \\ \cdot \prod_{r=1}^k   \Bigl(1+ \tfrac{z}{y_1}+  \tfrac{z}{y_1}\tfrac{x_1}{y_{2}}+ \dots+ \tfrac{z}{y_1} \prod_{j=2}^{r}\tfrac{x_{j-1}}{y_j}\Bigr)^{2t_{r+1}-2t_r}e^{-z-\sum_{j=1}^k x_j-\sum_{j=1}^k y_j} \d z\, \d \vv x\, \d \vv y.
    \end{multline*}
        Integrating out $x_k$, we see that
          \begin{multline*}
 \EE\left[\prod_{r=1}^k Z_r^{2t_{r+1}-2t_r}\right]
 \\ = \tfrac{\Gamma(\alpha-\tilde v+2 t_k)}{\Gamma(u-\tilde v) \Gamma(\alpha-\tilde v)^k \Gamma(\alpha+\tilde v)^k}
    \int _{(0,\infty)^{2k}}
    \left(\prod_{j=1}^{k-1} x_j^{\alpha-\tilde v+2t_j-1}\right)\left(\prod_{j=1}^ky_j\right)^{\alpha+\tilde v-1}
    \\ \cdot  \prod_{r=1}^k    \Bigl(1+ \tfrac{z}{y_1}+  \tfrac{z}{y_1}\tfrac{x_1}{y_{2}}+ \dots+ \tfrac{z}{y_1} \prod_{j=2}^{r}\tfrac{x_{j-1}}{y_j}\Bigr)^{2t_{r+1}-2t_r}e^{-z-\sum_{j=1}^{k-1} x_j-\sum_{j=1}^k y_j} \d z \, \d \vv x\, \d \vv y.
    \end{multline*}
    We now change the variable of integration to $z_0=z/y_1$,  $ z_j=x_{j}/y_{j+1}$, $j=1,\dots,k-1$,
    eliminating \(x_j\) and \(z\) while retaining \(y_j\), with Jacobian \(y_1\dots y_k\).

This gives
 \begin{multline*}
 \EE\left[\prod_{r=1}^k Z_r^{2t_{r+1}-2t_r}\right]
 = \tfrac{\Gamma(\alpha-\tilde v+2 t_k)}{\Gamma(u-\tilde v) \Gamma(\alpha-\tilde v)^k \Gamma(\alpha+\tilde v)^k}
  \int_{(0,\infty)^{k}} z_0^{u-\tilde v -1} \Big(\prod_{r=1}^{k-1}  z_{r}^{\alpha-\tilde v+2t_r-1} \Big)
 \\ \cdot  \Big(\prod_{r=1}^k\left(1+z_0+z_0z_1+\dots+z_0z_1\dots z_{r-1}\right)^{2t_{r+1}-2 t_r} \Big) \d \vv z \\
\cdot  \int_{(0,\infty)^{k} }   y_1^{\alpha+u-1} \left( \prod_{r=2}^k y_{r}^{2\alpha+2t_{r-1}-1}\right)
    e^{-\sum_{j=1}^{k} (1+z_{j-1})y_j}\d \vv y.
 \end{multline*}
 Integrating with respect to $\d \vv y  =\d y_1\dots\d y_k$ we obtain
  \begin{multline}\label{direct-last}
 \EE\left[\prod_{r=1}^k Z_r^{2t_{r+1}-2t_r}\right]
\\  = \tfrac{\Gamma(\alpha+u)\Gamma(\alpha-\tilde v+2 t_k)\prod_{r=1}^{k-1}\Gamma(2\alpha+2 t_r)}{\Gamma(u-\tilde v) \Gamma(\alpha-\tilde v)^k \Gamma(\alpha+\tilde v)^k}
 \int_{(0,\infty)^{k}}\, \tfrac{z_0^{u-\tilde v -1}}{(1+z_0)^{\alpha+u}} \left(\prod_{r=1}^{k-1}  \tfrac{z_{r}^{\alpha-\tilde v+2t_r-1}}{(1+z_r)^{2\alpha+2t_r}} \right)
 \\ \cdot  \left(\prod_{r=1}^k\left(1+z_0+z_0z_1+\dots+z_0z_1\dots z_{r-1}\right)^{2t_{r+1}-2 t_r} \right) \,\d \vv z.
\end{multline}
This proves \eqref{half-space-beta+}.
\arxiv{Indeed, the normalizing constant needed for Beta integrals is
$$\frac{\Gamma(\alpha+u)}{\Gamma(u-\tilde v)\Gamma(\alpha+\tilde v)}\prod_{j=1}^{k-1} \frac{\Gamma(2\alpha+2 t_j)}{\Gamma(\alpha-\tilde v+2 t_j)\Gamma(\alpha+\tilde v)}.$$
}
\end{proof}

\subsection{Proof of Theorem~\ref{Thm:IntegralHS}}
By \eqref{direct-last}, the left-hand side of \eqref{Z-int} is given by
 $$
\EE\left[\prod_{r=1}^k Z_r^{2t_{r+1}-2t_r}\right]
    =\tfrac{\GGG_{t_1,\dots,t_k}(u,\tilde v)}{\Gamma(u-\tilde v) \Gamma(\alpha-\tilde v)^k \Gamma(\alpha+\tilde v)^k},
$$
where
\begin{multline}\label{G-1}
    \GGG_{t_1,\dots,t_k}(u,\tilde v)
    = \Gamma(\alpha+u)\Gamma(\alpha-\tilde v+2 t_k)\prod_{j=1}^{k-1}\Gamma(2\alpha+2 t_j)
 \\ \cdot  \int_{(0,\infty)^{k}} \tfrac{z_0^{u-\tilde v -1}}{(1+z_0)^{\alpha+u}} \left(\prod_{j=1}^{k-1}  \tfrac{z_{j}^{\alpha-\tilde v+2t_j-1}}{(1+z_j)^{2\alpha+2t_j}} \right)
  \left(\prod_{j=1}^k R_j(\vv z)^{2t_{j+1}-2 t_j} \right) \d \vv z,
\end{multline}
 $\vv z=(z_0,z_1,\dots, z_{k-1})$ and
  \begin{equation}\label{Rn}
  R_m(\vv x) = 1 + \sum_{j=1}^m \,\prod_{r=1}^j x_r \quad \text{for } m = 1, \dots, k. %
  \end{equation}
Theorem~\ref{Thm:IntegralHS} follows directly from the next result.

\begin{lemma}\label{Thm:half-space}
Let $\alpha>0$, $u\in \CC$ with $\re(u)>-\alpha$,  and $\tilde v<\re(u)$ such that $ |\tilde v|<\alpha$.
  For  $\re(u)>t_1>t_2>\dots>t_k> |\tilde v|$, with $t_{k+1}=0$, denote
\begin{multline}\label{IIIk}
    \II_{t_1,\dots,t_k}(u,\tilde v) = \tfrac{1}{(2\pi)^k \Gamma(u+\tilde v)\prod_{j=1}^k \Gamma(2t_j-2t_{j+1})}
\\ \cdot  \int_{(0,\infty)^{k}}
\prod_{j=1}^{k-1}
\left|
\Gamma(t_j-t_{j+1}+\i y_j+\i y_{j+1})
\Gamma(t_j-t_{j+1}+\i y_j-\i y_{j+1})
\right|^2
\\ \cdot  \left(
\prod_{j=1}^k
\tfrac{|\Gamma(\alpha+t_j+\i y_j)|^2}
     {|\Gamma(2\i y_j)|^2}
\right)
\left|\Gamma(t_k+\tilde v+\i y_k)\right|^2 \left|\Gamma(t_k-\tilde v+\i y_k)\right|^2
\\ \cdot
\Gamma(u-t_1+\i y_1)\,\Gamma(u-t_1-\i y_1)
\,\d y .
\end{multline}

Then
\begin{equation}\label{IG}
    \GGG_{t_1,\dots,t_k}(u,\tilde v) = \II_{t_1,\dots,t_k}(u,\tilde v).
\end{equation}
\end{lemma}

\bigskip
\begin{proof}
{From \eqref{G-1}, it is clear that $\GGG_{t_1,\dots,t_k}(u,\tilde v)$ is an analytic function of $u$ for
$\re(u)>\tilde v$.
On the other hand, the integral \eqref{IIIk} is an analytic function of $u$ for
$\re(u)>t_1$. We will prove by induction on $k$ that \eqref{IG} holds for $\re(u)>t_1$.
}

By Lemma~\ref{Lem:H=L}, \eqref{G-1} is the same as
\begin{multline}\label{GGGk}
\GGG_{t_1,\dots,t_k}(u,\tilde v)
    = \Gamma(\alpha+u)\Gamma(\alpha-\tilde v+2 t_k)\prod_{j=1}^{k-1}\Gamma(2\alpha+2 t_j)
 \\ \cdot  \int_{(0,\infty)^k} \frac{ \prod_{i=1}^k x_i^{\alpha+\tilde v+2t_i-1} }{  (1+x_k)^{\alpha-\tilde v+2t_k} \prod_{i=1}^{k-1} (1+x_i)^{2\alpha+2t_i} } \frac{1}{\big(R_k(\vv x)\big)^{u+\tilde v}} \,\mathrm{d}x_1 \dots \mathrm{d}x_k.
\end{multline}

\bigskip
For $k=1$ from \eqref{GGGk} we get
\begin{align}
    \GGG_{t_1}(u,\tilde v) &= \Gamma(\alpha+u)\Gamma(\alpha-\tilde v+2t_1) \int_0^\infty \tfrac{x_1^{\alpha+\tilde v+2t_1-1}}{(1+x_1)^{\alpha+u+2t_1}} \,\mathrm{d} x_1\nonumber \\
    =& \tfrac{\Gamma(\alpha+u)\Gamma(\alpha-\tilde v+2t_1)\Gamma(\alpha+\tilde v+2t_1)\Gamma(u-\tilde v)}{\Gamma(\alpha+u+2t_1)},\label{k=1}
\end{align}
while \eqref{IIIk} yields
\begin{multline*}
 \II_{t_1}(u,\tilde v)
 = \tfrac{1}{\Gamma(u+\tilde v)\Gamma(2t_1)}\,\\ \cdot \frac{1}{2\pi}\int_{(0,\infty)}\,  \tfrac{|\Gamma(\alpha+t_1+\i y_1)\Gamma(t_1+\tilde v+\i y_1)\Gamma(t_1-\tilde v+\i y_1)|^2}{|\Gamma(2\i y_1)|^2}
\Gamma(u-t_1+\i y_1)\Gamma(u-t_1-\i y_1)\,\d y_1,
\end{multline*}
which in view of \eqref{WI+} with $a_1=\alpha+t_1$, $a_2=u-t_1$, $a_3=t_1+\tilde v$, and $a_4=t_1-\tilde v$ evaluates to \eqref{k=1}.

\bigskip

Assume the identity \eqref{IG} holds for $k-1$, where $k\geq 2$. To prove that it also holds for $k$ we will show that both $\GGG_{t_1,\dots,t_k}$ and $\II_{t_1,\dots,t_k}$ satisfy the  same integral recurrence formula. First we will identify the recurrence formula for $\GGG_{t_1,\dots,t_k}$ and then we will show that it is satisfied by $\II_{t_1,\dots,t_k}$.

We rewrite  the integrand in \eqref{GGGk} using the decomposition
$
R_k(\vv x)=1+x_1R_{k-1}(x_2,\dots,x_k),
$
 and applying  \eqref{A+B} with
$
A=x_1R_{k-1}(x_2,\dots,x_k)$, $B=1
$ {  and $\eps>0$ such that $\re(u)-\eps>t_1$ and  $\re(u)+\tilde v>\eps$. (Note that under our assumptions, $\re(u)+\tilde v>0$.) }
\begin{multline*}
   \frac{1}{R_k(\vv x)^{u+\tilde v}}
   \\ =\tfrac{1}{2\pi \Gamma(u+\tilde v)}   \int_\RR  \Gamma(u+\tilde v-\eps+\i s) \Gamma(\eps - \i s) x_{1}^{-u-\tilde v+\eps-\i s} \tfrac{1}{R_{k-1}(x_2,\dots,x_k)^{u+\tilde v-\eps+\i s} }\d s.
\end{multline*}

Inserting this representation into \eqref{GGGk} %
and changing the order of integration we obtain
\begin{multline*}
    \GGG_{t_1,\dots,t_k}(u,\tilde v)
    \\=
    \tfrac{\Gamma(\alpha+u)\Gamma(\alpha-\tilde v+2 t_k)\prod_{r=1}^{k-1}\Gamma(2\alpha+2 t_r)}{2\pi \Gamma(u+\tilde v)}   \int_\RR  \Gamma(u+\tilde v-\eps+\i s) \Gamma(\eps - \i s)
      \hfill
    \\ \cdot
    \Big[\int_{(0,\infty)^k} \tfrac{ x_1^{\alpha+2t_1-u+\eps-\i s-1} \prod_{j=2}^k x_j^{\alpha+\tilde v+2t_j-1} }{  (1+x_k)^{\alpha-v+2t_k} \prod_{i=1}^{k-1} (1+x_i)^{2\alpha+2t_i} }
     \tfrac{1}{R_{k-1}(x_2,\dots,x_k)^{u+\tilde v-\eps+\i s} }
    \,\mathrm{d}x_1 \dots \mathrm{d}x_k\Big]\,\d s.
  \end{multline*}
Integration with respect to $x_1$ gives \footnotesize
\begin{multline}
    \label{GGG-rec+}
    \GGG_{t_1,\dots,t_k}(u,\tilde v)=
    \tfrac{\Gamma(\alpha+u)\Gamma(\alpha-\tilde v+2 t_k)\prod_{r=2}^{k-1}\Gamma(2\alpha+2 t_r)}{2\pi \Gamma(u+\tilde v)} \\
    \cdot  \int_\RR  \Gamma(u+\tilde v-\eps+\i s) \Gamma(\eps - \i s)\Gamma(\alpha+2t_1-u+\eps-\i s) \Gamma(\alpha+u-\eps+\i s)
    \\ \cdot
    \left[\int_{(0,\infty)^{k-1}} \tfrac{   \prod_{j=2}^k x_j^{\alpha+\tilde v+2t_j-1} }{  (1+x_k)^{\alpha-v+2t_k} \prod_{i=2}^{k-1} (1+x_i)^{2\alpha+2t_i} }
     \tfrac{1}{R_{k-1}(x_2,\dots,x_k)^{u+\tilde v-\eps+\i s} }
    \,\mathrm{d}x_2 \dots \mathrm{d}x_k\right]\d s
    \\ = \tfrac{\Gamma(\alpha+u)}{2\pi \Gamma(u+\tilde v)} \int_{ \RR} \Gamma(\alpha+2t_1-u+\eps -\i s) \Gamma(u+\tilde v-\eps+\i s)\Gamma(\eps-\i s)\GGG_{t_2,\dots,t_k}(u-\eps+\i s,\tilde v) \d s.
\end{multline}
\normalsize

\bigskip
Next, we consider $\II_{t_1,\dots,t_k}$. We will first derive another formula for $\II_{t_1,\dots,t_k}$. To this end we rewrite \eqref{IIIk} as
\begin{multline*}
    \II_{t_1,\dots,t_k}(u,\tilde v) = \tfrac{1}{(2\pi)^k \Gamma(u+\tilde v)\prod_{j=1}^k \Gamma(2t_j-2t_{j+1})}
\\ \cdot  \int_{(0,\infty)^{k-1}}
\prod_{j=2}^{k-1}
\left|
\Gamma(t_j-t_{j+1}+\i y_j+\i y_{j+1})
\Gamma(t_j-t_{j+1}+\i y_j-\i y_{j+1})
\right|^2
\\ \cdot  \left(
\prod_{j=2}^k
\tfrac{|\Gamma(\alpha+t_j+\i y_j)|^2}
     {|\Gamma(2\i y_j)|^2}
\right)
\left|\Gamma(t_k+\tilde v+\i y_k)\right|^2 \left|\Gamma(t_k-\tilde v+\i y_k)\right|^2\\ \cdot   J_{t_1,t_2}(y_2)
\,\d y_1\dots\d y_k,
\end{multline*}
where $J_{t_1,t_2}(y_2)$ is
\begin{equation*}
 \int_0^\infty  \tfrac{|\Gamma(\alpha+t_1+\mathrm{i} y_1)\Gamma(t_1-t_2+\mathrm{i} y_1+\mathrm{i} y_2)\Gamma(t_1-t_2+\mathrm{i} y_1-\mathrm{i} y_2)|^2\Gamma(u-t_1+\mathrm{i} y_1)\Gamma(u-t_1-\mathrm{i} y_1)}{|\Gamma(2\mathrm{i} y_1)|^2}\,\mathrm{d} y_1.
\end{equation*}
Using again \eqref{WI+}, this time with parameters $a_1=\alpha+t_1$, $a_2=u-t_1$, $a_3=t_1-t_2+\mathrm{i} y_2$, and $a_3=t_1-t_2-\mathrm{i} y_2$ we see that
\begin{equation*}
  J_{t_1,t_2}(y_2)=  \tfrac{2\pi \Gamma(\alpha+u)\Gamma(2t_1-2t_2)|\Gamma(\alpha+2t_1-t_2+ \mathrm{i} y_2)|^2\Gamma(u-t_2+ \mathrm{i} y_2)\Gamma(u-t_2- \mathrm{i} y_2)}{\Gamma(\alpha+u+2t_1-2t_2)}. %
\end{equation*}
Consequently,
\begin{multline}\label{Ilast}
    \II_{t_1,\dots,t_k}(u,\tilde v) = \tfrac{\Gamma(\alpha+u)}{(2\pi)^{k-1}\Gamma(u+\tilde v)\Gamma(\alpha+u+2t_1-2t_2)\,\prod_{j=2}^k \Gamma(2t_j-2t_{j+1})}
\\ \cdot  \int_{(0,\infty)^{k-1}}\,|\Gamma(\alpha+2t_1-t_2+ \mathrm{i} y_2)|^2
\\ \cdot  \Big(\prod_{j=2}^{k-1}
\left|
\Gamma(t_j-t_{j+1}+\i y_j+\i y_{j+1})
\Gamma(t_j-t_{j+1}+\i y_j-\i y_{j+1})
\right|^2\Big)
\\ \cdot  \Big(
\prod_{j=2}^k
\tfrac{|\Gamma(\alpha+t_j+\i y_j)|^2}
     {|\Gamma(2\i y_j)|^2}
\Big)
\left|\Gamma(t_k+\tilde v+\i y_k)\right|^2 \left|\Gamma(t_k-\tilde v+\i y_k)\right|^2\,
\\ \cdot   \Gamma(u-t_2+ \mathrm{i} y_2)\Gamma(u-t_2- \mathrm{i} y_2)
\,\d y_2\dots\d y_k,
\end{multline}

On the other hand, by our choice of $\eps>0$ in \eqref{GGG-rec+}, we have
$\re(u)-\eps>t_1>t_2$, so the induction hypothesis \eqref{IG} applies.
Therefore, we may replace
$\GGG_{t_2,\dots,t_k}(u-\eps+\i s,\tilde v)$ in the last expression of
\eqref{GGG-rec+} by the integral representation \eqref{IIIk}, obtaining \color{black}
\begin{multline*}
  \tfrac{\Gamma(\alpha+u)}{2\pi \Gamma(u+\tilde v)} \int_{ \RR} \Gamma(\alpha+2t_1-u+\eps -\i s) \Gamma(u+\tilde v-\eps+\i s)\Gamma(\eps-\i s)\II_{t_2,\dots,t_k}(u-\eps+\i s,\tilde v) \d s  \\
  = \tfrac{\Gamma(\alpha+u)}{(2\pi)^k \Gamma(u+\tilde v)\,\prod_{j=2}^k \Gamma(2t_j-2t_{j+1})} \int_{ \RR} \Gamma(\alpha+2t_1-u+\eps -\i s)\Gamma(\eps-\i s)
\\ \cdot  \int_{(0,\infty)^{k-1}}
\Big(\prod_{j=2}^{k-1}
\left|
\Gamma(t_j-t_{j+1}+\i y_j+\i y_{j+1})
\Gamma(t_j-t_{j+1}+\i y_j-\i y_{j+1})
\right|^2\Big)\\ \cdot  \Big(
\prod_{j=2}^k
\tfrac{|\Gamma(\alpha+t_j+\i y_j)|^2}
     {|\Gamma(2\i y_j)|^2}
\Big)
\left|\Gamma(t_k+\tilde v+\i y_k)\right|^2 \left|\Gamma(t_k-\tilde v+\i y_k)\right|^2 \\ \cdot
\Gamma(u-\eps+\i s-t_2+\i y_2)\,\Gamma(u-\eps+\i s-t_2-\i y_2)
\,\d y_2\dots\d y_k\,\d s\\
= \frac{\Gamma(\alpha+u)}{(2\pi)^k \Gamma(u+\tilde v)\,\prod_{j=2}^k \Gamma(2t_j-2t_{j+1})}\hfill  \\
\cdot  \int_{(0,\infty)^{k-1}}
\Big(\prod_{j=2}^{k-1}
\left|
\Gamma(t_j-t_{j+1}+\i y_j+\i y_{j+1})
\Gamma(t_j-t_{j+1}+\i y_j-\i y_{j+1})
\right|^2\Big)
\\ \cdot \Big(
\prod_{j=2}^k
\tfrac{|\Gamma(\alpha+t_j+\i y_j)|^2}
     {|\Gamma(2\i y_j)|^2}
\Big)
\left|\Gamma(t_k+\tilde v+\i y_k)\right|^2 \left|\Gamma(t_k-\tilde v+\i y_k)\right|^2\\ \cdot  K_{t_1,t_2}(y_2)\,\,\d y_2\dots\d y_k,
\end{multline*}
where $K_{t_1,t_2}(y_2)$ is
\[
\int_{\mathbb R}\,\Gamma(\alpha+2t_1-u+\eps -\i s)\Gamma(\eps-\i s) \Gamma(u-\eps+\i s-t_2+\i y_2)\,\Gamma(u-\eps+\i s-t_2-\i y_2)\,\d s.
\]
{(The interchanges of the order of integration are justified by absolute integrability.)}
In view of \eqref{koek:1.6.3} with $a=u-\eps-t_2+\i y_2$, $b=u-\eps-t_2-\i y_2$, $c=\eps$, and $d=\alpha+2t_1-u+\eps$ we get
\begin{equation*}
   K_{t_1,t_2}(y_2)= 2\pi\,\tfrac{|\Gamma(\alpha+2t_1-t_2+\mathrm{i} y_2)|^2\Gamma(u-t_2+\mathrm{i} y_2)\Gamma(u-t_2-\mathrm{i} y_2)}{\Gamma(\alpha+u+2t_1-2t_2)}.
\end{equation*}
Inserting this explicit expression for $K_{t_1,t_2}(y_2)$ in the last integral,  we obtain the right-hand side of \eqref{Ilast}, which is $\II_{t_1,\dots,t_k}$.

That is $\II_{t_1,\dots,t_k}$ satisfies the same recurrence formula as $\GGG_{t_1,\dots,t_k}$, and thus the identity \eqref{IG} follows.
 Indeed, the identity \eqref{IG} has already been established for a single variable $t_k$. Moreover, if \eqref{IG} holds for the $k-1$ variables $t_2,\dots,t_k$ { for some $k\geq 2$}, then the recurrence relations imply that it also holds for {   the $k$ variables} $t_1,\dots,t_k$
such that $\re(u)>t_1>t_2>\dots>t_k >|\tilde v|$.
\end{proof}

\subsection{Proof of Theorem~\ref{Thm-half-lim}} {
Recall that weak convergence in $\RR^\NN$, endowed with the product topology, is equivalent to convergence of finite-dimensional distributions. Since $\vv L\topp N_0=0$ and $\log Z_0=0$,   it is enough to prove that, for every $k\ge1$, the increments of
$
(\vv L\topp N_0,\dots,\vv L\topp N_k)
$
converge in distribution to the corresponding increments of
$
(\log Z_0,\dots,\log Z_k)
$.}

In the proof, the arguments differ across the various regions of the phase diagram, and we therefore divide the proof into several parts.
 For the convenience of the reader, we give an overview of the proof.
\begin{enumerate}[(i)]
\item \emph{Maximal current:} $u>0$ and $v>0$. In this case,  in  Section~\ref{Sec:Thm1.3(1)} we establish convergence to process \eqref{Half-plane Z} with $\tilde v=0$.

\item \emph{High density:} $u>v$ and $-\alpha<v<0$. In this case,  in  Section~\ref{Sec:Thm1.3(1i)} we establish convergence to process \eqref{Half-plane Z} with $\tilde v=v$.

\item \emph{Low density:} $u<v$ and {  $-\alpha<u<0$}. In this case,  in
Section~\ref{Sec:Thm1.3(1ii)} we establish convergence to process \eqref{Half-plane Z} with  $\tilde v=u$.
\item  \emph{Coexistence line:} %
{   $-\alpha<u=v\le 0$.} In this case,  in
Section~\ref{Sec:Thm1.3(iv)} we establish convergence to process \eqref{Half-plane Z} with  $\tilde  v=v$.
\item { \emph{Left  boundary of the maximal current region:}} $u=0,v>0$.  In this case,  in
Section~\ref{Sec:Thm1.3(v)} we establish convergence to process \eqref{Half-plane Z} with  $\tilde v=0$.
\item { \emph{Lower boundary of the maximal current region:}} $u>0,v=0$.   In this case,  in
Section~\ref{Sec:Thm1.3(vi)} we establish convergence to process \eqref{Half-plane Z} with  $\tilde v=0$.
\end{enumerate}
The analysis at the boundary of the maximal current region (points (v) and (vi) above) relies on several subtle   properties of exponential functionals of random walks. For the reader's convenience, we summarize the required results in Lemma~\ref{Conj} in Section~\ref{sec:AuxFacts}.

\subsubsection{Proof of Theorem~\ref{Thm-half-lim}  in maximal current region} %
\label{Sec:Thm1.3(1)}
We fist prove that for $u,v>0$, the  limit { as $N\to\infty$} of the Laplace transform of  $\{\vv L\topp N_k\} _{k=1,2\dots} $  does not depend on $v$ and is given by the following integral.

 \begin{proposition}\label{Prop:HalfSpace lim}
 Let $u,v,\alpha>0$.  Fix $k=1,\dots$ and $u>t_1>\dots>t_k>t_{k+1}:=0$. Then
\begin{multline*}%
\lim_{N\to\infty }  \EE\left[e^{\sum_{j=1}^k -2 t_j( \vv L\topp N_j-\vv L\topp N_{j-1})}\right]
= \tfrac{1}{(2\pi)^{k} \Gamma(u)^2 \Gamma(\alpha)^{2k}\prod_{j=1}^k \Gamma(2 t_{j}-2 t_{j+1})}
\\ \cdot  \int _{(0,\infty)^{k}}
 \prod_{j=1}^{k-1} |\Gamma(t_{j}-t_{j+1}+ \i {y_{j}}+\i {y_{j+1}})\Gamma( t_{j}-t_{j+1} + \i {y_{j}}-\i {y_{j+1}} )|^2
\\ \cdot  \left(\prod_{j=1}^{k}\tfrac{\left|\Gamma(\alpha+t_{j}+\i {y_j})\right|^{2}}{|\Gamma(2\i {y_j})|^2} \right) \left|\Gamma(u-t_{1} +\i {y_1} )\right|^2\,|
\Gamma(t_{k}+\i y_{k})|^4\, \d \vv y.
\end{multline*}
 \end{proposition}

\begin{proof}
 Applying \eqref{Psi(0)} and passing to the limit $N\to\infty$ in Theorem~\ref{BarraquandThm1.11}, with $d=k+1$ and $\vv t=(t_1,\dots,t_k,0,\dots,0)$, yields the result. We omit the details.

\arxiv{  For completeness, we provide omitted argument.

By Theorem \ref{BarraquandThm1.11} used with $d=k+1$ and $\vv t=(t_1,\dots,t_k,0,\dots,0)$, the   Laplace transform  for $u,v>0$ is
$$ \EE\left[e^{-2\sum_{j=1}^k t_j( \vv L\topp N_j-\vv L\topp N_{j-1})}\right] = \frac{I_1(N) }{I_2(N)},$$
where
\begin{equation}
    \label{I_1(N)}
    I_1(N)=\frac{A(N)}{B(N)}
\end{equation}
with
\begin{multline*}
    A(N)= \tfrac{1}{(2\pi)^{k+1} \Gamma(u+v)\prod_{j=1}^k \Gamma(2 t_{j}-2 t_{j+1})}
 \int_{(0,\infty)^{k+1}} \prod_{j=1}^{k} |\Gamma(t_{j}-t_{j+1}+ \i {y_{j}}+\i {y_{j+1}})\Gamma( t_{j}-t_{j+1} + \i {y_{j}}-\i {y_{j+1}} )|^2
\\ \cdot
 \frac{\left|\Gamma(\alpha+ \i y_{k+1})\right|^{2N-2k}}{|\Gamma(2\i {y_{k+1}})|^2}
 \left(\prod_{j=1}^k\tfrac{\left|\Gamma(\alpha+t_{j}+\i {y_j})\right|^{2}}{|\Gamma(2\i {y_j})|^2} \right) \left|\Gamma(u-t_{1} +\i {y_1} )\right|^2\,|\Gamma(v+\i y_{ k+1 })|^2\, \d \vv y,
\end{multline*}
$$
    B(N)= \int_0^\infty |\Gamma(\alpha+\i x)|^{2N-2k} \tfrac{|\Gamma(v+\i x)|^2 }{|\Gamma(2\i x)|^2} \d x
$$
and
\begin{equation}\label{I_2(N)}
I_2(N)= \frac{\GG_N(\alpha,u,v)}{B(N)}=
 \frac{
   \frac{1}{2\pi\Gamma(u+v)} \int_0^\infty |\Gamma(\alpha+\i x)|^{2N} \tfrac{|\Gamma(v+\i x)|^2|\Gamma(u+\i x)|^2}{|\Gamma(2\i x)|^2} \d x
    }
     {
    \int_0^\infty |\Gamma(\alpha+\i x)|^{2N-2k} \tfrac{|\Gamma(v+\i x)|^2 }{|\Gamma(2\i x)|^2} \d x  }
     \to \frac{\Gamma(u)^2\Gamma(\alpha)^{2k}}{2\pi \Gamma(u+v)}.
\end{equation}
To compute the limit of  \eqref{I_1(N)}  we use  \eqref{Psi(0)} with
\begin{multline*} \Psi(x)=\tfrac{1}{(2\pi)^{k+1} \Gamma(u+v) \prod_{j=1}^k \Gamma(2 t_{j}-2 t_{j+1})}
 \int_{(0,\infty)^{k}} \prod_{j=1}^{k-1} |\Gamma(t_{j}-t_{j+1}+ \i {y_{j}}+\i {y_{j+1}})\Gamma( t_{j}-t_{j+1} + \i {y_{j}}-\i {y_{j+1}} )|^2
\\ \cdot     |\Gamma(t_{k}+ \i y_{k}+\i x)\Gamma( t_{k}+ \i y_{k}-\i x )|^2
 \left(\prod_{j=1}^k\tfrac{\left|\Gamma(\alpha+t_{j}+\i {y_j})\right|^{2}}{|\Gamma(2\i {y_j})|^2} \right) \left|\Gamma(u-t_{1} +\i {y_1} )\right|^2\,\, \d \vv y
 \end{multline*}
and
$$\phi(x)=\tfrac{|\Gamma(v+\i x)|^2 }{|\Gamma(2\i x)|^2}.$$
 We obtain
 \begin{multline*}
     \lim_{N\to\infty}\frac{A(N)}{B(N)}=\Psi(0)
     \\ = \tfrac{1}{(2\pi)^{k+1} \Gamma(u+v) \prod_{j=1}^k \Gamma(2 t_{j}-2 t_{j+1})}
\int _{(0,\infty)^{k}}
 \prod_{j=1}^{k-1} |\Gamma(t_{j}-t_{j+1}+ \i {y_{j}}+\i {y_{j+1}})\Gamma( t_{j}-t_{j+1} + \i {y_{j}}-\i {y_{j+1}} )|^2
\\ \cdot  \left(\prod_{j=1}^{k}\tfrac{\left|\Gamma(\alpha+t_{j}+\i {y_j})\right|^{2}}{|\Gamma(2\i {y_j})|^2} \right) \left|\Gamma(u-t_{1} +\i {y_1} )\right|^2\,|
\Gamma(t_{k}+\i y_{k})|^4\, \d \vv y.
 \end{multline*}
For the limit in \eqref{I_2(N)} we use \eqref{Psi(0)} with
$$
\Psi(x)=|\Gamma(u+\i x)|^2 |\Gamma(\alpha+\i x)|^{2k}
$$
and
$$\phi(x)=\tfrac{|\Gamma(v+\i x)|^2 }{|\Gamma(2\i x)|^2}.$$}
\end{proof}

We are now ready to conclude the proof of this part of Theorem~\ref{Thm-half-lim}.  %
We verify convergence of the finite-dimensional distributions.

By Proposition~\ref{Prop:HalfSpace lim}, the multipoint Laplace transform of the increments of
$
\{\vv L\topp N_1,\dots,\vv L\topp N_k\}
$
converges.   Theorem~\ref{Thm:IntegralHS} identifies the limit as %
\[
\EE\!\left[\prod_{r=1}^k Z_r^{\,2t_{r+1}-2t_r}\right]
=
\EE\!\left[
Z_1^{-2t_1}
\left(\tfrac{Z_2}{Z_1}\right)^{-2t_2}
\dots
\left(\tfrac{Z_k}{Z_{k-1}}\right)^{-2t_k}
\right],
\]
which is the multipoint Laplace transform of the increments of
$
\{\log Z_1,\dots,\log Z_k\}
$ with $\tilde v=0$.

Since $\vv L\topp N_0=0$ and $\log Z_0=0$, this determines the limiting multipoint Laplace transforms on an open set of parameters. By \cite[Theorem~A.1]{Bryc-Wang-2017ASEP}, convergence of the Laplace transforms on an open set implies convergence of the corresponding finite-dimensional distributions. \qed

\subsubsection{Proof of Theorem \ref{Thm-half-lim} in the high density region} \label{Sec:Thm1.3(1i)}
We assume $u>v$ and $-\alpha<v<0$.
 Consider  $\vv t=(t_1,\dots,t_k,0,0,\dots,0)\in\RR^N$ so that
    $$\lim_{N\to\infty }  \EE\left[e^{\sum_{j=1}^k -2 t_j( \vv L\topp N_j-\vv L\topp N_{j-1})}\right] = \lim_{N\to\infty} \frac{\GG_{t_1,\dots,t_k,0,\dots 0}(\alpha,u,v)}{\GG_N(\alpha,u,v)}.$$
    From \eqref{pert-MU}, we see that
     \begin{multline*}%
    \GG_{t_1,\dots,t_k,0,\dots 0}(\alpha,u,v)
   \\  =\Gamma(\alpha+u) \Gamma(\alpha-v)^{N-1} \Gamma(\alpha{+}v)^{N-k} \Big(\prod_{r=1}^k \Gamma(\alpha+v+2t_r)\Big)  \\ \cdot  \EE \left[(1+\xi_1+\xi_1\xi_2+\dots+\xi_1\dots\xi_{N-1})^{-u-v}\right]
     \\=\Gamma(\alpha+u) \Gamma(\alpha-v)^{N-1} \Gamma(\alpha+v)^{N-k} \biggl(\prod_{r=1}^k \Gamma(\alpha+v+2t_r)\biggr)
     \hfill \\\cdot  \EE \Bigl[\Big(1+\xi_1 + \xi_1\xi_2+\dots+\xi_1\dots\xi_{k-1}\\+\xi_1\dots\xi_{k}\underbrace{(1+\xi_{k+1}+\xi_{k+1}\xi_{k+2}+\dots+\xi_{k+1}\dots\xi_{N-1})}_{R_{N-k-1}} \Big)^{-u-v}\Bigr]
     \\=\Gamma(\alpha+u) \Gamma(\alpha-v)^{N-1} \Gamma(\alpha{+}v)^{N-k} \prod_{r=1}^k \Gamma(\alpha+v+2t_r)
     \\\cdot  \EE \left[\left(1+\xi_1+\xi_1\xi_2+\dots+\xi_1\dots\xi_{k}R_{N-k-1}\right)^{-u-v}\right]
     \end{multline*}
with independent {  random variables} $\xi_j\simeq \mathrm{Beta}_{II}(\alpha+v+2 t_j,\alpha-v)$, $j=1,\dots,k$ and
$\xi_j\simeq \mathrm{Beta}_{II}(\alpha+v,\alpha-v)$ for $j> k$.
Consequently, if $v<0$, then $R_{N-k-1}\to R$ almost surely, where $R-1\simeq \mathrm{Beta}_{II}(\alpha+v,-2v)$, see \cite[Example 9]{chamayou1991explicit} and  we see that
 \begin{multline*}
  \lim_{N\to\infty }  \EE\left[e^{\sum_{j=1}^k -2 t_j( \vv L\topp N_j-\vv L\topp N_{j-1})}\right]
  \\=
 \tfrac{\EE \left[\left(1+\xi_1+\xi_1\xi_2+\dots+\xi_1\dots\xi_{k-1}R' \right)^{-u-v}\right].}{\EE[R^{-u-v}]} \prod_{r=1}^k \tfrac{\Gamma(\alpha+v+2t_r)}{\Gamma(\alpha+v)}.
 \end{multline*}
  Moreover, $R'=\xi_{k}R$, where $\xi_{k}$ and $R=1+\xi_{k+1}+\xi_{k+1}\xi_{k+2}+\dots$ are independent, which  gives $R'\simeq \mathrm{Beta}_{II}(\alpha+v+2t_k,-2v)$.
  With explicit {  formula}
  $$\mathbb E\,[R^{-u-v}]=\tfrac{\Gamma(u-v)\Gamma(\alpha-v)}{\Gamma(\alpha+u)\Gamma(-2v)},$$
we can write   the answer as an explicit integral. We obtain
 \begin{multline*}%
     \lim_{N\to\infty }  \EE\left[e^{\sum_{j=1}^k -2 t_j( \vv L\topp N_j-\vv L\topp N_{j-1})}\right]
    =   \tfrac{\Gamma(\alpha+u) \Gamma(\alpha-v+2t_k) }{\Gamma(u-v)\Gamma(\alpha+v)^k\Gamma(\alpha-v)^k}
   \prod_{r=1}^{k-1} {\Gamma(2\alpha+2t_j)}
   \\ \cdot  \int_{(0,\infty)^k}
    \tfrac{ \prod_{j=1}^k x_j^{\alpha+v+2t_j-1} }{ \left(1 + \sum_{j=1}^k \,\prod_{r=1}^j x_r \right)^{u+v} (1+x_k)^{\alpha-v+2t_k} \prod_{j=1}^{k-1} (1+x_j)^{2\alpha+2t_j} } \,\mathrm{d}x_1 \dots \mathrm{d}x_k.
 \end{multline*}
 Our goal is to show that this limit coincides with the expression \eqref{half-space-beta+} for $\tilde v=v$, which we also write as an explicit integral.
  Using notation \eqref{Rn}, we obtain
 \begin{multline*}
    \EE\left[\prod_{r=1}^k Z_r^{2t_{r+1}-2t_r}\right]
= \tfrac{\Gamma(\alpha+u)\Gamma(\alpha-v+2 t_k)}{\Gamma(u-v)\Gamma(\alpha+v)^k \Gamma(\alpha-v)^k}  \prod_{j=1}^{k-1}  \Gamma(2\alpha+2 t_j)
\\ \cdot
\int_{(0,\infty)^k} \tfrac{ y_1^{u-v-1} \left( \prod_{j=2}^k y_j^{\alpha-v+2t_{j-1}-1} \right) \prod_{j=1}^{k} \big(R_j(\vv y)\big)^{2t_{j+1}-2t_j} }{  (1+y_1)^{\alpha+u} \prod_{j=2}^k (1+y_j)^{2\alpha+2t_{j-1}} } \,\mathrm{d}y_1 \dots \mathrm{d}y_k,
 \end{multline*}
Since  the integrals \eqref{Lk} and \eqref{Hk} are equal, this ends  the proof. \qed

\subsubsection{Proof of Theorem \ref{Thm-half-lim}  in the low density region}   \label{Sec:Thm1.3(1ii)}
We assume  $u<v$ and $-\alpha<u<0$. The half-space stationary measure \eqref{Half-plane Z} with parameter $\tilde v=u$ is a $\mathrm{Gamma}^{-1}(\alpha-u)$ multiplicative random walk. So we need to show that for each  $k$ and $t_1,\dots,t_k \geq 0 $   we have
   \begin{equation}\label{limNk-walk}
\lim_{N\to\infty }  \EE\left[e^{\sum_{j=1}^k -2 t_j( \vv L\topp N_j-\vv L\topp N_{j-1})}\right]
=\prod_{j=1}^k \tfrac{\Gamma(\alpha-u+2 t_j)}{\Gamma(\alpha-u)}.
\end{equation}
 If $N>k$,  from \eqref{pert-MV}   we obtain
       \begin{multline*}%
     \GG_{t_1,\dots,t_k,0,\dots 0}(\alpha,u,v)
     =\Gamma(\alpha+v) \Gamma(\alpha+u)^N  \Gamma(\alpha-u)^{N-1-k} \prod_{j=1}^k \Gamma(\alpha-u+2t_j) \\ \cdot
     \EE [(\underbrace{1+\zeta_1+\zeta_1\zeta_2+\dots+\zeta_1\dots \zeta_{N-k-2}}_{R_{N-k-2}(\vv \zeta)}
     \\ +\zeta_1\dots \zeta_{N-k-1}  (\underbrace{1+\xi_1+\dots+ \xi_1\dots  \xi_k}_{R_k(\vv\xi)}))^{-u-v}]
     \end{multline*}
with independent   random variables $\zeta_j\simeq \mathrm{Beta}_{II}(\alpha+u,\alpha-u)$, $j=1,\dots,N-k$,
and $  \xi_j\simeq \mathrm{Beta}_{II}(\alpha+u,\alpha-u+2t_{k-j+1})$ for $j=1,\dots,k$.
As previously, for   $0<p<-2u$, we have
$$\EE[\zeta_1^p]= \tfrac{\Gamma(\alpha+u+p)\Gamma(\alpha-u-p)}{\Gamma(\alpha+u)\Gamma(\alpha-u)}<1,$$
 the  series \eqref{Rseries}  converges in probability and the  moments of order $-(u+v)$ converge.
 \arxiv{Indeed, as previously, %
 if $u+v<0$, then    $p=-u-v\in(0,-2u)$ and $p$-th moments converge.  If $u+v\geq 0$, then   $R_n^{-u-v}\leq 1$. }
In particular, as $N\to\infty$,
\[
\EE[\left(R_{N-k-2}(\vv \zeta)+\zeta_1\dots \zeta_{N-k-1}R_k(\vv \xi)\right)^{-u-v}]  \to \EE[R^{-u-v}],
\]
as the last term is negligible, with $$\EE\left[\big(\zeta_1\dots \zeta_{N-k-1}
     R_k(\vv \xi)\big)^p\right] = \EE[\zeta_1^p]^{N-k-1}\EE[ R_k(\vv \xi)^p]$$ converging to $0$ exponentially fast as $N\to\infty$.
To complete the negligibility argument, we distinguish the cases
$u+v\ge 0$, where we take $p=-u$, and $u+v<0$, where we take
$p=-u-v$. By the triangle inequality for the metric \eqref{p-metric},
\[
\begin{aligned}
\| R_{N-k-2}(\vv \zeta)
&+\zeta_1\cdots \zeta_{N-k-1}R_k(\vv \xi)-R\|_p  \\
&\le
\|R_{N-k-2}(\vv \zeta)-R\|_p
+\|\zeta_1\|_p^{\,N-k-1}\,
\|R_k(\vv \xi)\|_p
\longrightarrow 0,
\end{aligned}
\]
as $N\to\infty$. This completes the proof when $u+v<0$.

Suppose now that $u+v\ge0$. Then
\[
\left(R_{N-k-2}(\vv \zeta)
+\zeta_1\cdots \zeta_{N-k-1}R_k(\vv \xi)\right)^{-u-v}
\le 1,
\]
which implies uniform integrability, and $\zeta_1\cdots \zeta_{N-k-1}R_k(\vv \xi)\to 0$ in $p$-th moment and in probability. So
\[
\left(R_{N-k-2}(\vv \zeta)
+\zeta_1\cdots \zeta_{N-k-1}R_k(\vv \xi)\right)^{-u-v}
\to R^{-u-v}
\]
in probability. Hence, by uniform integrability, the two sequences have the same limit of expectations.

 In view of Corollary \ref{Cor:from-perp},
  after canceling the common factor $\EE[R^{-u-v}]$, we obtain \eqref{limNk-walk}.  \qed

\subsubsection{Proof of Theorem~\ref{Thm-half-lim} on the coexistence line} \label{Sec:Thm1.3(iv)}
{  The case $u=v=0$ is trivial, as $\vv L$ is simply a sum of i.i.d. log-gamma random variables. Thus, it suffices to consider the case $u=v<0$.}
Denote by $w>0$ the common value of $-u=-v$.
From \eqref{pert-MV} we obtain
\begin{equation*}\label{Coex+}
   \mathbb E\left[e^{-2\sum_{j=1}^k\,t_j(\vv L_j-\vv L_{j-1})}\right]=\tfrac{\GG_{t_1,\dots,t_k,0,0\dots}(\alpha,-w,-w)}{\GG_{N}(\alpha,-w,-w)}\\
=
Q_N
\prod_{r=1}^k \tfrac{\Gamma(\alpha+w+2t_r)}{\Gamma(\alpha+w)} ,
\end{equation*}
where
$$
Q_N=\tfrac{\EE \left[\left(R_{N-k-1}+\zeta_1\dots \zeta_{N-k-1}(1+   \xi_1+\dots+  \xi_1\dots  \xi_k)\right)^{2w}\right]}{\EE[R_{N-1}^{2w}]}
$$
with independent  {  random variables}  $\zeta_j\simeq \mathrm{Beta}_{II}(\alpha-w,\alpha+w)$, $j=1,\dots,N-k$,
and $  \xi_j\simeq \mathrm{Beta}_{II}(\alpha-w,\alpha+w+2t_{k-j+1})$ for $j=1,\dots,k$.

We have $R_N\nearrow R$ where $R-1\simeq\mathrm{Beta}_{II}(\alpha-w,2w)$ (see \cite{chamayou1991explicit}), so $\|R_N\|_{2w}\to \infty$. Recalling notation \eqref{p-metric}, we have
\begin{multline*}
  \|R_{N-1}\|_{2w}-\|R_k\|_{2w}\leq \|R_{N-k-1}\|_{2w}
\\ \leq \|R_{N-k-1}+\zeta_1\dots \zeta_{N-k-1}(1+   \xi_1+\dots+  \xi_1\dots  \xi_k)\|_{2w}
\\ \leq \|R_{N-k-1}\|_{2w}+\|\zeta_1\dots \zeta_{N-k-1}(1+   \xi_1+\dots+  \xi_1\dots  \xi_k)\|_{2w}
\\ \leq    \|R_{N-1}\|_{2w}+\|R_k\|_{2w}+\|\zeta_1\dots \zeta_{N-k-1}(1+   \xi_1+\dots+  \xi_1\dots  \xi_k)\|_{2w}.
\end{multline*}
Thus $Q_N\to 1$ and $$
\lim_{N\to\infty}  \mathbb E\left[e^{-2\sum_{i=1}^k\,t_i(\vv L\topp N_i-\vv L\topp N_{i-1})}\right]= \prod_{r=1}^k \tfrac{\Gamma(\alpha+w+2t_r)}{\Gamma(\alpha+w)} $$
corresponds to $\tilde v=-w$, an  inverse-Gamma multiplicative random walk. \qed

\subsubsection{Exponential functionals of random walks}\label{sec:AuxFacts}
Our proof of Theorem~\ref{Thm-half-lim}  at the boundary of the maximal current region relies on additional facts concerning exponential functionals of symmetric random walks.
   We invoke  results \cite[Theorem~1]{Hirano-1997} and \cite[Theorem~2.7]{Xu-Wei:2023}, which   cover the functions needed in our setting. For the reader's convenience, we recall these results in Appendix~\ref{Se:ExpFun}.
  We note that for compactly supported functions, related  asymptotics appeared in \cite[Theorems~A and~B]{lePage1997local}.

We use these results to prove the following.

\begin{lemma}\label{Conj}
 Let  $\tau_1,\tau_2,\dots,$ be i.i.d. random variables,  $\tau_j\simeq\mathrm{Beta}_{II}(\alpha,\alpha)$. Define $P_k=\prod_{j=1}^k \tau_j$ and $R_n=1+P_1+\dots+P_n$.
For fixed   $\theta>0$, we have:
\begin{equation} \label{XuWei}
  \lim_{m\to\infty}  \sqrt{m}\, \EE[R_m^{-\theta}] =c
\end{equation}
for some $c\in(0,\infty)$.
Furthermore,
  for  $K>0$, we have
\begin{equation}\label{ratio-limit}
    \lim_{m\to\infty} \tfrac{\EE[(R_m+KP_m)^{-\theta}]}{\EE[R_m^{-\theta}]}=1.
\end{equation}
\end{lemma}

\begin{proof}

To prove \eqref{XuWei}, we apply \cite[Theorem 2.7]{Xu-Wei:2023}, see Theorem \ref{Thm:Xu-Wei-T1}, to the symmetric i.i.d.\ random variables \(X_j=\log \tau_j\) and the function \(F(x)=(1+x)^{-\theta}\), which is bounded on $(0,\infty)$. We verify that the assumptions listed in Section \ref{Se:ExpFun} are satisfied:
\begin{enumerate}[C.1]
\item[\ref{A2.1}:] We have  $\mathcal{D}_F=(0,\theta]$ with $\theta_F=\theta$.

\item[\ref{C2.2}:]  For \(0<\delta\le x\le y\), the mean value theorem yields
\[
0\le F(x)-F(y)
\le \tfrac{\theta}{(1+\delta)^{\theta+1}}(y-x)
=C_\delta (y-x).
\]

\item[ \ref{A2.5}:]
Using \eqref{C-S}, for $0\leq \la <\alpha$, we have
\[
\mathcal{L}_X(\lambda)
=\EE[e^{\lambda X}]
=\EE[\tau^\lambda]
=\tfrac{\Gamma(\alpha-\lambda)\Gamma(\alpha+\lambda)}
       {\Gamma(\alpha)^2}
\ge 1.
\]
Hence
$
\inf_{0\le \lambda\le { \theta}}\mathcal{L}_X(\lambda)=1,
$
and the infimum is attained at \(\Lambda=0\).

\item[\ref{C2.6}:] For symmetric random variables {  $\log \tau_j$ with second moments and density},  $\rho=1/2$.
\end{enumerate}
With $\rho$ and $\ell_1(n)$ as described in Remark \ref{Rem:Xuwei},  the conclusion \eqref{eq:2.3} of \cite[Theorem 2.7]{Xu-Wei:2023} gives \eqref{XuWei}.

\medskip
Next we verify that
 \begin{equation}\label{m-limit}
    \lim_{m\to\infty} \sqrt{m}\; \EE\left[ P_m R_m^{-1-\theta}\right] =0. %
\end{equation}
To prove  \eqref{m-limit}, we apply  \cite[Theorem 1]{Hirano-1997}, see Theorem \ref{Thm:Hirano-T1},  to the i.i.d. random variables $\xi_j=\log \tau_j$ and functions $f(x)=x$, $h(x)=(1+x)^{-1-\theta}$ and $W(x)=e^x$.
The assumptions in \cite[Theorem 1]{Hirano-1997}  are satisfied: conditions (a), (b), (c) hold because  $\EE[\log \tau]=0$, $\mathrm{Var}[\log \tau]=2 \psi_1(\alpha)$ and $\EE[\tau^\theta] = \tfrac{\Gamma(\alpha-\theta)\Gamma(\alpha+\theta)}{\Gamma(\alpha)^2}$.
Conditions (i), (ii) hold trivially, and condition (iii) holds with $a=1$, $\beta=1+\theta$, $\eps=\eta=1$, which satisfy \eqref{Hiii-beta} for $\theta>0$.

With $S_n=\sum_{k=1}^n \xi_k$, the conclusion \eqref{Hirano-concl} of this theorem says that
$ \EE\left[ P_m R_m^{-1-\theta}\right]\sim c m^{-3/2}$   as $m\to\infty$,
which implies \eqref{m-limit}.

\medskip
To derive \eqref{ratio-limit} from \eqref{XuWei} and \eqref{m-limit}, we apply under the expectation the standard mean-value estimate
\[
0 \le x^{-\theta}-(x+a)^{-\theta}
\le a\theta x^{-1-\theta},
\qquad a\ge 0,\; x>0.
\]
We obtain
$$
0\leq \tfrac{\EE[R_m^{-\theta}]-\EE[(R_m+K P_m)^{-\theta}] }{\EE[R_m^{-\theta}]}\leq \tfrac{\theta K}{\sqrt{m}\,\EE[R_m^{-\theta}]} \cdot \sqrt{m}\, \EE\left[ P_m R_m^{-1-\theta}\right]. %
$$
Since by \eqref{XuWei}, the sequence  $\{\sqrt{m}\,\EE[R_m^{-\theta}]\}_{m\in\ZZ_{\geq 1}}$ is bounded away from $0$,  by \eqref{m-limit}, the
right-hand side of the above expression converges to $0$ and \eqref{ratio-limit} follows.
\end{proof}

\subsubsection{Proof of Theorem~\ref{Thm-half-lim} at the boundary between the maximal current and low-density regions}\label{Sec:Thm1.3(v)}
Assume $u=0,v>0$.
From \eqref{pert-MV} we obtain
\begin{equation*}\label{v>0}
   \mathbb E\left[e^{-2\sum_{i=1}^k\,t_i(\vv L_i-\vv L_{i-1})}\right]=\tfrac{\GG_{t_1,\dots,t_k,0,0\dots}(\alpha,0,v)}{\GG_{N}(\alpha,0,v)}\\
= Q_N
\prod_{r=1}^k \tfrac{\Gamma(\alpha+2t_r)}{\Gamma(\alpha)},
\end{equation*}
where
\begin{equation}\label{QN(v)}
  Q_N=\tfrac{\EE \left[\left(R_{N-k-2}+\zeta_1\dots \zeta_{N-k-1}(1+ \tilde \xi_1+\dots+\tilde \xi_1\dots\tilde \xi_k)\right)^{-v}\right]}{\EE[R_{N-1}^{-v}]}
\end{equation}
with independent {  random variables} $\zeta_j\simeq \mathrm{Beta}_{II}(\alpha,\alpha)$, $j=1,\dots,N-k$,
and $\tilde \xi_j\simeq \mathrm{Beta}_{II}(\alpha,\alpha+2t_{k-j+1})$ for $j=1,\dots,k$. Here,  using notation \eqref{Rn} with $\vv y$ replaced by $\vv \zeta=(\zeta_1,\zeta_2,\dots)$, we put  $R_n=R_n(\vv \zeta)$.

We first verify that  $Q_N\geq 1$ whenever $t_j\geq 0$. Consider the following coupling between the i.i.d. random variables
$\zeta_j'\simeq \mathrm{Beta}_{II}(\alpha,\alpha)$ and  $\tilde \xi_j$
\[
\tilde \xi_j:=\tfrac{\tau_j(\alpha)}{\tau'_j(\alpha)+\tau_j'(2t_{k-j+1})}
\leq
\zeta_j':=\tfrac{\tau_j(\alpha)}{\tau_j'(\alpha)},
\]
where $\tau_j(\alpha)$, $\tau_j'(\alpha)$, and $\tau_j'(2t)$ are independent gamma random variables, $\tau_j(a)\simeq \mathrm{Gamma}(a)$, $j=1,\dots,k$. Let  $\vv \zeta'$ denote  an  independent copy of $\vv \zeta$, coupled with $\tilde{\vv \xi}$ as shown above. Then we have
\begin{multline*}
  R_{N-k-2}(\vv\zeta)+\zeta_1\dots \zeta_{N-k-1}R_k(\tilde{\vv \xi})
\\ \leq
R_{N-k-2}(\vv\zeta)+\zeta_1\dots \zeta_{N-k-1}R_k(\vv \zeta')
\stackrel{d}{=}
R_{N-1}(\vv\zeta).
\end{multline*}
Since $v>0$, the above inequality implies that the numerator defining $Q_N$ dominates the denominator, and hence  $Q_N\geq 1$.

To derive the upper bound for $Q_N$, we drop the second term in the numerator of \eqref{QN(v)}.
Using \eqref{XuWei} with $\theta=v$, we obtain
$$ Q_N\leq \tfrac{\EE \left[ R_{N-k-1}^{-v}\right]}{\EE[R_{N-1}^{-v}]} \sim \tfrac{\sqrt{N-1}}{\sqrt{N-k-1}} \to 1 \quad \mbox{  as $N\to\infty$}.
  $$
This concludes the proof. \qed

\subsubsection{Proof of Theorem~\ref{Thm-half-lim} at the boundary between the maximal current and high-density regions}  \label{Sec:Thm1.3(vi)}
  Assume   $u>0,v=0$ and take $t_1,\dots,t_k>-\alpha/2$, $t_{k+1}=0$.
 From \eqref{pert-MU} we obtain

\begin{equation}\label{poprawiny+}
   \tfrac{\GG_{t_1,\dots,t_k,0,0\dots}(\alpha,u ,0)}{\GG_{N}(\alpha,u,0)}
=\tfrac{ \Gamma(2\alpha)^{N-k-1}\prod_{j=1}^{k} \Gamma(2\alpha+2t_j)}{\EE[R_{N-1}^{-u}(\vv \tau)]\Gamma(\alpha)^{2N-2}}\,  I_N
\end{equation}
where
\[
I_N:=\int_{(0,\infty)^{N-1}}\,(1+x_1+x_1x_2+\ldots+x_1\ldots x_{N-1})^{-u}\,\prod_{j=1}^{N-1}\,\tfrac{x_j^{\alpha+2t_j-1}}{(1+x_j)^{2\alpha+2t_j}}\,\mathrm d\mathbf x
\]
with $t_j=0$ for $j>k$.
Applying the change of variables \eqref{subRSKlike} used in the proof of  Lemma~\ref{Lem:H=L}, %
we obtain
\begin{multline}
    \label{IN-sing}
    I_N =\int_{(0,\infty)^{N-1}} \tfrac{y_1^{\alpha+u-1}}{(1+y_1)^{\alpha+u}} \Biggl( \prod_{j=2}^{k+1} \tfrac{y_j^{2\alpha+2t_{j-1}-1}}{(1+y_j)^{2\alpha+2t_{j-1}}} \Biggr)\,\Biggl( \prod_{j=1}^{k} R_j^{2t_{j+1}-2t_j}(\vv y) \Biggr)\\\ \cdot  \Biggl( \prod_{j=k+2}^{N-1} \tfrac{y_j^{2\alpha-1}}{(1+y_j)^{2\alpha}} \Biggr) R_{N-1}^{-\alpha}(\vv y)\,  \mathrm d\mathbf{y}.
\end{multline}
Note that the integral is finite due to the presence of the factor {  $R_{N-1}^{-\alpha}(\vv y)\leq (y_1\dots y_{N-1})^{-\alpha}$}.
\arxiv{
\begin{proof}[Proof of \eqref{IN-sing}] To shorten notation, write   $R_j=R_j(\vv y)$. The irregular substitutions are:
$$x_1=\frac{y_2}{R_2}, \quad 1+x_1=\frac{R_1}{R_2}(1+y_2),\quad  x_{N-1}=\frac{R_{N-2}}{y_1\dots y_{N-1}}, \quad 1+x_{N-1}=\frac{R_{N-1}}{y_1\dots y_{N-1}} .$$
Also write $R_{N-1}(\vv x)=\frac{R_1}{y_1}$. The Jacobian  is:
$J=\tfrac{R_1}{y_1^2 y_2\dots y_{N-1}R_{N-1}}$, see \eqref{B.8J}.
Using \eqref{subRSKlike}, see \eqref{xfrac2yfrac}, the   non-boundary factors  become:
$$\frac{x_j^{\alpha+2 t_j-1}}{(1+x_j)^{2\alpha+2 t_j}} = \frac{y_{j+1}^{\alpha +2
   t_j-1}}{(1+y_{j+1})^{2\alpha+2t_j}} \cdot \frac{R_{j-1}^{\alpha +2
   t_j-1}}{R_j^{2\alpha+2t_j}} R_{j+1}^{\alpha+1}.
   $$
  We obtain:
\begin{multline*}
   I_N=\int_{(0,\infty)^{N-1}}\,(R_{N-1}(\vv x))^{-u}\,\tfrac{x_1^{\alpha+2t_1-1}}{(1+x_1)^{2\alpha+2t_1}} \left(\prod_{j=2}^{N-2}\,\tfrac{x_j^{\alpha+2t_j-1}}{(1+x_j)^{2\alpha+2t_j}}\right) \, \tfrac{x_{N-1}^{\alpha -1}}{(1+x_{N-1})^{2\alpha}}\mathrm d\mathbf x
\\ =  \int_{(0,\infty)^{N-1}}\frac{y_1^u}{R_1^u}
\left(\prod_{j=1}^{N-2} \tfrac{y_{j+1}^{\alpha+2 t_j-1}}{(1+y_{j+1})^{2\alpha+2 t_j}}\right)\frac{R_2^{\alpha+1}}{R_1^{2\alpha+2 t_1} }
\left(\prod_{j=2}^{N-2}
\frac{R_{j-1}^{\alpha +2
   t_j-1}}{R_j^{2\alpha+2t_j}} R_{j+1}^{\alpha+1}\right)
 \frac{R_{N-2}^{\alpha-1}}{(y_1\dots y_{N-1})^{\alpha-1}}  \frac{(y_1\dots y_{N-1})^{2\alpha}}{R_{N-1}^{2\alpha}} \\\cdot  \frac{R_1}{y_1 (y_1\dots y_{N-1})R_{N-1}} \d\mathbf y.
\end{multline*}
After cancellations, for  $N\geq k +3$ the integrand becomes
$$
    \frac{y_1^{\alpha+u-1}}{R_1^{\alpha+u+2t_1-2t_2}}
    \left(\prod_{j=2}^{k+1} \frac{y_j^{2\alpha+2t_{j-1}-1}}{(1+y_j)^{2\alpha+2t_{j-1}}}\right) \cdot  \left(\prod_{j=k+2}^{N-1}
    \frac{y_j^{2\alpha-1}}{(1+y_j)^{2\alpha}}\right) \cdot  \left(\prod_{j=2}^{k} R_j^{2t_{j+1}-2t_j}\right)
      \cdot R_{N-1}^{-\alpha}.
$$
Since $R_1=1+y_1$, we obtain
\begin{equation*}
\frac{y_1^{\alpha+u-1}}{(1+y_1)^{\alpha+u}}
    \left(\prod_{j=2}^{k+1} \frac{y_j^{2\alpha+2t_{j-1}-1}}{(1+y_j)^{2\alpha+2t_{j-1}}}\right) \cdot
    \left(\prod_{j=k+2}^{N-1} \frac{y_j^{2\alpha-1}}{(1+y_j)^{2\alpha}}\right) \cdot
    \left(\prod_{j=1}^{k} R_j^{2t_{j+1}-2t_j}\right)  \cdot R_{N-1}^{-\alpha}.
\end{equation*}
This gives \eqref{IN-sing}. \end{proof}
 }
We rewrite $I_N$ as
\begin{multline*}
	I_N
	 =\int_{(0,\infty)^{k+1}}\,\tfrac{y_1^{\alpha+u-1}}{(1+y_1)^{\alpha+u}} \Biggl( \prod_{j=2}^{k+1} \tfrac{y_j^{2\alpha+2t_{j-1}-1}}{(1+y_j)^{2\alpha+2t_{j-1}}} \Biggr) \Biggl( \prod_{j=1}^{k} R_j(\mathbf{y})^{2t_{j+1}-2t_j} \Biggr) \\
	 \cdot \Biggl(\int_{(0,\infty)^{N-k-2}}\,R_{N-1}(\mathbf{y})^{-\alpha}\, \Biggl(\prod_{j=k+2}^{N-1} \tfrac{y_j^{2\alpha-1}}{(1+y_j)^{2\alpha}} \Biggr)  \, \mathrm dy_{k+2}\ldots\mathrm dy_{N-1}\Biggr)
    \\ \d y_1\ldots \d y_{k+1}.
\end{multline*}

Write  $R_{N-1}(\vv y)=R_{k}(\vv y)+P_{k+1}(\vv y) R_{N-k-2}(y_{k+2},\dots y_{N-1})$, where $P_{n}(\vv y)=\prod_{i=1}^{n} y_i$ so that the inner integral becomes

\begin{multline*}
\int_{(0,\infty)^{N-k-2}}\left(R_{k}(\vv y)+P_{k+1}(\vv y) R_{N-k-2}(y_{k+2},\dots y_{N-1})\right)^{-\alpha}\, \\ \cdot  \Biggl(\prod_{j=k+2}^{N-1} \tfrac{y_j^{2\alpha-1}}{(1+y_j)^{2\alpha}} \Biggr)  \, \mathrm dy_{k+2}\ldots\mathrm dy_{N-1}.
\end{multline*}

We now perform the change of variables
$$
y_{k+1+j}=\frac{1}{s_j}, \qquad j=1,\dots,N-k-2,$$
in the inner integral{, with the Jacobian $J=s_1^{-2}\dots s_{N-k-2}^{-2}$}. This leaves $R_k(\vv y)$, $P_{k+1}(\vv y)$,  and the outer integral unchanged, while
 $$R_{N-k-2}(y_{k+2},\dots y_{N-1}) =\tfrac{R_{N-k-2}(\vv s)}{s_1\dots s_{N-k-2}} =\tfrac{R_{N-k-2}(\vv s)}{P_{N-k-2}(\vv s)} . $$
After a calculation, we obtain
\begin{multline*}
   I_N
	 =\int_{(0,\infty)^{k+1}}\,\tfrac{y_1^{u-1}}{(1+y_1)^{\alpha+u}} \Biggl( \prod_{j=2}^{k+1} \tfrac{y_j^{\alpha+2t_{j-1}-1}}{(1+y_j)^{2\alpha+2t_{j-1}}} \Biggr) \Biggl( \prod_{j=1}^{k} R_j(\mathbf{y})^{2t_{j+1}-2t_j} \Biggr) \\
	\cdot \Biggl(\int_{(0,\infty)^{N-k-2}}\left(\tfrac{R_{k}(\vv y)P_{N-k-2}(\vv s)}{P_{k+1}(\vv y)} + R_{N-k-2}(\vv s) \right)^{-\alpha}\,
    \\ \cdot  \Big(\prod_{j=1}^{N-k-2} \tfrac{s_j^{\alpha-1}}{(1+s_j)^{2\alpha}} \Big)   \d s_1\dots\d s_{N-k-2}\Biggr) \mathrm d y_1\ldots\mathrm d y_{k+1}.
\end{multline*}
We can now express $I_N$  in terms of another collection of independent Beta random variables. Let
$\zeta_0\simeq \mathrm{Beta}_{II}(u,\alpha)$ and
$\zeta_r\simeq \mathrm{Beta}_{II}(\alpha+2 t_r,\alpha)$ for $r=1,\dots,k$ be independent. %
Let
$\tau_j\simeq\mathrm{Beta}_{II}(\alpha,\alpha)$, $j\ge 1$, be i.i.d. (and independent from the family $\{\zeta_j\}$).
Then, with $\vv \zeta=(\zeta_0,\dots,\zeta_k)$ and $\vv \tau=(\tau_1,\dots,\tau_{N-k-2})$, we get
\begin{multline*}
I_N=\tfrac{\Gamma(\alpha)\Gamma(u)}{\Gamma(\alpha+u)}\cdot \left(\prod_{r=1}^k \tfrac{\Gamma(\alpha)\Gamma(\alpha+2 t_r)}{\Gamma(2\alpha+2 t_r)}\right) \cdot \tfrac{\Gamma(\alpha)^{2N-2k-4}}{\Gamma(2\alpha)^{N-k-2}}
\\ \cdot \EE\Big[ \left(R_{N-k-2}(\vv \tau)+ \tfrac{R_k(\vv \zeta)P_{N-k-2}(\vv \tau)}{P_{k+1}(\vv \zeta)} \right)^{-\alpha}\prod_{j=1}^kR_j^{2t_{j+1}-2t_j}(\vv \zeta)\Big].
\end{multline*}
We obtain, see \eqref{poprawiny+},
\begin{multline*}
    \tfrac{\GG_{t_1,\dots,t_k,0,0\dots}(\alpha,u ,0)}{\GG_{N}(\alpha,u,0)} =
    \tfrac{\Gamma(2\alpha)\Gamma(u)}{\Gamma(\alpha+u)\Gamma(\alpha)^{k+1}}
  \left(\prod_{r=1}^k  \Gamma(\alpha+2 t_r)\right)
\\ \cdot  \tfrac{\EE\left[ \big(R_{N-k-2}(\vv \tau)+ \frac{R_k(\vv \zeta)P_{N-k-2}(\vv \tau)}{P_{k+1}(\vv \zeta)} \big)^{-\alpha}\prod_{j=1}^kR_j^{2t_{j+1}-2t_j}(\vv \zeta)\right] }{\EE[R_{N-1}^{-u}(\vv \tau)]}.
\end{multline*}
The preceding argument extends directly to $\EE[R_{N-1}^{-u}(\tau)]$, corresponding to the case $k=0$. Hence,
\begin{multline}\label{poprawiny1}
    \tfrac{\GG_{t_1,\dots,t_k,0,0\dots}(\alpha,u ,0)}{\GG_{N}(\alpha,u,0)}
    \\ =
     \left(\prod_{r=1}^k  \tfrac{\Gamma(\alpha+2 t_r)}{\Gamma(\alpha)} \right)
   \tfrac{\EE\left[ \big(R_{N-k-2}(\vv \tau)+ \frac{R_k(\vv \zeta)P_{N-k-2}(\vv \tau)}{P_{k+1}(\vv \zeta)} \big)^{-\alpha}\prod_{j=1}^kR_j^{2t_{j+1}-2t_j}(\vv \zeta)\right] }{\EE\left[ \big(R_{N-2}(\vv \tau)+ \frac{P_{N-2}(\vv \tau)}{\zeta_0}\big)^{-\alpha}\right]}.
\end{multline}

The result follows from the following lemma, to be proved later.

\begin{lemma}\label{Lem:poprawiny}
For every fixed $k\ge 0$ and $\vv \zeta=( \zeta_0,\dots,\zeta_k)$,
\begin{equation}\label{poprawiny2}
\lim_{m\to\infty}
\tfrac{
\EE\left[
\left(
R_m(\vv\tau)
+\frac{R_k(\vv\zeta)P_m(\vv\tau)}
       {P_{k+1}(\vv\zeta)}
\right)^{-\alpha}
\prod_{j=1}^k
R_j^{\,2t_{j+1}-2t_j}(\vv\zeta)
\right]
}
{\EE\left[R_m^{-\alpha}(\vv\tau)\right]}
=
\EE\!\left[
\prod_{j=1}^k
R_j^{\,2t_{j+1}-2t_j}(\vv\zeta)
\right].
\end{equation}
\end{lemma}

Indeed, by \eqref{XuWei} with $\theta=\alpha$,
\begin{equation}\label{granica1}
    \lim_{N\to\infty}
\tfrac{\EE\,[R_{N-k-2}^{-\alpha}(\vv\tau)]}
     {\EE[R_{N-2}^{-\alpha}(\vv\tau)]}
=1.
\end{equation}
Therefore, dividing the numerator and denominator on the right-hand side of \eqref{poprawiny1} by $\EE[R_{N-k-2}^{-\alpha}(\vv\tau)]$ and
$\EE[R_{N-2}^{-\alpha}(\vv\tau)]$ respectively, and using \eqref{granica1}, from Lemma \ref{Lem:poprawiny} (applied twice: with  $m=N-k-2$ for fixed $k$, and $m=N-2$, $k=0$) we obtain

\begin{equation}\label{poprawiny3}
\lim_{N\to\infty}
\tfrac{\GG_{t_1,\dots,t_k,0,0,\dots}(\alpha,u,0)}
     {\GG_N(\alpha,u,0)}
=
\left(
\prod_{r=1}^k
\tfrac{\Gamma(\alpha+2t_r)}{\Gamma(\alpha)}
\right)
\EE\left[
\prod_{j=1}^k
R_j^{\,2t_{j+1}-2t_j}
(\zeta_0,\dots,\zeta_{j-1})
\right].
\end{equation}
The right-hand side coincides with \eqref{half-space-beta+}. Since
$t_1,\dots,t_k>-\alpha/2$, convergence of Laplace transforms implies weak
convergence, completing the proof.

\arxiv{More explicitly, to obtain \eqref{poprawiny3} we write
\begin{multline*}
    \frac{\EE\left[ \big(R_{N-k-2}(\vv \tau)+ \tfrac{R_k(\vv \zeta)P_{N-k-2}(\vv \tau)}{P_{k+1}(\vv \zeta)} \big)^{-\alpha}\prod_{j=1}^kR_j^{2t_{j+1}-2t_j}(\vv \zeta)\right] }{\EE\left[ \big(R_{N-2}(\vv \tau)+ \tfrac{P_{N-2}(\vv \tau)}{\zeta_0}\big)^{-\alpha}\right]}
    \\ = \frac{\EE\left[ \big(R_{N-k-2}(\vv \tau)+ \tfrac{R_k(\vv \zeta)P_{N-k-2}(\vv \tau)}{P_{k+1}(\vv \zeta)} \big)^{-\alpha}\prod_{j=1}^kR_j^{2t_{j+1}-2t_j}(\vv \zeta)\right] }{\EE\left[ R_{N-k-2}^{-\alpha}(\vv \tau) \right]}\\ \cdot   \frac{\EE\left[ R_{N-2}^{-\alpha}(\vv \tau) \right] }{\EE\left[ \big(R_{N-2}(\vv \tau)+ \tfrac{P_{N-2}(\vv \tau)}{\zeta_0}\big)^{-\alpha}\right]}. \cdot   \frac{\EE\left[ R_{N-k-2}^{-\alpha}(\vv \tau) \right]}{\EE\left[ R_{N-2}^{-\alpha}(\vv \tau) \right]}
\\ \sim \EE\!\left[
\prod_{j=1}^k
R_j^{\,2t_{j+1}-2t_j}(\vv\zeta)
\right]. \cdot  1 \cdot  1 \quad \mbox{as $N\to\infty$}.
\end{multline*}
}
\begin{proof}[Proof of Lemma~\ref{Lem:poprawiny}]
    Since $\zeta_0,\dots,\zeta_k, \tau_1,\dots,\tau_{N-1}$ are independent, by a trivial bound $(R+A)^{-\alpha}\leq R^{-\alpha}$ for $R\geq 1,A\geq 0$, we see that
    the expression under the limit in \eqref{poprawiny2} is bounded above by the right-hand side.

    For the lower bound, we use the  limit \eqref{ratio-limit} as follows. For  $K>0$ consider the event
    $$A_K=\left\{\tfrac{R_k(\vv \zeta)}{P_{k+1}(\vv \zeta)}\leq K\right\}.$$
    The numerator in the limit in \eqref{poprawiny2} is then bounded from below by the following expression.
\begin{multline*}
       \int_{A_K}  \left(R_{m}(\vv \tau)+ \tfrac{R_k(\vv \zeta)P_{m}(\vv \tau)}{P_{k+1}(\vv \zeta)} \right)^{-\alpha}\left(\prod_{j=1}^kR_j^{2t_{j+1}-2t_j}(\vv \zeta)\right)  \d \mathbb{P}
       \\ \geq
     \int_{A_K} \left(R_{m}(\vv \tau)+ K P_{m}(\vv \tau) \right)^{-\alpha}\left(\prod_{j=1}^kR_j^{2t_{j+1}-2t_j}(\vv \zeta)\right)  \d \mathbb{P}
     \\=
     \EE\left[ \left(R_{m}(\vv \tau)+ K P_{m}(\vv \tau) \right)^{-\alpha}\right]
     \int_{A_K} \Big(\prod_{j=1}^kR_j^{2t_{j+1}-2t_j}(\vv \zeta)\Big)  \d \mathbb{P},
\end{multline*}
where in the last step we used independence.
By \eqref{ratio-limit} with $\theta=\alpha$, we see that
    \begin{multline*}
        \liminf_{m\to\infty}  \tfrac{\EE\left[ \left(R_{m}(\vv \tau)+
    \tfrac{R_k(\vv \zeta)P_{m}(\vv \tau)}{P_{k+1}(\vv \zeta)} \right)^{-\alpha}\prod_{j=1}^kR_j^{2t_{j+1}-2t_j}(\vv \zeta)\right] }{\EE[R_{m}^{-\alpha}(\vv \tau)]}
   \\   \ge   \liminf_{m\to\infty} \tfrac{\EE[ \left(R_{m}(\vv \tau)+ K P_{m}(\vv \tau) \right)^{-\alpha}] }{\EE[R_{m}^{-\alpha}(\vv \tau)]}  \int_{A_K} \Big(\prod_{j=1}^kR_j^{2t_{j+1}-2t_j}(\vv \zeta)\Big)  \d \mathbb{P}
  \\   =  \int_{A_K} \Big(\prod_{j=1}^kR_j^{2t_{j+1}-2t_j}(\vv \zeta) \Big) \d \mathbb{P}.
    \end{multline*}
   To end the proof, we let $K\to\infty$.
\end{proof}

 \appendix
\section{Properties of Gamma function}

We use the following properties of gamma function:
\begin{lemma}\label{Lem:prop-Gamma}{$ $}

\begin{enumerate}  [(i)]
\item If $\alpha,  x>0$, then   \cite[\href{http://dlmf.nist.gov/5.8.E3}{(5.8.3)}]{NIST:DLMF}
\begin{equation}\label{G-prod}
  |\Gamma(\alpha+\i x)|^2 = \tfrac{\Gamma(\alpha)^2}{\prod_{k=0}^\infty\left(1+\frac{x^2}{(\alpha+k)^2}\right)}.
\end{equation}
\item If $x\ne 0,-1,\dots$, then  \cite[\href{http://dlmf.nist.gov/5.2.E2}{(5.2.2)}]{NIST:DLMF}
\begin{equation}\label{digamma}
   \tfrac{\partial \Gamma(x)}{\partial x}=\psi(x)\Gamma(x),
\end{equation}
where $\psi$ is the {\em  digamma function}.
\item Function $x\mapsto \Gamma(\alpha+x)\Gamma(\alpha-x)$  is convex on $(-\alpha,\alpha)$ and in particular it is
a strictly decreasing function on $(-\alpha,0]$ with the minimum $\Gamma(\alpha)^2$ at $x=0$.
\item If  $\alpha>0$ and $0<p<\alpha$, then
\begin{equation}\label{C-S}
   \Gamma(\alpha)^2< \Gamma(\alpha+p)\Gamma(\alpha-p).
\end{equation}

\item For a Pochhammer  symbol
$(a)_j =\Gamma(j+a)/\Gamma(a)$ with $a\not\in\ZZ_{\leq -1}$, we have
\begin{equation}
    \label{InvPochh}
    (-a)_j=\tfrac{(-1)^j \Gamma (1+a )}{\Gamma (1+a-j)}, \quad j=0,1,\dots.
\end{equation}
\end{enumerate}
\end{lemma}

\arxiv
{
Formula \eqref{G-prod}  is a consequence of  a more general \cite[\href{http://dlmf.nist.gov/5.8.E5}{(5.8.5)}]{NIST:DLMF}
which for $m=2$ implies that with complex $a_1=a_2=\alpha\not\in\ZZ_{\leq 0}$  and $b_1=\alpha-z\not\in\ZZ_{\leq 0}$, $b_2=\alpha+z\not\in\ZZ_{\leq 0}$  we have
$$\Gamma(\alpha-z)\Gamma(\alpha+z)=\Gamma(\alpha)^2 \prod_{k=0}^\infty \frac{(\alpha+k)^2}{(\alpha+k-z)(\alpha+k+z)}.$$

Property  (iii)  is a consequence of log-convexity, as   the second derivative of
$$
\log \Gamma(\alpha+x)\Gamma(\alpha-x)=2 \log \Gamma(\alpha)- \sum_{k=0}^\infty \log \left(1-\tfrac{x^2}{(\alpha+k)^2}\right)$$
is
$$
\sum_{k=0}^\infty \frac{1}{(\alpha+k+x)^2}+\sum_{k=0}^\infty \frac{1}{(\alpha+k-x)^2}>0.
$$
Inequality \eqref{C-S}  follows from  strict log-convexity of $x\mapsto \Gamma(\alpha+x)\Gamma(\alpha-x)$ as we have equality at $p=0$.
}

\section{Useful integrals}
The
de-Branges  Beta integral
   \cite{de1972gauss}, as stated in \cite[(1.6.12)]{koekoek2010hypergeometric}    is as follows:
 If $a$, $b$ and $c$ are positive except possibly for a pair of complex conjugates with positive real parts, then
\begin{equation}\label{GG}
\tfrac{1}{2\pi}\int_0^\infty \tfrac{|\Gamma(a+\i x)\Gamma(b+\i x)\Gamma( c+\i x)|^2}{|\Gamma(2 \i x)|^2}\d x =
\Gamma(a+b)\Gamma(a+c)\Gamma(b+c).
\end{equation}

We also need an integral \cite[(1.6.3)]{koekoek2010hypergeometric}
\cite[\href{http://dlmf.nist.gov/5.13.E3}{(5.13.3)}]{NIST:DLMF}
\begin{equation}\label{koek:1.6.3}
\tfrac{1}{2\pi}\int_{-\infty}^\infty \Gamma(a+\i t)\Gamma(b+\i t)\Gamma(c-\i t)\Gamma( d-\i t)\d t=
\tfrac{\Gamma(a+c)\Gamma(a+d)\Gamma(b+c)\Gamma(b+d)}{\Gamma(a+b+c+d)},
\end{equation}
which holds if  $\re(a),\re(b),\re(c),\re(d)>0$. { (This identity is known as Barnes' first lemma; for a short proof, see \cite[page 43]{goldfeld2011kontorovich}.)}

De-Branges-Wilson integral   \cite{de1972gauss}, \cite{wilson1980some}, see also \cite[\href{http://dlmf.nist.gov/5.13.E5}{(5.13.5)}]{NIST:DLMF},  is
{  %

\begin{equation}\label{WI+}\tfrac{1}{4\pi}\int_{-\infty}^{\infty}\tfrac{\prod_{k=1}^{4}\Gamma\left(a_{k}+\i t%
\right)\Gamma\left(a_{k}-\i t\right)}{\Gamma\left(2\i t\right)\Gamma\left(-2\i t%
\right)}\,\mathrm{d}t=\tfrac{\prod_{1\leq j<k\leq 4}\Gamma\left(a_{j}+a_{k}%
\right)}{\Gamma\left(a_{1}+a_{2}+a_{3}+a_{4}\right)}.
\end{equation}
This formula is valid for all complex $a_j$ such that $\re(a_j)>0$. Noting that the integrand is a symmetric function of variable $t$, we may consider the same integral over $(0,\infty)$.
}
The limit case  $a_4\to\infty$ of \eqref{WI+} gives
\begin{equation}\label{GG+}
\tfrac{1}{2\pi}\int_0^\infty \tfrac{\Gamma(a+\i x)\Gamma(a-\i x)\Gamma(b+\i x)\Gamma(b-\i x)\Gamma(c+\i x)\Gamma( c-\i x)}{|\Gamma(2 \i x)|^2}\d x =
\Gamma(a+b)\Gamma(a+c)\Gamma(b+c),
\end{equation}
which extends \eqref{GG} to complex parameters $a,b,c$ with positive real part. This formula can also be obtained from \eqref{GG} by analytic continuation;
compare \cite[page 697]{wilson1980some}.

\begin{lemma}\label{Lem:H=L} Recall notation \eqref{Rn}. For $k=1,2\dots,$ $\alpha>0$, $u>-\alpha$, $ v< u$, $|v|<\alpha$   and $t_1,\dots,t_k \ge 0$,  the following integrals:
\begin{equation}\label{Lk}
    L_k = \int_{(0,\infty)^k} \tfrac{ \prod_{i=1}^k x_i^{\alpha+v+2t_i-1} }{ \big(R_k(\vv x)\big)^{u+v} (1+x_k)^{\alpha-v+2t_k} \prod_{i=1}^{k-1} (1+x_i)^{2\alpha+2t_i} } \,\mathrm{d}x_1 \dots \mathrm{d}x_k
\end{equation}
and
\begin{equation}\label{Hk}
    H_k = \int_{(0,\infty)^k} \tfrac{ y_1^{u-v-1} \left( \prod_{j=2}^k y_j^{\alpha-v+2t_{j-1}-1} \right) \prod_{j=1}^{k} \big(R_j(\vv y)\big)^{2t_{j+1}-2t_j} }{  (1+y_1)^{\alpha+u} \prod_{j=2}^k (1+y_j)^{2\alpha+2t_{j-1}} } \,\mathrm{d}y_1 \dots \mathrm{d}y_k,
\end{equation}
 are finite and equal. %
 \end{lemma}

 \begin{proof}
 To show that integral  \eqref{Lk} is finite, we note  that if $u+v>0$ then $\d x_k$ integral is finite as $R_k(\vv x)\geq x_1\dots x_{k-1}(1+x_k)$ gives
$$\tfrac{x_k^{\alpha+v+2t_k-1}}{(1+x_k)^{\alpha-v+2 t_k}R_k(\vv x)^{u+v}}
\leq \tfrac{x_k^{\alpha+v+2t_k-1}}{(1+x_k)^{\alpha+u+2 t_k}
(x_1\dots x_{k-1})^{u+v}}.$$

On the other hand if $u+v\leq 0$, then $v<0$ and  in \eqref{Lk} we have $R_k(\vv   x)^{-u-v}\leq C_k \sum_{r=0}^k (\prod_{j=1}^r x_j)^{-u-v}$, so the integral is bounded by a finite sum of finite beta-integrals.

 To verify that the integrals are equal, we  perform the following change of variables in $L_k$:

\begin{align}\label{subRSKlike}
    x_j = \tfrac{y_{j+1} R_{j-1}(\vv y)}{R_{j+1}(\vv y)} \quad \text{for } j = 1, \dots, k-1, \quad\mbox{and }\;
    x_k = \tfrac{R_{k-1}(\vv y)}{y_1 y_2 \dots y_k}.
\end{align}

To compute the Jacobian $J=J(y_1,\ldots,y_k)$ of this change of variables, we note that $\vv x:=(x_1,\dots,x_k)=(\phi_4\circ\dots\circ\phi_1)(\vv y)$, where
\begin{multline*}
   \phi_1(\vv y):=\Big(1+\sum_{j=1}^m\,\prod_{i=1}^j\,y_i\Big)_{\,m=1,\ldots,k}, \hfill
  \\ \phi_2(\vv R):=(\tfrac1{R_1}-\tfrac{1}{R_2},\tfrac{1}{R_2}-\tfrac{1}{R_3},\ldots,\tfrac{1}{R_{k-1}}-\tfrac{1}{R_k},\tfrac{1}{R_k}), \hfill \\
   \phi_3(\vv s):=\Big(\tfrac{s_m}{1-\sum_{j=1}^k s_j}\Big)_{\,m=1,\ldots,k}, \quad
     \phi_4(\vv q):=(q_1,q_2/q_1,\ldots,q_k/q_{k-1}). \hfill
\end{multline*}

Clearly, $J$ is the product of the consecutive Jacobians $J_1,\dots,J_4$ of the transformations $\phi_1,\dots,\phi_4$.
Denote $p_j=\prod_{i=1}^j\,y_i$.  We have
\[
J_1(\vv y)=\prod_{m=1}^{k-1}\,p_m,\quad J_2(\vv R)=\prod_{m=1}^k\,R_m^{-2}\quad\mbox{and}\quad J_4(\vv q)=\prod_{m=1}^{k-1}\,q_m^{-1}.
\]
 To compute the Jacobian $J_3$, we denote $S=\sum_{j=1}^k\,s_j$. Then, {  with $\vv q=\phi_3(\vv s)$,} we get
\[
\tfrac{\partial q_m}{\partial s_j}=\begin{cases} \tfrac{s_m}{(1-S)^2},\quad j\neq m,\\
                                                 \tfrac1{1-S}+\tfrac {s_m}{(1-S)^2},\quad j=m.
                                                 \end{cases}
\]

Using  matrix determinant lemma   $\det(I+\vv a \vv b^T) = 1+\vv b^T \vv a$, we obtain
\[
J_3(\vv s)=\det\left(\tfrac1{1-S}\,I+\tfrac1{(1-S)^2}\,\mathbf 1\,\vv s^T\right)=(1-S)^{-k-1}=\left(\tfrac{1+y_1}{y_1}\right)^{k+1}.
\]

We also note that $q_m=\tfrac{(1+y_1)p_{m+1}}{y_1R_mR_{m+1}}$, $m=1,\ldots,k-1$. Consequently,
\begin{multline*}
    J=\left(\prod_{m=1}^{k-1}\,p_m\right)\,\left(\prod_{m=1}^k\,R_m^{-2}\right)\,\left(\tfrac{1+y_1}{y_1}\right)^{k+1}\,\left(\prod_{m=1}^{k-1}\,\tfrac{y_1R_mR_{m+1}}{(1+y_1)p_{m+1}}\right)
    \\=\tfrac{p_1}{p_k}\,\tfrac{1}{R_1R_k}\, \tfrac{(1+y_1)^2}{y_1^2}=\tfrac{1+y_1}{y_1p_kR_k}.
\end{multline*}
Thus, the Jacobian  is
\begin{equation}\label{B.8J}
    \mathrm{d}x_1 \dots \mathrm{d}x_k = \tfrac{1+y_1}{y_1 \left( \prod_{j=1}^k y_j \right) R_k(\vv y)} \,\mathrm{d}y_1 \dots \mathrm{d}y_k.
\end{equation}
We also note that
\[
1+x_j=\tfrac{R_{j+1}+y_{j+1}R_{j-1}}{R_{j+1}}=\tfrac{(1+y_{j+1})R_j}{R_{j+1}},\;j=1,\ldots,k-1,\ \mbox{and}\quad 1+x_k=\tfrac{R_k}{y_1\ldots y_k}.
\]
{  Thus, for   $1<j<k$, we get
\begin{equation}\label{xfrac2yfrac}
   \tfrac{x_j^{\alpha+v+2 t_j-1}}{(1+x_j)^{2\alpha+2 t_j}} =
\tfrac{y_{j+1}^{\alpha +v+2
   t_j-1}}{(1+y_{j+1})^{2\alpha+2t_j}} \cdot \tfrac{R_{j-1}^{\alpha +v+2
   t_j-1}}{R_j^{2\alpha+2t_j}}\, R_{j+1}^{\alpha-v+1}.
\end{equation}
}

Moreover,
\[
R_k(\vv x)=1+x_1+\ldots+x_1\ldots x_k=\tfrac{1+y_1}{y_1}.
\]
\arxiv{Indeed, with $R_j=R_j(\vv y)$ we have
\begin{multline*}
    \frac{y_1}{R_1(\vv y)} R_k(\vv x)
    \\= \frac{y_1}{R_1}+\frac{y_1y_2}{R_1R_2}+\frac{y_1y_2y_3}{R_2R_3}+\dots+\frac{y_1\dots y_{k}}{R_{k-1}R_k}+\frac{1}{R_k}
= \frac{R_1-1}{R_1}+\frac{R_2-R_1}{R_1R_2}+\frac{R_3-R_2}{R_2R_3}+\dots+\frac{R_k-R_{k-1}}{R_{k-1}R_k}+\frac{1}{R_k}   \\= \left(1-\frac{1}{R_1}\right)+\left(\frac{1}{R_1}-\frac{1}{R_2}\right)+\left(\frac{1}{R_2}-\frac{1}{R_3}\right)+\dots+\left(\frac{1}{R_{k-1}}-\frac{1}{R_k}\right)+\frac{1}{R_k}=1.
\end{multline*}

}
Inserting all these formulas in the integral \eqref{Lk} we obtain \eqref{Hk}.
\end{proof}

\section{Asymptotics of exponential functionals of random walks}\label{Se:ExpFun}

For the convenience of the reader we paraphrase the theorems that we need, preserving most of the notation  from the original papers \cite{Hirano-1997,Xu-Wei:2023}.

\subsection{Katsuhiro Hirano \cite{Hirano-1997}} %
This reference uses notation $S_n=\xi_1+\dots+\xi_n$, where $\{\xi_j\}$ are i.i.d. and satisfy conditions:  (a) $\EE[\xi]=0$; (b)  $0<\EE[\xi^2]<\infty$; (c) $\EE [e^{\theta \xi}]$ converges for some $\theta>0$.
\begin{theorem}[\cite{Hirano-1997}, Theorem 1]\label{Thm:Hirano-T1}
Let functions $f$, $h$, and $W$ satisfy the following conditions:

\begin{enumerate}
\item[(i)] $f$ and $h$ are continuous on $[0,\infty)$.

\item[(ii)] $W$ is continuous on $\mathbb{R}$ and non-negative.

\item[(iii)] There exist positive numbers $a$, $\beta$, $\varepsilon$, $\eta$ such that
\begin{equation*}\label{Hiii-fh}
  \sup_{x>0} x^{-a}|f(x)| < \infty,
\qquad
\sup_{x>0} x^{\beta}|h(x)| < \infty,
\end{equation*}

and
\begin{equation*}
    \label{Hiii-W}
    \limsup_{x\to\infty} W(x)e^{-\varepsilon x} < \infty,
\qquad
\liminf_{x\to\infty} W(x)e^{-\eta x} > 0,
\end{equation*}
with
\begin{equation}\label{Hiii-beta}
    \beta\eta > a > 0.
\end{equation}
\end{enumerate}

If conditions (a), (b), and (c) are satisfied, then
\begin{equation}
    \label{Hirano-concl}
    \EE\!\left[
f\!\left(e^{S_n}\right)
\,h\!\left(\sum_{i=1}^{n} W(S_i)\right)
\right]
\sim c\,n^{-3/2},
\qquad n\to\infty,
\end{equation}
for some constant $c$.

\end{theorem}
\subsection{Wei Xu \cite{Xu-Wei:2023}.} %
This reference uses notation:
$$S_n=\sum_{j=1}^nX_j, \quad  I_n = \sum_{j=1}^n e^{-S_j},$$
with centered i.i.d. $X_1,X_2,\dots$.   In the paper, $F$ is a positive   bounded function on $(0,\infty)$ that vanishes at infinity.
The assumptions are:
\begin{assumption}[{\cite[Assumption 2.1]{Xu-Wei:2023}}]\label{A2.1}
The set
$
\mathcal{D}_F
=\bigl\{\nu>0:\sup_{x>0} x^\nu F(x)<\infty\bigr\}
$
is nonempty, with   $\theta_F:=\sup \mathcal{D}_F\in(0,\infty]$.
\end{assumption}
 \begin{assumption}[{\cite[Condition 2.2]{Xu-Wei:2023}}]\label{C2.2}
  For any  \(0<\delta\),  there exists a constant $C_\delta$ such that for $x,y\geq \delta$ we have
$
|F(x)-F(y)|\leq C_\delta|x-y|$.
\end{assumption}
  \begin{assumption}{\cite[Assumption 2.5]{Xu-Wei:2023}}\label{A2.5}
     Let $\mathcal{L}_X(\lambda)
=\EE[e^{\lambda X}]$. The infimum
of $\mathcal{L}_X(\lambda)$ over $[0,\theta_F]$
is attained at some  $\Lambda\in[0,\theta_F]$, i.e., $\mathcal{L}_X(\Lambda)= \inf_{\lambda\in[0,\theta_F]} \mathcal{L}_X(\lambda)$.

 \end{assumption}

  \begin{assumption}[Spitzer's condition, {\cite[Condition 2.6]{Xu-Wei:2023}}]\label{C2.6}
   There exists a constant $\rho\in(0,1)$ such that, as $n\to\infty$,
\[
\frac{1}{n}\sum_{k=1}^{n}\mathbb{P}(S_k>0)\longrightarrow \rho.
\]
(All symmetric random walks
satisfy Spitzer’s condition with $\rho=1/2$.)
  \end{assumption}

The following result deals with the oscillating case, where $\lim\sup_{n\to\infty}S_n=-\liminf_{n\to\infty}S_n=\infty$.
\begin{theorem}[{\cite[Theorem 2.7]{Xu-Wei:2023}}]\label{Thm:Xu-Wei-T1}
In the oscillating case, if $\theta_F>\Lambda=0$ and  Assumptions (\ref{A2.1}-\ref{C2.6}) are satisfied, then, as
$n\to\infty$, we have
\begin{equation}\label{eq:2.3}
\EE[F(I_n)]
\sim
C_{F}\,n^{\rho-1}\ell_1(n),
\end{equation}
where the  slowly varying function $\ell_1$ and limiting constant $C_{F}\in(0,\infty)$ are given by  explicit expressions  \cite[(2.2) and (2.4)]{Xu-Wei:2023}, which are omitted here.
\end{theorem}
\begin{remark}\label{Rem:Xuwei}
In this paper   we are interested in symmetric   i.i.d.  random variables \(\{X_j\}\) that have a density. Then  asymptotics \eqref{eq:2.3} holds with \(\rho=\frac12\), and \cite[formula (2.2)]{Xu-Wei:2023} gives \(\ell_1(n)=\frac{1}{\sqrt{\pi}}\). Oscillations of symmetric random walks were studied in \cite{Imhof1976}.
\end{remark}

\subsection*{Acknowledgments}
We thank Guillaume Barraquand for sharing an early version of \cite{Barraquand-2024-integral}.  We thank Alexey Kuznetsov for bringing to our attention Hartogs' theorem, an exercise in \cite{Conway78},
and most importantly, for providing  the analytic continuation argument  that we used in the proof of Proposition \ref{Prop2}. We thank Yizao Wang and Wei Xu for helpful discussions.

This research was partially sponsored by the National Science Center Poland [project no.  2023/51/B/ST1/01535].

\arxiv{
\subsection*{Use of AI tools declaration} During the preparation of this work, the authors used ChatGPT-4 and Mathematica 14.2 to assist in checking calculations and improving the presentation of the manuscript. In particular, after the analytic proof of Theorem \ref{Thm:PhD} had been completed, discussions with ChatGPT helped clarify several technical points concerning the $\mathrm{Beta}_{II}$   framework underlying the  probabilistic proof.
}

\end{document}